\documentclass[reqno]{amsart}
\UseRawInputEncoding
\pdfoutput=1
\usepackage{amsmath, amsthm, amssymb, color}
\usepackage{graphicx}
\usepackage[mathscr]{euscript}
\usepackage{latexsym}
\usepackage{hyperref}
\usepackage{xcolor}
\usepackage{setspace}
\usepackage{tikz}
\usetikzlibrary{arrows}

\numberwithin{equation}{section}
\newtheorem{Theorem}{Theorem}[section]
\newtheorem{Lemma}[Theorem]{Lemma}

\theoremstyle{definition}

\newtheorem{remark}[Theorem]{Remark}
\newtheorem{Proposition}[Theorem]{Proposition}

\newcounter{RomanNumber}

\newcommand{\dis}{\displaystyle}

\def\f{\frac}
\def\fc{\mathfrak{c}}
\def\dis{\displaystyle}

\allowdisplaybreaks[4]

\author[Y. Wang]{Yong Wang}
\address[Y. Wang]{Academy of Mathematics and Systems Science, Chinese Academy of Sciences, Beijing 100190, China; School of Mathematical Sciences, University of Chinese Academy of Sciences, Beijing 100049, China}
\email{yongwang@amss.ac.cn}

\author[C.G Xiao]{Changguo Xiao}
\address[C.G Xiao]{School of Mathematics and Statistics, Guangxi Normal University, Guilin, Guangxi 541004, China}
\email{changguoxiao@mailbox.gxnu.edu.cn}

\thanks{* Corresponding author: changguoxiao@mailbox.gxnu.edu.cn}
\bigbreak

\date{}

\begin{document}
	\title[Hydrodynamic limit and Newtonian limit from the rBE to the cE equations]{Hydrodynamic limit and Newtonian limit from the relativistic Boltzmann equation to the classical Euler equations}

  \begin{abstract}
  	The hydrodynamic limit and Newtonian limit are important in the relativistic kinetic theory. We justify rigorously the validity of the two independent limits from the special relativistic Boltzmann equation to the classical Euler equations without assuming any dependence between the Knudsen number $\varepsilon$ and the light speed $\mathfrak{c}$. The convergence rates are also obtained. This is achieved by Hilbert expansion of relativistic Boltzmann equation. New difficulties arise when tacking the uniform in $\mathfrak{c}$ and $\varepsilon$ estimates for the Hilbert expansion, which have been overcome by establishing some uniform-in-$\mathfrak{c}$ estimate for relativistic Boltzmann operators. 
  \end{abstract}

    \maketitle
    \setcounter{tocdepth}{1}
    \tableofcontents
  \section{Introduction}

\subsection{The relativistic Boltzmann equation}
  We consider the special relativistic Boltzmann equation
  \begin{equation}\label{1.1-0}
  	p^{\mu}\partial_{\mu}F=\frac{1}{\varepsilon}\mathcal{C}(F, F),
  \end{equation}
  which describes the dynamics of single-species relativistic particles. The dimensionless parameter $\varepsilon$ is the Knudsen number, which is proportional to the mean free path. The unknown $F(t,x,p)\ge 0$ is a distribution function for relativistic particles with position $x=(x_1,x_2,x_3)\in \Omega$ and particle momentum $p=(p^1,p^2,p^3)\in \mathbb R^3$ at time $t>0$. The collision term $\mathcal{C}(h_1, h_2)$ is defined by
  \begin{equation*}
  	\mathcal{C}(h_1, h_2)=\frac{1}{2} \int_{\mathbb{R}^{3}} \frac{d q}{q^{0}} \int_{\mathbb{R}^{3}} \frac{d p^{\prime}}{p^{\prime 0}} \int_{\mathbb{R}^{3}} \frac{d q^{\prime}}{q'^{0}} W\left(p, q \mid p^{\prime}, q^{\prime}\right)\left[h_1\left(p^{\prime}\right) h_2\left(q^{\prime}\right)-h_1(p) h_2(q)\right],
  \end{equation*}
  where the transition rate $W\left(p, q \mid p^{\prime}, q^{\prime}\right)$ has the form
  \begin{align}\label{1.3-0}
  	W\left(p, q \mid p^{\prime}, q^{\prime}\right)=s \varsigma(g, \vartheta)\delta(p^0+q^0-p'^{0}-q'^{0}) \delta^{(3)}(p+q-p'-q').
  \end{align}
  The streaming term of the relativistic Boltzmann equation \eqref{1.1-0} is given by
  \begin{align*}
  	p^\mu\partial_\mu=\frac{p^0}{\mathfrak{c}}\partial_t+p\cdot \nabla_x,
  \end{align*}
  where $\mathfrak{c}$ denotes the speed of light and $p^0$ denotes the energy of a relativistic particle with
  $$p^{0}=\sqrt{m_0^2\mathfrak{c}^2+|p|^{2}}.$$
  Here $m_0$ denotes the rest mass of particle. Now we can rewrite \eqref{1.1-0} as
  \begin{align}\label{rb}
  	\partial_t F+\hat{p}\cdot\nabla_x F=\frac{1}{\varepsilon}Q(F,F),
  \end{align}
  where $\hat{p}$ denotes the normalized particle velocity
  \begin{align*}
  	\hat{p}:=\mathfrak{c}\frac{p}{p^0}=\frac{\mathfrak{c}p}{\sqrt{m_0^2\mathfrak{c}^2+|p|^2}}.
  \end{align*}
  The collision term $Q(h_1,h_2)$ in \eqref{rb} has the form
  \begin{equation*}
  	Q(h_1, h_2)=\frac{\mathfrak{c}}{2} \frac{1}{p^{0}}\int_{\mathbb{R}^{3}} \frac{d q}{q^{0}} \int_{\mathbb{R}^{3}} \frac{d p^{\prime}}{p^{\prime 0}} \int_{\mathbb{R}^{3}} \frac{d q^{\prime}}{q'^{0}} W\left(p, q \mid p^{\prime}, q^{\prime}\right)\left[h_1\left(p^{\prime}\right) h_2\left(q^{\prime}\right)-h_1(p) h_2(q)\right].
  \end{equation*}
  
  We denote the energy-momentum 4-vector as $p^{\mu}=(p^0,p^1,p^2,p^3)$. The energy-momentum 4-vector with the lower index is written as a product in the Minkowski metric $p_{\mu}=g_{\mu \nu} p^{\nu}$, where the Minkowski metric is given by $g_{\mu \nu}=\operatorname{diag}(-1,1,1,1)$. The inner product of energy-momentum 4-vectors $p^{\mu}$ and $q_{\mu}$ is defined via the Minkowski metric
  \begin{align*}
  	p^{\mu} q_{\mu}=p^{\mu} g_{\mu \nu} q^{\nu}=-p^{0} q^{0}+\sum_{i=1}^{3} p^{i} q^{i}.
  \end{align*}
  Then it is clear that
  $$p^{\mu} p_{\mu}=-m_0^2\mathfrak{c}^2.$$
  We note that the inner product of energy-momentum 4-vectors is Lorentz invariant. 
  
  The quantity $s$ is the square of the energy in the \textit{center of momentum system}, $p+q=0$, and is given as
  \begin{align*}
  	s=s(p, q)=-\left(p^{\mu}+q^{\mu}\right)\left(p_{\mu}+q_{\mu}\right)=2\left(p^{0} q^{0}-p \cdot q+m_0^{2} \mathfrak{c}^{2}\right) \geq 4 m_0^{2} \mathfrak{c}^{2}.
  \end{align*}
  And the relative momentum $g$ in \eqref{1.3} is denoted as
  \begin{align*}
  	g=g(p, q)=\sqrt{\left(p^{\mu}-q^{\mu}\right)\left(p_{\mu}-q_{\mu}\right)}=\sqrt{2\left(p^{0} q^{0}-p \cdot q-m_0^{2} \mathfrak{c}^{2}\right)} \geq 0.
  \end{align*}
  It is direct to know that
  \begin{align*}
  	s=g^2+4m_0^2\mathfrak{c}^2.
  \end{align*}
  The post-collision momentum pair $(p'^{\mu},q'^{\mu})$ and the pre-collision momentum pair $(p^{\mu},q^{\mu})$ satisfy the relation
  \begin{align}\label{1.8}
  	p^{\mu}+q^{\mu}=p^{\prime \mu}+q^{\prime \mu}.
  \end{align}
  One may also write \eqref{1.8} as
  \begin{align}
  	p^0+q^0&=p^{\prime 0}+q^{\prime 0},\label{1.9}\\
  	p+q&=p^{\prime}+q^{\prime }\label{1.10},
  \end{align}
  where \eqref{1.9} represents the principle of conservation of energy and \eqref{1.10} represents the conservation of momentum after a binary collision.
  
  Using Lorentz transformations in \cite{Groot, Strain}, in the \textit{center of momentum system}, $Q(F,F)$ can be written as
  \begin{align}\label{1.11}
  	Q(F,F)=&\int_{\mathbb R^3}\int_{\mathbb S^2}v_\phi \varsigma(g,\vartheta)\Big[F(p')F(q')-F(p)F(q)\Big]d\omega dq\nonumber\\
  	:=&Q^{+}(F,F)-Q^{-}(F,F),
  \end{align}
  where $v_\phi=v_\phi(p,q)$ is the M{\o}ller velocity
  \begin{align*}
  	v_\phi(p,q):=\frac{\mathfrak{c}}{2}\sqrt{\left|\frac{p}{p^0}-\frac{q}{q^0} \right|^2-\left|\frac{p}{p^0}\times \frac{q}{q^0}\right|^2}=\frac{\mathfrak{c}}{4}\frac{g\sqrt{s}}{p^0q^0}.
  \end{align*}
  The pre-post collisional momentum in \eqref{1.11} satisfies
  \begin{equation*}
  	\left\{
  	\begin{aligned}
  		&p'=\frac12(p+q)+\frac12 g\Big(\omega+(\gamma_0-1)(p+q)\frac{(p+q)\cdot \omega}{|p+q|^2}\Big),\\
  		&q'=\frac12(p+q)-\frac12 g\Big(\omega+(\gamma_0-1)(p+q)\frac{(p+q)\cdot \omega}{|p+q|^2}\Big),
  	\end{aligned}
  	\right.
  \end{equation*}
  where $\gamma_0:=(p^0+q^0)/\sqrt{s}$. The pre-post collisional energy is given by
  \begin{equation*}
  	\left\{
  	\begin{aligned}
  		&p'^0=\frac12(p^0+q^0)+\frac{1}{2}\frac{g}{\sqrt{s}}(p+q)\cdot \omega,\\
  		&q'^0=\frac12(p^0+q^0)-\frac{1}{2}\frac{g}{\sqrt{s}}(p+q)\cdot \omega.
  	\end{aligned}
  	\right.
  \end{equation*}
  The scattering angle $\vartheta$ is defined by
  \begin{align*}
  	\cos \vartheta:=\frac{(p^\mu-q^\mu)(p'_\mu-q'_\mu)}{g^2}.
  \end{align*}
  The angle is well defined under \eqref{1.8} and we refer to \cite[Lemma 3.15.3]{Glassey}.
  
  The function $\varsigma(g,\vartheta)$ in \eqref{1.3-0} is called the differential cross-section or scattering kernel. The relativistic differential cross section $\varsigma(g, \vartheta)$ measures the interactions between relativistic particles. Throughout the present paper, we consider the ``hard ball" particles 
  \begin{align*}
  	\varsigma(g, \vartheta) = \text{constant}.
  \end{align*}
  Without loss of generality, we take $\varsigma(g, \vartheta) = 1$ for simplicity. The Newtonian limit in this situation, as $\mathfrak{c}\rightarrow \infty$,  is the Newtonian hard-sphere Boltzmann collision operator \cite{Strain2}.

 \subsection{Hilbert expansion}
 In the present paper, we are concerned with both the hydrodynamic limit and Newtonian limit from the relativistic Boltzmann equation to the classical Euler equations.
 To achieve this, we perform a Hilbert expansion for the relativistic Boltzmann equation \eqref{rb} with small Knudsen number $\varepsilon$. To emphasize the dependence on  $\varepsilon$ and $\mathfrak{c}$ for relativistic Boltzmann solutions, we denote the solutions of \eqref{rb} as $F^{\varepsilon,\mathfrak{c}}$ and decompose $F^{\varepsilon,\mathfrak{c}}$ as the sum 
 \begin{align}\label{1.16-0}
 	F^{\varepsilon,\mathfrak{c}}=\sum_{n=0}^{2k-1}\varepsilon^n  F^{\mathfrak{c}}_n+\varepsilon^k F^{\varepsilon,\mathfrak{c}}_R, \quad k\ge 3, 
 \end{align}
 where $F^{\mathfrak{c}}_0, F^{\mathfrak{c}}_1, \ldots, F^{\mathfrak{c}}_{2k-1}$ in \eqref{1.16-0} will depend upon $\mathfrak{c}$ but be independent of $\varepsilon$. Also, $F^{\varepsilon,\mathfrak{c}}_R$ is called the remainder term which will depend upon $\varepsilon$ and $\mathfrak{c}$. For $\mathfrak{c}=1$, Speck-Strain\cite{Speck} have already established the Hilbert expansion for the relativistic Boltzmann equation. Since we shall consider both the hydrodynamic limit $\varepsilon\to 0$ and Newtonian limit $\mathfrak{c}\to \infty$ of the relativistic Boltzmann equation, it is crucial to derive the uniform-in-$\mathfrak{c}$ estimates on $F^{\mathfrak{c}}_n\ (n=0,1,\cdots,2k-1)$ and uniform in $\mathfrak{c}$ and $\varepsilon$ estimates on $F^{\varepsilon,\mathfrak{c}}_R$.
 
 To determine the coefficients $F^{\mathfrak{c}}_0(t,x,p), \cdots$, $F^{\mathfrak{c}}_{2k-1}(t,x,p)$, we begin by plugging the expansion \eqref{1.16-0} into \eqref{rb} to obtain
 \begin{align}\label{1.17-0}
 	&\partial_t\Big(\sum_{n=0}^{2k-1}\varepsilon^n F^{\mathfrak{c}}_n+\varepsilon^k F^{\varepsilon,\mathfrak{c}}_R\Big)+\hat{p}\cdot \nabla_x\Big(\sum_{n=0}^{2k-1}\varepsilon^n F^{\mathfrak{c}}_n+\varepsilon^k F^{\varepsilon,\mathfrak{c}}_R\Big)\nonumber\\
 	&=\frac{1}{\varepsilon}Q_{\mathfrak{c}}\Big(\sum_{n=0}^{2k-1}\varepsilon^n F^{\mathfrak{c}}_n+\varepsilon^k F^{\varepsilon,\mathfrak{c}}_R,\sum_{n=0}^{2k-1}\varepsilon^n F^{\mathfrak{c}}_n+\varepsilon^k F^{\varepsilon,\mathfrak{c}}_R\Big).
 \end{align}
 Comparing the order of $\varepsilon$ in \eqref{1.17-0}, one has
 \begin{equation}\label{1.18-0}
 	\begin{split}
 		0 & =Q_{\mathfrak{c}}\left(F^{\mathfrak{c}}_0, F^{\mathfrak{c}}_0\right), \\
 		\partial_t F^{\mathfrak{c}}_0+\hat{p} \cdot \nabla_x F^{\mathfrak{c}}_0 & =Q_{\mathfrak{c}}\left(F^{\mathfrak{c}}_0, F^{\mathfrak{c}}_1\right)+Q_{\mathfrak{c}}\left(F^{\mathfrak{c}}_1, F^{\mathfrak{c}}_0\right), \\
 		\partial_t F^{\mathfrak{c}}_1+\hat{p} \cdot \nabla_x F^{\mathfrak{c}}_1 & =Q_{\mathfrak{c}}\left(F^{\mathfrak{c}}_0, F^{\mathfrak{c}}_2\right)+Q_{\mathfrak{c}}\left(F^{\mathfrak{c}}_2, F^{\mathfrak{c}}_0\right)+Q_{\mathfrak{c}}\left(F^{\mathfrak{c}}_1, F^{\mathfrak{c}}_1\right), \\
 		\cdots &  \cdots \cdots \\
 		\partial_t F^{\mathfrak{c}}_n+\hat{p} \cdot \nabla_x F^{\mathfrak{c}}_n & =\sum_{\substack{i+j=n+1 \\
 				i,j\ge 0}} Q_{\mathfrak{c}}\left(F^{\mathfrak{c}}_i, F^{\mathfrak{c}}_j\right),\\
 		\cdots &  \cdots \cdots \\
 		\partial_t F^{\mathfrak{c}}_{2k-1}+\hat{p} \cdot \nabla_x F^{\mathfrak{c}}_{2k-1} & =\sum_{\substack{i+j=2k \\
 				i,j\ge 1}} Q_{\mathfrak{c}}\left(F^{\mathfrak{c}}_i, F^{\mathfrak{c}}_j\right).
 	\end{split}
 \end{equation}
 The remainder $F^{\varepsilon,\mathfrak{c}}_R$ satisfies the equation
 \begin{align}\label{1.19-0}
 	& \partial_t F^{\varepsilon,\mathfrak{c}}_{R}+\hat{p} \cdot \nabla_x F^{\varepsilon,\mathfrak{c}}_{R}-\frac{1}{\varepsilon}\left\{Q_{\mathfrak{c}}\left(F^{\mathfrak{c}}_0, F^{\varepsilon,\mathfrak{c}}_{R}\right)+Q_{\mathfrak{c}}\left(F^{\varepsilon,\mathfrak{c}}_{R}, F^{\mathfrak{c}}_0\right)\right\}\nonumber\\
 	&=\varepsilon^{k-1} Q_{\mathfrak{c}}\left(F^{\varepsilon,\mathfrak{c}}_{R}, F^{\varepsilon,\mathfrak{c}}_{R}\right)+\sum_{i=1}^{2k-1} \varepsilon^{i-1}\left\{Q_{\mathfrak{c}}\left(F^{\mathfrak{c}}_i, F^{\varepsilon,\mathfrak{c}}_{R}\right)+Q_{\mathfrak{c}}\left(F^{\varepsilon,\mathfrak{c}}_{R}, F^{\mathfrak{c}}_i\right)\right\}+\varepsilon^k A,
 \end{align}
 where
 \begin{align*}
 	A:=\sum_{\substack{i+j\ge 2k+1 \\ 2 \leq i, j \leq 2k-1}} \varepsilon^{i+j-1-2k} Q_{\mathfrak{c}}\left(F^{\mathfrak{c}}_i, F^{\mathfrak{c}}_j\right).
 \end{align*}
 From \cite[Chap.2]{Groot}, the first equation of $\eqref{1.18-0}$ implies that $F^{\mathfrak{c}}_0$ is a local Maxwellian of the form $\mathbf{M}_{\mathfrak{c}}(n_0,u,T_0;p)$, i.e., 
 \begin{align}\label{1.13-00}
 	F^{\mathfrak{c}}_0(t,x,p)=\mathbf{M}_{\mathfrak{c}}(n_0,u,T_0;p)=\frac{n_0\gamma}{4\pi \mathfrak{c}^3K_2(\gamma)}\exp\Big\{\frac{u^{\mu}p_{\mu}}{T_0}\Big\},
 \end{align} 
  where $\gamma$ a dimensionless variable defined as
  \begin{align*}
  	\gamma=\frac{m_0\mathfrak{c}^2}{k_B T_0}
  \end{align*}
 and $T_0(t,x)>0$ represents the temperature, $n_0(t,x)> 0$ is the proper number density, $(u^0,u)$ is the four-velocity. $K_{j}(\gamma)\ (j=0,1,2,\cdots)$ are the modified second order Bessel functions defined in \eqref{2.1-01}.
  
\subsection{The relativistic Euler equations and classical Euler equations}
Similar to \cite{Cercignani1,Guo}, for $\alpha$, $\beta\in \{0,1,2,3\}$, we define the first momentum as 
\begin{align*}
	I^{\alpha}[\mathbf{M}_{\mathfrak{c}}]:=\int_{\mathbb{R}^3}\frac{p^{\alpha}}{p^0}\mathbf{M}_{\mathfrak{c}}dp
\end{align*}
and the second momentum as
\begin{align*}
	T^{\alpha \beta}[\mathbf{M}_{\mathfrak{c}}]:=\int_{\mathbb{R}^3}\frac{p^{\alpha}p^{\beta}}{p^0}\mathbf{M}_{\mathfrak{c}}dp.
\end{align*}
It has been shown in \cite[Proposition 3.3]{Speck} that 
\begin{align}
	I^{\alpha}[\mathbf{M}_{\mathfrak{c}}]&=\frac{n_0u^{\alpha}}{\mathfrak{c}},\label{1.16-01}\\
	T^{\alpha \beta}[\mathbf{M}_{\mathfrak{c}}]&=\frac{e_0+P_0}{\mathfrak{c}^3}u^{\alpha}u^{\beta}+\frac{P_0g^{\alpha\beta}}{\mathfrak{c}},\label{1.17-01}
\end{align}
where $e_0(t, x)> 0$ is the proper energy density and $P_0(t, x)> 0$ is the pressure. 

Projecting the second equation in \eqref{1.18-0} onto $1$, $p$, $p^0$, which are five collision invariants for the relativistic Boltzmann collision operator $Q_{\mathfrak{c}}(\cdot,\cdot)$, and using \eqref{1.16-01}-\eqref{1.17-01}, one obtains that $(n_0,u,T_0)$ satisfies the relativistic Euler equations.
\begin{equation}\label{re-eu}
	\left\{   \begin{aligned}
		&\frac{1}{\mathfrak{c}} \partial_t\left(n_0 u^0\right)+\nabla_x \cdot\left(n_0 u\right)=0, \\
		&\frac{1}{\mathfrak{c}} \partial_t\left[\left(e_0+P_0\right) u^0 u\right]+\nabla_x \cdot\left[\left(e_0+P_0\right) u \otimes u\right] 
		+\mathfrak{c}^2 \nabla_x P_0=0, \\
		&\frac{1}{\mathfrak{c}} \partial_t\left[\left(e_0+P_0\right)\left(u^0\right)^2-\mathfrak{c}^2 P_0\right]+\nabla_x \cdot\left[\left(e_0+P_0\right) u^0 u\right]=0.
	\end{aligned}   \right.
\end{equation}
The fluid variables $n_0$, $T_0$, $S$, $P_0$, $e_0$ in \eqref{re-eu} satisfy the following relations
\begin{align}
	 P_0&=k_B n_0 T_0=m_0 \mathfrak{c}^2 \frac{n_0}{\gamma}, \label{3.32-0}\\
	 e_0&=m_0 \mathfrak{c}^2 n_0 \frac{K_1(\gamma)}{K_2(\gamma)}+3 P_0=m_0 \mathfrak{c}^2 n_0 \frac{K_3(\gamma)}{K_2(\gamma)}-P_0, \label{3.33-0}\\
	 n_0&=4 \pi e^4 m_0^3 \mathfrak{c}^3 \exp \left(\frac{-S}{k_B}\right) \frac{K_2(\gamma)}{\gamma} \exp \left(\gamma \frac{K_1(\gamma)}{K_2(\gamma)}\right)\label{3.34-0},
\end{align}
where $k_B>0$ is Boltzmann's constant, $S(t,x)>0$ is the entropy per particle which is defined by \eqref{3.34-0} and $K_i(\gamma)$ is the modified Bessel function of the second kind defined later.

Denote
\begin{equation*} 
	V:=\begin{pmatrix}
		P_0 \\
		u \\
		S
	\end{pmatrix}.
\end{equation*}
We assume that \eqref{re-eu} is supplemented with initial data \begin{align}\label{1.19-20}
	V|_{t=0}=V_0.
\end{align}
The existence of local smooth solutions of the relativistic Euler equations \eqref{re-eu} with initial condition \eqref{1.19-20} can be established by standard hyperbolic symmetrized method and it holds that
\begin{align}\label{1.23-00}
	\|V-\overline{V}\|_{H^{N_0}}\lesssim 1,
\end{align}
where $\overline{V}:=(\overline{P},0,\overline{S})$ is a constant background with $\overline{P}>0$ and $\overline{S}>0$. We point out that the estimate in \eqref{1.23-00} is uniform-in-$\mathfrak{c}$, which is important for us, see Lemma \ref{thm-re} for details.

From \cite{Calvo,Speck}, we have the following two properties :\\
\noindent \textit{Property 1}: The map $\mathbf{\Phi}:(n_0, T_0) \mapsto (P_0, S)$  is an auto-diffeomorphism of the region $(0, \infty) \times(0, \infty)$, where the map is defined by \eqref{3.32-0}--\eqref{3.34-0}.\\
\textit{Property 2}: Under the equations of state \eqref{3.32-0}--\eqref{3.34-0}, there hold
\begin{enumerate}
	\item There exists a smooth function $\mathcal{H}$ such that $P_0$ can be expressed in terms of $e_0$ and $S$ as $P_0=\mathcal{H}(e_0, S)$.
	\item The relativistic Euler equations is hyperbolic.
	\item The relativistic Euler equations is causal (the speed of sound $a:=\mathfrak{c} \sqrt{\left.\frac{\partial P_0}{\partial e_0}\right|_S}$ is real and less than the speed of light). Actually, it holds that $0<a<\frac{\mathfrak{c}}{\sqrt{3}}$.
\end{enumerate}
From \cite[Proposition 3.4]{Speck}, the following Gibbs relation (see \cite{Choquet}) holds
\begin{align*}
	T_0dS=d\Big(\frac{e_0}{n_0}\Big)+P_0d\Big(\frac{1}{n_0}\Big),
\end{align*}
which is equivalent to 
\begin{align*}\label{1.14-00}
	\frac{\partial e_0}{\partial n_0}\Big|_{S}=\frac{e_0+P_0}{n_0},\quad \frac{\partial e_0}{\partial S}\Big|_{n_0}=n_0T_0.
\end{align*}
For simplicity of presentation, in the rest of this paper, we always assume that 
\begin{align*}
	k_B=1,\quad m_0=1.
\end{align*}


Formally, when $\mathfrak{c}$ tends to infinity, the relativistic Euler equations \eqref{re-eu} reduces to
	\begin{equation}\label{ce}
	\left\{   \begin{aligned}
		&\partial_t \rho+\nabla_x \cdot(\rho \mathfrak{u})=0, \\
		&\partial_t(\rho \mathfrak{u})+\nabla_x \cdot(\rho \mathfrak{u} \otimes \mathfrak{u})+\nabla_x \mathcal{P}=0, \\
		&\partial_t\Big(\rho\Big(\frac{1}{2}|\mathfrak{u}|^2+\mathcal{E}\Big)\Big)+\nabla_x \cdot\Big(\Big(\rho\Big(\frac{1}{2}|\mathfrak{u}|^2+\mathcal{E}\Big)+\mathcal{P}\Big) \mathfrak{u}\Big)=0,
	\end{aligned}   \right.
\end{equation}
which is the classical compressible Euler equations. Here $\rho(t,x)>0$ denotes the density of the fluid, $\mathfrak{u}(t,x)$ is velocity, $\mathcal{P}(t,x)$ is pressure, $\mathcal{E}(t,x)>0$ is internal energy per unit mass. The fluid variables $\rho$, $\theta$, $\eta$, $\mathcal{P}$ and $\mathcal{E}$ satisfy the following relations
\begin{align}\label{1.16-00}
	\mathcal{P}=\rho \theta,\quad \eta=-\ln(A_0\rho \theta^{-\frac{3}{2}}), \quad \mathcal{E}=\frac{3}{2}\theta,
\end{align}
where $A_0=(2\pi)^{-\frac{3}{2}}e^{-\frac{5}{2}}$ and $\theta(t,x)$ is the temperature of the fluid, $\eta(t,x)$ is the physical entropy. It is clear that
\begin{align*}
	\theta d\eta=d \mathcal{E}-\frac{\mathcal{P}}{\rho^2}d\rho.
\end{align*}

Denote
\begin{equation*} 
	W:=\begin{pmatrix}
		\mathcal{P} \\
		\mathfrak{u} \\
		\eta
	\end{pmatrix},
\end{equation*}
then the classical Euler equations \eqref{ce} can be written as a symmetric hyperbolic system 
\begin{align}\label{3.24-00}
	\mathbf{D}_0 \partial_t W+\sum_{j=1}^3 \mathbf{D}_j \partial_j W=0,
\end{align}
where 
\begin{align*}
	\mathbf{D}_0=\begin{pmatrix}
		1 & 0 & 0\\
		0& \sigma^2\rho^2\mathbf{I} & 0\\
		0 & 0 & 1 
	\end{pmatrix},\quad 
    \mathbf{D}_j=\begin{pmatrix}
    	\mathfrak{u}_j & \sigma^2\rho \mathbf{e}^t_j & 0\\
    	\sigma^2\rho \mathbf{e}_j & \sigma^2\rho^2 \mathfrak{u}_j\mathbf{I} & 0\\
    	0 & 0 & \mathfrak{u}_j
    \end{pmatrix}.
\end{align*}
The quantity $\sigma=\sqrt{\frac{\partial \mathcal{P}}{\partial \rho}\Big|_\eta}>0$ is the sound speed of the classical Euler equations. 

For simplicity, we supplement \eqref{3.24-00} with the same initial data and constant background as in the relativistic Euler case, that is, 
\begin{align*}
	W|_{t=0}=V_0,\quad \overline{W}=\overline{V},
\end{align*}
where $\overline{W}:=(\overline{\mathcal{P}},0,\overline{\eta})$. It is a classical result from the theory of symmetric hyperbolic systems that \eqref{3.24-00} admits a local smooth solution with smooth initial data, see Lemma \ref{thm-ce} for details.

\subsection{A brief history of the hydrodynamic and Newtonian limits for (relativistic) Boltzmann equation }
    For the hydrodynamic limit of the non-relativistic Boltzmann equation, there have been extensive studies on this subject and we only mention a few works. In the founding work of Maxwell \cite{Maxwell} and Boltzmann \cite{Boltzmann}, it is shown that the
    Boltzmann equation is closely related to the fluid dynamical systems for both compressible
    and incompressible flows.
    Hilbert and Enskog-Chapman independently developed a set of formal small-parameter expansion methods, called Hilbert expansion and Enskog-Chapman expansion respectively, and established the connection between Boltzmann equation and compressible (incompressible) Euler equations, compressible (incompressible) Navier-Stokes (Fourier) systems and the acoustic system, etc. It is a important and challenging problem to rigorously
    justify these formal approximations. In fact, the purpose of Hilbert's sixth problem
    \cite{Hilbert} is to establish the laws of motion of continua from more microscopic physical
    models, such as Boltzmann theory, from a rigorous mathematical standpoint. For the hydrodynamic limit from Boltzmann equation to compressible Euler equations, Caflisch \cite{Caflisch} strictly justified the validity of the limit by employing the truncated Hilbert expansion method, see also \cite{Lachowicz,Nishida,Ukai}, and \cite{Guo2,Guo4} for an application of $L^2-L^{\infty}$ approach. For the hydrodynamic limit to incompressible Navier-Stokes system, see \cite{Bardos1,Bardos2,Bardos4,Caprino,Esposito,Esposito1,Esposito2,Golse1,Guo3,Guo7,Jang,Jiang1,Masmoudi,Saint-Raymond,Wu} and references cited therein. For the compressible Euler limit and acoustic limit of Boltzmann equation with specular reflection boundary conditions, we refer the reader to the recent work of Guo-Huang-Wang \cite{Guo5}. 
    For other works which connect to hydrodynamic limit, we refer to \cite{Bardos3,Liu,Golse, Guo6} and  the review articles \cite{Golse2, Masmoudi1,Villani}.  
    
    Although there have been satisfactory results on the hydrodynamic limit of the non-relativistic Boltzmann equation, much less is known on the hydrodynamic limit or/and Newtonian limit of the relativistic Boltzmann equation despite of its importance. For the Newtonian limit of the relativistic particles, Calogero \cite{Calogero} established the existence of local-in-time relativistic Boltzmann solutions in periodic box, and then proved that such solutions converge, in a suitable norm, to the Newtonian Boltzmann solutions as $\mathfrak{c}\rightarrow \infty$. Later, for the case near vacuum,  Strain \cite{Strain2} proved the unique global-in-time mild solution and justified the Newtonian limit for arbitrary time intervals $[0,T]$. For the hydrodynamic limit of the relativistic Boltzmann equation, Speck-Strain \cite{Speck} demonstrated the hydrodynamic limit from the relativistic Boltzmann equation to the relativistic Euler equations for local-in-time smooth solutions. It is shown in \cite{Guo} that solutions of the relativistic Vlasov-Maxwell-Boltzmann system converge to solutions of the relativistic Euler-Maxwell system globally in time, as the Knudsen number $\varepsilon \rightarrow 0$. 
    
    In the present paper, we are concerned with both the hydrodynamic limit $\varepsilon\to 0$ and Newtonian limit $\mathfrak{c}\to \infty$ from the relativistic Boltzmann equation to the classical Euler equations. This is achieved by employing the Hilbert expansion method and uniform in $\mathfrak{c}$ and $\varepsilon$ estimates on the Hilbert expansion. 
    

\subsection{Main results}
 We consider the perturbation around the local Maxwellian $\mathbf{M}_{\mathfrak{c}}$:
 \begin{align}\label{1.3}
 	F(t,x,p)=\mathbf{M}_{\mathfrak{c}}(t,x,p)+\sqrt{\mathbf{M}_{\mathfrak{c}}(t,x,p)}f(t,x,p),
 \end{align}
	We define the linearized collision operator $\mathbf{L}_{\mathfrak{c}}f$ and nonlinear collision operator $\Gamma_{\mathfrak{c}}\left(f_1, f_2\right)$:
\begin{align*}
	&\mathbf{L}_{\mathfrak{c}} f  :=-\frac{1}{\sqrt{\mathbf{M}_{\mathfrak{c}}}}[Q_{\mathfrak{c}}(\sqrt{\mathbf{M}_{\mathfrak{c}}} f, \mathbf{M}_{\mathfrak{c}})+Q_{\mathfrak{c}}(\mathbf{M}_{\mathfrak{c}}, \sqrt{\mathbf{M}_{\mathfrak{c}}} f)]=\nu_{\mathfrak{c}}  f-\mathbf{K}_{\mathfrak{c}}f, \\
	&\Gamma_{\mathfrak{c}}\left(f_1, f_2\right):=\frac{1}{\sqrt{\mathbf{M}_{\mathfrak{c}}}} Q_{\mathfrak{c}}\left(\sqrt{\mathbf{M}_{\mathfrak{c}}} f_1, \sqrt{\mathbf{M}_{\mathfrak{c}}} f_2\right),
\end{align*}
where the collision frequency $\nu_{\mathfrak{c}}=\nu_{\mathfrak{c}}(t, x, p)$ is defined as
\begin{align}\label{1.34-00}
	\nu_{\mathfrak{c}}(p):=\int_{\mathbb{R}^3}\int_{\mathbb{S}^2} v_{\phi}(p,q) \mathbf{M}_{\mathfrak{c}}(q)d \omega d q=\frac{\mathfrak{c}}{2}\frac{1}{p^0}\int_{\mathbb{R}^3}\frac{dq}{q^0}\int_{\mathbb{R}^3}\frac{dq'}{q'^0}\int_{\mathbb{R}^3}\frac{dp'}{p'^0}W(p,q\mid p',q')\mathbf{M}_{\mathfrak{c}}(q)
\end{align}
and $\mathbf{K}_{\mathfrak{c}}f$ takes the following form:
\begin{align*}
	\mathbf{K}_{\mathfrak{c}}f&:= \int_{\mathbb{R}^3}\int_{\mathbb{S}^2} v_{\phi}(p,q) \sqrt{\mathbf{M}_{\mathfrak{c}}(q)}\left[\sqrt{\mathbf{M}_{\mathfrak{c}}(q')} f\left(p^{\prime}\right)+\sqrt{\mathbf{M}_{\mathfrak{c}}(p')} f\left(q^{\prime}\right)\right]d\omega dq\\
	& \qquad-\int_{\mathbb{R}^3}\int_{\mathbb{S}^2} v_{\phi}(p,q)  \sqrt{\mathbf{M}_{\mathfrak{c}}(p)\mathbf{M}_{\mathfrak{c}}(q)} f(q)d\omega dq \\
	&=\mathbf{K}_{\mathfrak{c}2}f-\mathbf{K}_{\mathfrak{c}1}f.
\end{align*}

We introduce the global Maxwellian $J_{\mathfrak{c}}(p)$ as
\begin{align}\label{1.50-0}
	J_{\mathfrak{c}}(p):=\frac{n_M\gamma_M}{4\pi \mathfrak{c}^3K_2(\gamma_M)}\exp{\Big(\frac{-\mathfrak{c} p^0}{T_M}\Big)},
\end{align}
where $n_M$, $T_M$ are positive constants and $\gamma_M=\frac{\mathfrak{c}^2}{T_M}$. For each $\ell \geq 0$, we also define the weight function $w_{\ell}$ as
\begin{align}\label{1.46-00}
    w_{\ell}:=w_{\ell}(p)=(1+|p|^2)^{\frac{\ell}{2}}.
\end{align}
We then define a corresponding weighted $L^{\infty}$ norm by
\begin{align*}
	\|h\|_{\infty,\ell}=\|w_{\ell}h\|_{L^{\infty}}.
\end{align*}

Our first result is the Hilbert expansion of relativistic Boltzmann equation with uniform-in-$\mathfrak{c}$ estimates. Notice that $k$ and $N_0$ are defined in \eqref{1.16-0}  and Lemma \ref{thm-re}. 
 
\begin{Theorem}\label{thm1.1-0}
	Assume $k=3$, $N_0\ge 10$. Let $(n_0(t,x), u(t,x), T_0(t,x))$ be a smooth solution to the relativistic Euler equations \eqref{re-eu} with initial data $\eqref{1.19-20}$ and constant background $\overline{V}$ for $(t, x) \in[0, T] \times \mathbb{R}^3$. Suppose that $\mathbf{M}_{\mathfrak{c}}(t,x,p)$ is the local relativistic Maxwellian in  \eqref{1.13-00} and there exist constants $C>0$, $n_M>0$, $T_M>0$, and $\alpha \in(\frac{1}{2},1)$ such that 
	\begin{align}\label{1.53-0}
		\frac{J_{\mathfrak{c}}(p)}{C} \leq  \mathbf{M}_{\mathfrak{c}}(t,x,p)\leq C J_{\mathfrak{c}}^\alpha(p).
	\end{align}
	Define initially
    \begin{align*}
	F^{\varepsilon,\mathfrak{c}}(0, x,  p)=\mathbf{M}_{\mathfrak{c}}(0, x, p)+\sum_{n=1}^{2k-1} \varepsilon^n F^{\mathfrak{c}}_n(0,x, p)+\varepsilon^k F^{\varepsilon,\mathfrak{c}}_{R}(0, x, p) \geq 0
    \end{align*}
    with
    \begin{align*}
    	\varepsilon^{\frac{3}{2}}\Big\|\frac{F^{\varepsilon,\mathfrak{c}}_{R}(0)}{\sqrt{J_{\mathfrak{c}}}} \Big\|_{\infty, \ell}+\Big\|
    	\frac{F^{\varepsilon,\mathfrak{c}}_{R}(0)}{\sqrt{\mathbf{M}_{\mathfrak{c}}(0)}} \Big\|_2\le C<\infty.
    \end{align*}
	Then there exist two independent positive constants $\varepsilon_0\in (0,1]$ and $\mathfrak{c}_0\gg 1$ such that, for each $0<\varepsilon\le \varepsilon_0$ and $\mathfrak{c}\ge \mathfrak{c}_0$, there admits a unique classical solution $F^{\varepsilon,\mathfrak{c}}$ of the relativistic Boltzmann equation \eqref{re-eu} for $(t, x, p) \in [0, T] \times \mathbb{R}^3 \times \mathbb{R}^3$ in the following form of expansion
	\begin{align*}
		F^{\varepsilon,\mathfrak{c}}(t,x,p)=\mathbf{M}_{\mathfrak{c}}(t,x,p)+\sum_{n=1}^{2k-1}\varepsilon^nF^{\varepsilon}_n(t,x,p)+\varepsilon^{k} F^{\varepsilon,\mathfrak{c}}_R(t,x,p)\ge 0,
	\end{align*}
	where the functions $F^{\varepsilon}_n\ (n=1,\cdots,2k-1)$ are constructed in Proposition \ref{prop6.3}.
	
	Furthermore, there exists a constant $C_T>0$ such that for all $\varepsilon \in\left(0, \varepsilon_0\right]$ and for any $\ell \geq 9$, the following estimate holds: 
	\begin{align*}
		& \varepsilon^{\frac{3}{2}} \sup _{0 \leq t \leq  T}\left\|\frac{F^{\varepsilon,\mathfrak{c}}_{R}(t)}{\sqrt{J_{\mathfrak{c}}}} \right\|_{\infty, \ell}+\sup _{0 \leq t \leq T}\left\| \frac{F^{\varepsilon,\mathfrak{c}}_{R}(t)}{\sqrt{\mathbf{M}_{\mathfrak{c}}(t)}} \right\|_2 \nonumber\\
		& \leq C_T\left\{\varepsilon^{\frac{3}{2}}\Big\|\frac{F^{\varepsilon,\mathfrak{c}}_{R}(0)}{\sqrt{J_{\mathfrak{c}}}} \Big\|_{\infty, \ell}+\Big\|
		\frac{F^{\varepsilon,\mathfrak{c}}_{R}(0)}{\sqrt{\mathbf{M}_{\mathfrak{c}}(0)}} \Big\|_2+1\right\}.
	\end{align*}
	Moreover, we have that
	\begin{align}\label{1.53-00}
	   \sup _{0 \leq t \leq T}\left\|\frac{F^{\varepsilon,\mathfrak{c}}(t)-\mathbf{M}_{\mathfrak{c}}(t)}{\sqrt{J_{\mathfrak{c}}}}\right\|_{\infty}+\sup _{0 \leq t \leq T}\left\|\frac{F^{\varepsilon,\mathfrak{c}}(t)-\mathbf{M}_{\mathfrak{c}}(t)}{\sqrt{\mathbf{M}_{\mathfrak{c}}(t)}}\right\|_2 \leq C_T \varepsilon,
    \end{align}
	where the constant $C$ and $C_T>0$ are independent of $\varepsilon$ and $\mathfrak{c}$.
\end{Theorem}
    \begin{remark}
    	It follows from \eqref{1.53-00} that we have established the uniform-in-$\mathfrak{c}$ hydrodynamic limit from the relativistic Boltzmann equation to the relativistic Euler equations.
    \end{remark}
     
     \begin{remark}
 	   When $\frac{|u|}{\mathfrak{c}}$ is suitably small, it has been shown in \cite[Lemma 1.1]{Speck} that there exist positive constants $C>0$, $n_M>0$, $T_M>0$, and $\alpha \in(\frac{1}{2},1)$, which are independent of $\mathfrak{c}$, such that \eqref{1.53-0} holds. 
     \end{remark}
     \begin{remark}
     	The uniform-in-$\mathfrak{c}$ estimates for the relativistic Boltzmann collision operators developed here can also be applied to the Newtonian limit from the relativistic Boltzmann equation to the Newtonian Boltzmann equation. This will be considered in a forthcoming paper.
     \end{remark}
     With the uniform-in-$\mathfrak{c}$ estimates in Theorem \ref{thm1.1-0}, one can further obtain both the hydrodynamic limit $\varepsilon \rightarrow 0$ and the Newtonian limit $\mathfrak{c} \rightarrow \infty$ at the same time.

\begin{Theorem}\label{thm1.3-0}
	Assume that all conditions in Theorem \ref{thm1.1-0} are satisfied. Suppose that 
	$(\rho(t,x)$, $\mathfrak{u}(t,x), \theta(t,x))$
	is a smooth solution to the classical Euler equations \eqref{ce} with the same initial data and constant background as the relativistic Euler case. Let $\mu$ be the local Maxwellian of classical Boltzmann equation, i.e.,
	\begin{align}\label{1.34-20}
		\mu(t,x,p)=\frac{\rho}{(2\pi \theta)^{\frac{3}{2}}}e^{-\frac{|p-\mathfrak{u}|^2}{2\theta}}.
	\end{align} 
    Then there exist independent positive constants $\varepsilon_0\in (0,1]$ and $\mathfrak{c}_0\gg 1$ such that for all $0<\varepsilon\le \varepsilon_0$ and $\mathfrak{c}\ge \mathfrak{c}_0$, the following estimate holds:
	\begin{align}\label{1.54-00}
		 \sup _{0 \leq t \leq T}\left\|\big(F^{\varepsilon,\mathfrak{c}}-\mu\big)(t)e^{\delta_0|p|}\right\|_{\infty}\le C_T\varepsilon+C_T\mathfrak{c}^{-\frac{3}{2}},
	\end{align} 
    where all the positive constants $\varepsilon_0$, $\mathfrak{c}_0$, $C_T$ and $\delta_0$ are independent of $\varepsilon$ and $\mathfrak{c}$.
\end{Theorem}

\begin{remark}
	\eqref{1.54-00} indicates that we have established both the hydrodynamic limit and the Newtonian limit from the relativistic Boltzmann equation to the classical Euler equations. 
    We point out that the two limits $\varepsilon \rightarrow 0$ and $\mathfrak{c} \rightarrow \infty$ can be taken  independently at the same time without assuming any dependence between $\varepsilon$ and $\mathfrak{c}$.
\end{remark}
    \begin{remark}
    	It is worth noting that we make no efforts to obtain the best convergence rates, which is not our main focus here. Actually, for the Newtonian limit, one can obtain the convergence rate $\frac{1}{\mathfrak{c}^{2-\epsilon}}$ for any given small $\epsilon>0$.
    \end{remark}
    
    \begin{remark}
    	Due to the effect of special relativity, we can only obtain the particle velocity weight $e^{\delta_0|p|}$ in \eqref{1.54-00}.
    \end{remark}

\subsection{Main difficulties and strategy of the proof}
We make some comments on the main ideas of the proof and explain the main difficulties and techniques involved in the process.

It is noted that, for the relativistic Boltzmann equation \eqref{rb}, one can not transform the solution $F(t,x,p)$ to $F(t,x,\mathfrak{p})$ with a change of variables $p=\mathfrak{c} \mathfrak{p}$. Now we take the global Maxwellian $\mathbf{M}_{\mathfrak{c}}(1,0,1;p)$ as an example. In fact, $\mathbf{M}_{\mathfrak{c}}(1,0,1;p)\cong e^{\mathfrak{c}^2-\mathfrak{c}p^0}$. It is clear that
\begin{align*}
	e^{\mathfrak{c}^2-\mathfrak{c}p^0}=e^{-\frac{\mathfrak{c}^2|\mathfrak{p}|^2}{1+\sqrt{1+|\mathfrak{p}|^2}}},
\end{align*}
which is actually still a function of $\mathfrak{p}$ and $\mathfrak{c}$. On the other hand, for the normalized particle velocity $\hat{p}$, it holds that
\begin{align*}
	\hat{p}=\mathfrak{c}\frac{p}{p^0}=\frac{\mathfrak{c}p}{\sqrt{|p|^2+\mathfrak{c}^2}}=\frac{\mathfrak{c}\mathfrak{p}}{\sqrt{1+|\mathfrak{p}|^2}},
\end{align*}
which is also a function of $\mathfrak{p}$ and $\mathfrak{c}$. Hence the collision term $Q_{\mathfrak{c}}(F,F)$ can not be transformed into a new form depending only on $\mathfrak{p}$. Thus the roles of the light speed $\mathfrak{c}$ and the Knudsen number $\varepsilon$ are totally different. Therefore it is important to establish the uniform-in-$\mathfrak{c}$ estimate for the relativistic Boltzmann collision in the present paper.

To justify both the hydrodynamic limit and Newtonian limit from the relativistic Boltzmann equation to the classical Euler equations, we utilize the Hilbert expansion for the relativistic Boltzmann equation with respect to the small Knudsen number. The key point is to establish the uniform-in-$\mathfrak{c}$ estimates for the Hilbert expansion.


Firstly, we prove the existence of smooth solutions to the relativistic Euler equations with uniform-in-$\mathfrak{c}$ estimates, see Lemma \ref{thm-re}. Then, by applying the energy method of the symmetric hyperbolic systems, we establish the Newtonian limit from the relativistic Euler equations \eqref{re-eu} to the classical Euler equations \eqref{ce} with convergence rate $\mathfrak{c}^{-2}$, see section \ref{sec3.3} for details of proof.


Secondly, we aim to establish the uniform-in-$\mathfrak{c}$ bounds for the Hilbert expansion $F^{\mathfrak{c}}_n\ (n\ge 1)$ as well as the remainder $F^{\varepsilon,\mathfrak{c}}_R$. As explained above, since the collision operators must dependent on the speed of light $\mathfrak{c}$, the main difficulty lies in the uniform-in-$\mathfrak{c}$ estimates on the collision operators $Q_{\mathfrak{c}}(\cdot,\cdot)$, $\mathbf{L}_{\mathfrak{c}}$ and $\mathbf{L}_{\mathfrak{c}}^{-1}$.


For the relativistic Boltzmann equation, due to the complexity of the local relativistic Maxwellian $\mathbf{M}_{\mathfrak{c}}$, the expression of the kernel $k_{\mathfrak{c}}(p,q)$ (see \eqref{2.9-20}) of $\mathbf{K}_{\mathfrak{c}}$ is very complicated and it is not an easy job to obtain the uniform-in-$\mathfrak{c}$ estimate for $k_{\mathfrak{c}}(p,q)$. By applying the Lorentz transformation and dividing the integration region into three parts: $\{|\bar{p}-\bar{q}|\ge \mathfrak{c}^{\frac{1}{8}}\}$, $\{|\bar{p}-\bar{q}|\le \mathfrak{c}^{\frac{1}{8}},\, \&\, |p|\le \mathfrak{c}\}$ and $\{|\bar{p}-\bar{p}|\le \mathfrak{c}^{\frac{1}{8}},\, \&\, |p|\ge \mathfrak{c}\}$,
one can get
\begin{equation*} 
	\int_{\mathbb{R}^3}|k_{\mathfrak{c}}(p,q)|dq\lesssim
	\left\{
	\begin{aligned}
		&(1+|p|)^{-1}, \quad |p|\le \mathfrak{c},\\
		&\mathfrak{c}^{-1},\qquad\qquad\, |p|\ge \mathfrak{c},
	\end{aligned}
	\right.
\end{equation*}
see Lemmas \ref{lem4.3-00}--\ref{lem2.8} for details. Similarly, we can also prove
\begin{align*}
	\nu_{\mathfrak{c}}(p)\cong
	\left\{
	\begin{aligned}
		& 1+|p| , \quad |p|\le \mathfrak{c},\\
		&\mathfrak{c},\qquad\quad\,\,  |p|\ge \mathfrak{c},
	\end{aligned}
	\right.
\end{align*}
see Lemma \ref{lem2.9} for details.  

\smallskip

Let $k(p,q)$ be the kernel of the classical Boltzmann equation of hard sphere (see \eqref{4.93-20}). Observe that $k_{\mathfrak{c}}(p,q)$ and $k(p,q)$ depend on the relativistic Euler solutions and classical Euler solutions, respectively. By tedious calculations and the Newtonian limit of the relativistic Euler equations (see Proposition \ref{thm-retoce}), we can establish the following
\begin{equation}\label{1.33-3}
	\int_{\mathbb{R}^3}|k_{\mathfrak{c}}(p,q)-k(p,q)|dq\lesssim \mathfrak{c}^{-\frac38} \rightarrow 0\quad \mbox{as}\quad \mathfrak{c}\rightarrow \infty,
\end{equation}
see Lemmas \ref{lem2.13}--\ref{lem2.14} for details. Since the orthonormal basis $\{\chi^{\mathfrak{c}}_{\alpha}\}_{\alpha=0}^4$ of the null space $\mathcal{N}_{\mathfrak{c}}$ also depend on $\mathfrak{c}$, we also need to prove that 
\begin{equation}\label{1.33-4}
	\lim_{\mathfrak{c}\rightarrow \infty}\chi^{\mathfrak{c}}_{\alpha}=\chi_{\alpha},
\end{equation}
where $\{\chi_{\alpha}\}_{\alpha=0}^4$ is the corresponding orthonormal basis of the null space $\mathcal{N}$ for the classical Boltzmann equation, see Lemma \ref{lem2.10-0} for details. 
With the help of \eqref{1.33-3}--\eqref{1.33-4} and a contradiction argument, one can finally obtain the following uniform-in-$\mathfrak{c}$ coercivity estimate for $\mathbf{L}_{\mathfrak{c}}$
\begin{align*}
	\langle \mathbf{L}_{\mathfrak{c}}g, g\rangle \geq \zeta_0\|(\mathbf{I}-\mathbf{P}_{\mathfrak{c}})g\|_{\nu_{\mathfrak{c}}}^{2},\quad  g\in L^2_{\nu}(\mathbb{R}^3).
\end{align*}
Here we emphasize that $\zeta_0>0$ is a positive constant independent of $\mathfrak{c}$.
With the uniform-in-$\mathfrak{c}$ coercivity estimate for $\mathbf{L}_{\mathfrak{c}}$, one can derive the uniform-in-$\mathfrak{c}$ exponential decay for $\mathbf{L}_{\mathfrak{c}}^{-1}$ by similar arguments as in \cite{Jiang}, see section \ref{sec4.2} for details.

Utilizing the above uniform-in-$\mathfrak{c}$ estimates, we can establish the uniform bounds on the Hilbert expansions $F^{\mathfrak{c}}_n(t, x, p)\ (n\ge 1)$, see Proposition \ref{prop6.3} for details. 
Based on the estimates on $F^{\mathfrak{c}}_n(t, x, p)\ (n\ge 1)$, we use the $L^2-L^{\infty}$ framework in \cite{Guo1,Guo2,Speck} to control the remainder $F^{\varepsilon,\mathfrak{c}}_R$ uniformly in $\mathfrak{c}$ and $\varepsilon$, see Lemmas \ref{lem7.1}--\ref{lem7.2} for details. Hence, we established the Hilbert expansion of the relativistic Boltzmann equation with  uniform-in-$\mathfrak{c}$ estimates, see Theorem \ref{thm1.1-0}.
  
  Finally, by combining the Hilbert expansion in Theorem \ref{thm1.1-0} and the Newtonian limit of relativistic Euler equations in Proposition \ref{thm-retoce}, we can justify both the hydrodynamic limit and Newtonian limit of the relativistic Boltzmann equation to the classical Euler equations, see Theorem \ref{thm1.3-0} for details. 

\subsection{Organization of the paper}
In section 2, we present some results about Bessel functions and give explicit expressions for the kernel of the linearized relativistic collision operators.
Section 3 is dedicated to the existence of local-in-time solutions of the relativistic Euler equations and the Newtonian limit of the relativistic Euler equations.
In section 4, we develop a series of uniform-in-$\mathfrak{c}$ estimates to obtain the key coercivity estimate on the linearized operator $\mathbf{L}_{\mathfrak{c}}$ as well as  $\mathbf{L}_{\mathfrak{c}}^{-1}$, which allow us to establish the uniform-in-$\mathfrak{c}$ bounds on the Hilbert expansion $F^{\mathfrak{c}}_n$ in section 5.
In section 6, we use the $L^2-L^{\infty}$ method to derive  the uniform in $\mathfrak{c}$ and $\varepsilon$ estimates on the remainder $F^{\varepsilon,\mathfrak{c}}_R$ and prove the main theorems, Theorems \ref{thm1.1-0} and \ref{thm1.3-0}. 
In the appendix, we are devoted to present the orthonormal basis of the null space $\mathcal{N}_{\mathfrak{c}}$ of $\mathbf{L}_{\mathfrak{c}}$.

\subsection{Notations}
Throughout this paper, $C$ denotes a generic positive constant which is independent of $\mathfrak{c}$, $\varepsilon$ and $C_{a}, C_{b}, \ldots$ denote the generic positive constants depending on $a, b, \ldots$, respectively, which may vary from line to line. $A \lesssim B$ means that there exists a constant $C>0$, which is independent of $\mathfrak{c},\  \varepsilon$, such that $A \le C B$.  $A\cong B$ means that both $A \lesssim B$ and $B \lesssim A$ hold.  $\|\cdot\|_{2}$ denotes either the standard  $L^{2}\left( \mathbb{R}_{x}^{3}\right)$-norm or $L^{2}\left( \mathbb{R}_{p}^{3}\right)$-norm or $L^{2}\left(\mathbb{R}_x^3 \times \mathbb{R}_{p}^{3}\right)$-norm. Similarly, $\|\cdot\|_{\infty}$ denotes either the $L^{\infty}\left( \mathbb{R}_{x}^{3}\right)$-norm or $L^{\infty}\left( \mathbb{R}_{p}^{3}\right)$-norm or  $L^{\infty}\left(\mathbb{R}_x^3 \times \mathbb{R}_{p}^{3}\right)$-norm. We also introduce the weighted $L^{\infty}$ norm $\|\cdot\|_{\infty,\ell}= \|w_{\ell}\cdot\|_{\infty}$. We denote $\langle\cdot, \cdot\rangle$ as either the $L^{2}\left(\mathbb{R}_{x}^{3}\right)$ inner product or $L^{2}\left(\mathbb{R}_{p}^{3}\right)$ inner product or $L^{2}\left(\mathbb{R}_x^3 \times \mathbb{R}_{p}^{3}\right)$ inner product. Moreover, we denote $\|\cdot\|_{\nu_{\mathfrak{c}}}:=\|\sqrt{\nu_{\mathfrak{c}}} \cdot\|_{2}$.

\section{Preliminaries}

We define the modified Bessel function of the second kind   (see \cite[(3.19)]{Cercignani1})
\begin{align}\label{2.1-01}
	K_{j}(z)=\Big(\frac{z}{2}\Big)^j\frac{\Gamma(\frac{1}{2})}{\Gamma(j+\frac{1}{2})}\int_{1}^{\infty} e^{-z t}\left(t^{2}-1\right)^{j-\frac{1}{2}} dt,\quad j\ge 0,\ z>0.
\end{align}
We will frequently use the following properties for $K_j(z)$.
\begin{Lemma}\emph{(\cite{Olver, Watson})} \label{lem2.1-00}
	It holds that
	\begin{align*}
		K_{j+1}(z) & =\frac{2j}{z}K_j(z)+K_{j-1}(z), \quad j \geq 1,
	\end{align*}
	and 
	\begin{align*}
		\frac{d}{d z}\left(\frac{K_j(z)}{z^j}\right) & =-\left(\frac{K_{j+1}(z)}{z^j}\right), \quad j \geq 0.
	\end{align*}
	The asymptotic expansion for $K_j(z)$ takes the form
	\begin{align*}
		K_{j}(z) &=\sqrt{\frac{\pi}{2z}} \frac{1}{e^{z}}\left[\sum_{m=0}^{n-1}A_{j,m}z^{-m}+\gamma_{j,n}(z)z^{-n}\right],\quad j\ge 0,\ n\ge 1,
	\end{align*}
	where the following additional identities and inequalities also hold:
	\begin{align*}
		A_{j,0}&=1,\nonumber\\
		A_{j,m}&=\frac{1}{m!8^m}(4j^2-1)(4j^2-3^2)\cdots (4j^2-(2m-1)^2),\quad j\ge 0,\ m\ge 1,\nonumber\\
		|\gamma_{j,n}(z)|&\le 2|A_{j,n}|\exp{\Big(\big[j^2-\frac{1}{4}\big]z^{-1}\Big)},\quad j\ge 0,\ n\ge 1,\nonumber\\
		K_{j}(z)&<K_{j+1}(z),\quad j\ge 0.
	\end{align*} 
    Furthermore, for $j\le n+\frac{1}{2}$, one has a more exact estimate
    \begin{align*}
    	|\gamma_{j,n}(z)|&\le |A_{j,n}|.
    \end{align*}
\end{Lemma}

We next deduce the kernel of the linearized relativistic collision operator. Recall that 
\begin{align*}
	\mathbf{K}_{\mathfrak{c}}f&= \int_{\mathbb{R}^3}\int_{\mathbb{S}^2} v_{\phi}(p,q) \sqrt{\mathbf{M}_{\mathfrak{c}}(q)}\left[\sqrt{\mathbf{M}_{\mathfrak{c}}(q')} f\left(p^{\prime}\right)+\sqrt{\mathbf{M}_{\mathfrak{c}}(p')} f\left(q^{\prime}\right)\right]d\omega dq\\
	& \qquad-\int_{\mathbb{R}^3}\int_{\mathbb{S}^2} v_{\phi}(p,q)  \sqrt{\mathbf{M}_{\mathfrak{c}}(p)\mathbf{M}_{\mathfrak{c}}(q)} f(q)d\omega dq \\
	&=\frac{\mathfrak{c}}{2}\frac{1}{p^0}\int_{\mathbb{R}^3}\frac{dq}{q^0}\int_{\mathbb{R}^3}\frac{dq'}{q'^0}\int_{\mathbb{R}^3}\frac{dp'}{p'^0}W(p,q\mid p',q')\sqrt{\mathbf{M}_{\mathfrak{c}}(q)\mathbf{M}_{\mathfrak{c}}(q')}f(p')\\
	&\qquad+\frac{\mathfrak{c}}{2}\frac{1}{p^0}\int_{\mathbb{R}^3}\frac{dq}{q^0}\int_{\mathbb{R}^3}\frac{dq'}{q'^0}\int_{\mathbb{R}^3}\frac{dp'}{p'^0}W(p,q\mid p',q')\sqrt{\mathbf{M}_{\mathfrak{c}}(q)\mathbf{M}_{\mathfrak{c}}(p')}f(q')\\
	&\qquad -\frac{\mathfrak{c}}{2}\frac{1}{p^0}\int_{\mathbb{R}^3}\frac{dq}{q^0}\int_{\mathbb{R}^3}\frac{dq'}{q'^0}\int_{\mathbb{R}^3}\frac{dp'}{p'^0}W(p,q\mid p',q')\sqrt{\mathbf{M}_{\mathfrak{c}}(p)\mathbf{M}_{\mathfrak{c}}(q)}f(q)\\
	&:=\mathbf{K}_{\mathfrak{c}2}f-\mathbf{K}_{\mathfrak{c}1}f  .
\end{align*}
Then it is clear that the kernel of $\mathbf{K}_{\mathfrak{c}1}$ takes the form 
\begin{align}\label{1.34-0}
	k_{\mathfrak{c}1}(p,q)&=\int_{\mathbb{S}^2} v_{\phi}(p,q) \sqrt{\mathbf{M}_{\mathfrak{c}}(p)\mathbf{M}_{\mathfrak{c}}(q)} d\omega
	=\frac{\pi \mathfrak{c}g\sqrt{s}}{p^0q^0}\sqrt{\mathbf{M}_{\mathfrak{c}}(p)\mathbf{M}_{\mathfrak{c}}(q)}.
\end{align}

By similar arguments as in \cite{Strain}, we can deduce that each term of $\mathbf{K}_{\mathfrak{c}2}f$ is equal to
\begin{align*}
	\frac{\mathfrak{c}}{2}\frac{1}{p^0}\int_{\mathbb{R}^3}\frac{dq}{q^0}f(q)\Big\{\int_{\mathbb{R}^3}\frac{dq'}{q'^0}\int_{\mathbb{R}^3}\frac{dp'}{p'^0}\bar{s}\delta^{(4)}(p^\mu+p'^{\mu}-q^{\mu}-q'^{\mu})\sqrt{\mathbf{M}_{\mathfrak{c}}(p')\mathbf{M}_{\mathfrak{c}}(q')}\Big\},
\end{align*}
which yields that the kernel of $\mathbf{K}_{\mathfrak{c}2}$ is
\begin{align}\label{2.6-10}
	k_{\mathfrak{c}2}(p,q)=\frac{\mathfrak{c}}{p^0q^0}\int_{\mathbb{R}^3}\frac{dq'}{q'^0}\int_{\mathbb{R}^3}\frac{dp'}{p'^0}\bar{s}\delta^{(4)}(p^\mu+p'^{\mu}-q^{\mu}-q'^{\mu})\sqrt{\mathbf{M}_{\mathfrak{c}}(p')\mathbf{M}_{\mathfrak{c}}(q')},
\end{align}
where
\begin{align*}
	\bar{s}=\bar{g}^2+4\mathfrak{c}^2,\quad \bar{g}^2=g^2-\frac{1}{2}(p^{\mu}+q^{\mu})(p'_{\mu}+q'_{\mu}-p_{\mu}-q_{\mu}).
\end{align*}
We introduce the Lorentz transformation $\bar{\Lambda}$ 
\begin{equation}\label{2.4-20}
	\bar{\Lambda}=\left(\bar{\Lambda}^{\mu}_{\nu}\right)=\left(\begin{array}{cccc}
		\tilde{r} & \frac{\tilde{r} v_1}{\mathfrak{c}} & \frac{\tilde{r} v_2}{\mathfrak{c}} & \frac{\tilde{r} v_3}{\mathfrak{c}} \\
		\frac{\tilde{r} v_1}{\mathfrak{c}} & 1+(\tilde{r}-1) \frac{v_1^2}{|v|^2} & (\tilde{r}-1) \frac{v_1 v_2}{|v|^2} & (\tilde{r}-1) \frac{v_1 v_3}{\mid v^2} \\
		\frac{\tilde{r} v_2}{\mathfrak{c}} & (\tilde{r}-1) \frac{v_1 v_2}{|v|^2} & 1+(\tilde{r}-1) \frac{v_2^2}{|v|^2} & (\tilde{r}-1) \frac{v_2 v_3}{|v|^2} \\
		\frac{\tilde{r} v_3}{\mathfrak{c}} & (\tilde{r}-1) \frac{v_1 v_3}{|v|^2} & (\tilde{r}-1) \frac{v_2 v_3}{|v|^2} & 1+(\tilde{r}-1) \frac{v_3^2}{|v|^2}
	\end{array}\right)
\end{equation}
and its inverse transformation
\begin{equation*}
	\bar{\Lambda}^{-1}=\left(\begin{array}{cccc}
		\tilde{r} & -\frac{\tilde{r} v_1}{\mathfrak{c}} & -\frac{\tilde{r} v_2}{\mathfrak{c}} & -\frac{\tilde{r} v_3}{\mathfrak{c}} \\
		-\frac{\tilde{r} v_1}{\mathfrak{c}} & 1+(\tilde{r}-1) \frac{v_1^2}{|v|^2} & (\tilde{r}-1) \frac{v_1 v_2}{|v|^2} & (\tilde{r}-1) \frac{v_1 v_3}{\mid v^2} \\
		-\frac{\tilde{r} v_2}{\mathfrak{c}} & (\tilde{r}-1) \frac{v_1 v_2}{|v|^2} & 1+(\tilde{r}-1) \frac{v_2^2}{|v|^2} & (\tilde{r}-1) \frac{v_2 v_3}{|v|^2} \\
		-\frac{\tilde{r} v_3}{\mathfrak{c}} & (\tilde{r}-1) \frac{v_1 v_3}{|v|^2} & (\tilde{r}-1) \frac{v_2 v_3}{|v|^2} & 1+(\tilde{r}-1) \frac{v_3^2}{|v|^2}
	\end{array}\right),
\end{equation*}
where $\tilde{r}=\frac{u^0}{\mathfrak{c}}, v_i=\frac{\mathfrak{c}u_i}{u^0}$. A direct calculation shows that 
\begin{align*}
	\bar{\Lambda}^{-1}(u^0,u^1,u^2,u^3)^t=(\mathfrak{c},0,0,0)^t.
\end{align*} 
Assume $\bar{\Lambda}\bar{P}=P$, then one has
\begin{equation} \label{3.3-0}
	\bar{P}=\bar{\Lambda}^{-1} P=
	\begin{pmatrix}
		\frac{u^0 p^0-u \cdot p}{\mathfrak{c}}   \\ 
		-\frac{u_1 p^0}{\mathfrak{c}}+p_1+\big(\frac{u^0}{\mathfrak{c}}-1\big) \frac{u_1}{|u|^2} u \cdot p \\
		-\frac{u_2 p^0}{\mathfrak{c}}+p_2+\big(\frac{u^0}{\mathfrak{c}}-1\big) \frac{u_2}{|u|^2} u \cdot p\\
		-\frac{u_3 p^0}{\mathfrak{c}}+p_3+\big(\frac{u^0}{\mathfrak{c}}-1\big) \frac{u_3}{|u|^2} u \cdot p
	\end{pmatrix}.
\end{equation}
Using Lorentz transformation $\bar{\Lambda}$, we can express $k_{\mathfrak{c}2}(p,q)$ as
\begin{align*}
	k_{\mathfrak{c}2}(p,q)=\frac{\mathfrak{c}c_0}{p^0q^0}\int_{\mathbb{R}^3}\frac{dq'}{q'^0}\int_{\mathbb{R}^3}\frac{dp'}{p'^0}\tilde{s}\delta^{(4)}(\bar{p}^\mu+p'^{\mu}-\bar{q}^{\mu}-q'^{\mu})e^{-\frac{\mathfrak{c}(p^{\prime 0}+q^{\prime 0})}{2T_0}},
\end{align*}
where $\tilde{s}=-(\bar{p}^{\prime \mu}+p^{\prime \mu})(\bar{p}_{\mu}+p'_{\mu})$ and
\begin{align*}
	c_0:=\frac{n_0\gamma}{4\pi\mathfrak{c}^3K_2(\gamma)}.
\end{align*}

By similar arguments as in \cite{Strain}, we can write $k_{\mathfrak{c}2}(p,q)$ as 
\begin{align}\label{2.13-10}
	k_{\mathfrak{c}2}(p,q)=\frac{\mathfrak{c}c_0\pi s^{\frac{3}{2}}}{4gp^0 q^0}\int_0^{\infty} \frac{y(1+\sqrt{y^2+1})}{\sqrt{y^2+1}}e^{-\bar{\boldsymbol{\ell}}\sqrt{y^2+1}}I_0(\bar{\boldsymbol{j}}y)dy,
\end{align}
where
\begin{align*}
	I_{0}(r)=\frac{1}{2 \pi} \int_{0}^{2 \pi} e^{r \cos \Theta} d \Theta
\end{align*}
and 
\begin{align*}
	\bar{\boldsymbol{\ell}}=\frac{\boldsymbol{\ell}}{T_0},\quad 
	\bar{\boldsymbol{j}}=\frac{\boldsymbol{j}}{T_0},\quad \boldsymbol{\ell}=\frac{\mathfrak{c}}{2}(\bar{p}^0+\bar{q}^0),\quad \boldsymbol{j}=\mathfrak{c}\frac{|\bar{p}\times \bar{q}|}{g}.
\end{align*}
Using the fact that for any $R>r \geq 0$,
\begin{align*}
	\int_0^{\infty} \frac{e^{-R \sqrt{1+y^2}} y I_0(r y)}{\sqrt{1+y^2}} d y & =\frac{e^{-\sqrt{R^2-r^2}}}{\sqrt{R^2-r^2}}, \\
	\int_0^{\infty} e^{-R \sqrt{1+y^2}} y I_0(r y) d y & =\frac{R}{R^2-r^2}\left\{1+\frac{1}{\sqrt{R^2-r^2}}\right\} e^{-\sqrt{R^2-r^2}},
\end{align*}
one can express $k_{\mathfrak{c}2}(p,q)$ as
\begin{align}\label{1.44-0}
	k_{\mathfrak{c}2}(p, q)=\frac{\mathfrak{c}c_0\pi s^{\frac{3}{2}}}{4gp^0 q^0}\left[J_{1}(\bar{\boldsymbol{\ell}}, \bar{\boldsymbol{j}})+J_{2}(\bar{\boldsymbol{\ell}}, \bar{\boldsymbol{j}})\right],
\end{align}
where 
    \begin{align}\label{2.12}
	\begin{split}
		J_{1}(\bar{\boldsymbol{\ell}}, \bar{\boldsymbol{j}})=\frac{\bar{\boldsymbol{\ell}}}{\bar{\boldsymbol{\ell}}^{2}-\bar{\boldsymbol{j}}^{2}}\left[1+\frac{1}{\sqrt{\bar{\boldsymbol{\ell}}^{2}-\bar{\boldsymbol{j}}^{2}}}\right] e^{-\sqrt{\bar{\boldsymbol{\ell}}^{2}-\bar{\boldsymbol{j}}^{2}}}, \quad  
		J_{2}(\bar{\boldsymbol{\ell}}, \bar{\boldsymbol{j}})=\frac{1}{\sqrt{\bar{\boldsymbol{\ell}}^{2}-\bar{\boldsymbol{j}}^{2}}} e^{-\sqrt{\bar{\boldsymbol{\ell}}^{2}-\bar{\boldsymbol{j}}^{2}}}.
	\end{split}
    \end{align}
    For later use, we denote the kernel of $\mathbf{K}_{\mathfrak{c}}$ as 
    \begin{align}\label{2.9-20}
    	k_{\mathfrak{c}}(p, q):=k_{\mathfrak{c}2}(p, q)-k_{\mathfrak{c}1}(p, q).
    \end{align}

    It is well-known that $\mathbf{L}_{\mathfrak{c}}$ is a self-adjoint non-negative definite operator in $L_{p}^{2}$ space with the kernel
    \begin{align*} 
    	\mathcal{N}_{\mathfrak{c}}=\operatorname{span}\left\{\sqrt{\mathbf{M}_{\mathfrak{c}}},\ p_i\sqrt{\mathbf{M}_{\mathfrak{c}}} \ (i=1,2,3),\ p^0\sqrt{\mathbf{M}_{\mathfrak{c}}} \right\}.
    \end{align*}
    Let $\mathbf{P}_{\mathfrak{c}}$ be the orthogonal projection from $L_{p}^{2}$ onto $\mathcal{N}_{\mathfrak{c}}$. For given $f$,  we denote the macroscopic part $\mathbf{P}_{\mathfrak{c}} f$ as
    \begin{align*}
    	\mathbf{P}_{\mathfrak{c}} f=\Big\{a_f+b_f\cdot p+c_fp^0\Big\}\sqrt{\mathbf{M}_{\mathfrak{c}}},
    \end{align*}
    and further denote $\{\mathbf{I}-\mathbf{P}_{\mathfrak{c}}\}f$ to be the microscopic part of $f$. For the orthonormal basis of $\mathcal{N}_{\mathfrak{c}}$, see the appendix. 
    
\section{The Newtonian limit of the relativistic Euler equations}

\subsection{Reformulation of the relativistic Euler equations}

By a delicate computation, the relativistic Euler equations \eqref{ce} can be rewritten as the following symmetric hyperbolic system 
\begin{align}\label{4.8-0}
	\mathbf{B}_0 \partial_t V+\sum_{j=1}^3 \mathbf{B}_j \partial_j V=0,
\end{align}
where 
\begin{align*}
	\mathbf{B}_0=\begin{pmatrix}
		1 & n_0\f{\partial P_0}{\partial n_0}\f{u^t}{(u^0)^2} & 0\\
		n_0\f{\partial P_0}{\partial n_0} \f{u}{(u^0)^2}& \frac{1}{\mathfrak{c}^2}n_0\f{\partial P_0}{\partial n_0}(e_0+P_0)(\mathbf{I}-\f{u\otimes u}{(u^0)^2}) & 0\\
		0 & 0 & \f{1}{\fc}u^0 
	\end{pmatrix}
\end{align*}
and 
\begin{align*}
	\mathbf{B}_j=\begin{pmatrix}
		\dis\f{\fc}{u^0}u_j & \f{\fc}{u^0}n_0\f{\partial P_0}{\partial n_0}\mathbf{e}_j^t & 0\\
		\f{\fc}{u^0}n_0\f{\partial P_0}{\partial n_0}\mathbf{e}_j& \frac{1}{\mathfrak{c}u^0}n_0\f{\partial P_0}{\partial n_0}[(e_0+P_0)u_j\mathbf{I}-\f{u\otimes u}{(u^0)^2}u_j] & 0\\
		0 & 0& u_j
	\end{pmatrix}.
\end{align*}
It is clear that $\mathbf{B}_0$ and $\mathbf{B}_j\ (j=1,2,3)$ are symmetric. Recall that
\begin{align*}
	\frac{\partial e_0}{\partial n_0}\Big|_{S}=\frac{e_0+P_0}{n_0}, \quad \frac{\partial e_0}{\partial S}\Big|_{n_0}=n_0T_0,
\end{align*}
then one has
\begin{align*}
	n_0\frac{\partial P_0}{\partial n_0}\Big|_{S}=n_0\frac{\partial P_0}{\partial e_0}\Big|_{S}\cdot \frac{\partial e_0}{\partial n_0}\Big|_{S}=\frac{a^2}{\mathfrak{c}^2}(e_0+P_0),
\end{align*}
where $a^2=\mathfrak{c}^2\frac{\partial P_0}{\partial e_0}|_{S}$ is the square of sound speed. Using the fact that $a\in \big(0,\frac{\mathfrak{c}}{\sqrt{3}}\big)$ (see \cite{Calvo,Speck}), one can show that $\mathbf{B}_0$ is a positive definite matrix.  

Denoting
\begin{align*}
	\zeta_0:=\frac{a}{\mathfrak{c}^2}(e_0+P_0)=a n_0\frac{K_3(\gamma)}{K_2(\gamma)}>0,
\end{align*}
we can rewrite $\mathbf{B}_0$ as 
\begin{align*}
	\mathbf{B}_0=\begin{pmatrix}
		1 & a\zeta_0\f{u^t}{(u^0)^2} & 0\\
		a\zeta_0 \f{u}{(u^0)^2}& \zeta_0^2(\mathbf{I}-\f{u\otimes u}{(u^0)^2}) & 0\\
		0 & 0 & \frac{u^0}{\mathfrak{c}}
	\end{pmatrix}.
\end{align*}

\subsection{Local smooth solutions to the relativistic Euler and classical Euler}

Assume that 
\begin{align*}
	\eta_1(V)\le \eta_2(V)\le \eta_3(V)\le \eta_4(V)\le \eta_5(V)
\end{align*}
are the five eigenvalues of $\mathbf{B}_0$. Since $\mathbf{B}_0$ is positive definite, it follows that $\eta_i(V)>0$ for all $V\neq 0$, $i=1,2,\cdots, 5$. By Vieta's theorem, one has
\begin{align}\label{4.19}
	\sum_{i=1}^5 \eta_i(V)&=\sum_{i=1}^5 (\mathbf{B}_0)_{ii}=1+\frac{u^0}{\mathfrak{c}}+\zeta_0^2\Big(2+\frac{\mathfrak{c}^2}{(u^0)^2}\Big)
\end{align}
and 
\begin{align}\label{4.20}
	\Pi_{i=1}^5 \eta_i(V)&=\operatorname{det}\mathbf{B}_0=\frac{\zeta_0^6}{\mathfrak{c}(u^0)^3}\Big[\mathfrak{c}^4+|u|^2(\mathfrak{c}^2-a^2)\Big].
\end{align}
Since all the elements of $\mathbf{B}_0$ are smooth functions of $V$, it yields that $\eta_i(V)\ (i=1,\cdots, 5)$ are continuous functions of $V$. Therefore, for any compact subset $\mathcal{V}\subset \mathbb{R}^{+}\times \mathbb{R}^3\times \mathbb{R}^{+}$ and  suitably large $\mathfrak{c}$, the RHS of \eqref{4.19} and \eqref{4.20} are bounded by positive constants from below and above and the constants are independent of $\mathfrak{c}$. Thus there exists a positive constant $\beta>0$ which is independent of $\mathfrak{c}$, such that
\begin{align}\label{3.21-01}
	\beta \mathbf{I}_5\le \mathbf{B}_0(V)\le \beta^{-1} \mathbf{I}_5, \quad V\in \mathcal{V}
\end{align}
holds in the sense of quadratic forms.



\begin{Lemma}[Local existence for the relativistic Euler equations]\label{thm-re}
	Considering the relativistic Euler equations \eqref{4.8-0} with a complete equation of state \eqref{3.32-0}-\eqref{3.34-0} in some open domain
	$\mathscr{V} \subset\left\{(P_0, u, S) \in \mathbb{R}^{+}\times \mathbb{R}^3 \times \mathbb{R}^{+}\right\}$, 
	we assume that $\overline{V}=(\overline{P},0,\overline{S})\in \mathcal{V}$ with $\overline{P}>0$, $\overline{S}>0$ being given constants which are independent of the light speed $\mathfrak{c}$.
	Suppose that
	$$
	V_0 \in \overline{V}+H^{N_0}\left(\mathbb{R}^3\right),
	$$
	with $N_0\ge 3$ and $V_0\in \mathscr{V}_1 \subset \subset \mathscr{V}$. Then there exist a local existing time $T_1>0$ which is independent of $\mathfrak{c}$, and a unique classical solution $V \in C^1\left([0, T] \times \mathbb{R}^3 \right)$ of the Cauchy problem associated with \eqref{4.8-0} and the initial data $V(0)=V_0$ such that $V-\overline{V}$ belongs to $C\left([0, T_1] ; H^{N_0}\right) \cap C^1\left([0, T_1] ; H^{N_0-1}\right)$ and the following estimate holds
	\begin{align*}
	\|V-\overline{V}\|_{C\left([0, T_1] ; H^{N_0}\right) \cap C^1\left([0, T_1] ; H^{N_0-1}\right)}\le C_1,
    \end{align*}
	where $C_1$ depends on $\|V_0-\overline{V}\|_{H^{N_0}}$ and is independent of $\mathfrak{c}$.
\end{Lemma}

\begin{proof}
	The proof is very similar to the one in \cite[Theorem 10.1]{Gavage}. The only difference lies in the argument of the independence of $\mathfrak{c}$ for $T_1$ and the upper bound for the solution. The fact that $T_1$ is independent of $\mathfrak{c}$ comes from the one that $\beta$ in \eqref{3.21-01} is  independent of $\mathfrak{c}$. In addition, from the specific expressions for the elements of $\mathbf{B}_{\alpha}\ (\alpha=0,1,2,3)$, we can easily derive that 
	\begin{align*}
		\|\nabla_x\mathbf{B}_{0}(V)\|_{H^{N_0-1}}+\sum_{j=1}^3	\|\mathbf{B}_{j}(V)\|_{H^{N_0}}\le C\|V-\overline{V}\|_{H^{N_0}},
	\end{align*}
	where $C$ depends on $\|V-\overline{V}\|_{L^{\infty}}$ and is independent of $\mathfrak{c}$. The remaining arguments are very similar to ones in \cite[Theorem 10.1]{Gavage} and we omit the details here for brevity. Therefore the proof is completed.
\end{proof}


For later use, we present the local result for the classical Euler equations \eqref{3.24-00}, see \cite{Friedrichs,Gavage, Kato,Majda} for instance. 

\begin{Lemma}\emph{\cite{Gavage}}\label{thm-ce}
	Considering the classical Euler equations \eqref{3.24-00} with equation of state \eqref{1.16-00} in some open domain
	$\mathscr{W} \subset\left\{(\mathcal{P}, \mathfrak{u}, \eta) \in \mathbb{R}^{+}\times \mathbb{R}^3\times \mathbb{R}^{+} \right\}$, 
	we assume that $\overline{W}=(\overline{\mathcal{P}},0,\overline{\eta})\in \mathcal{W}$ with $\overline{\mathcal{P}}>0$, $\overline{\eta}>0$ being given constants. Suppose that
	$$
	W_0 \in \overline{W}+H^{N_0}\left(\mathbb{R}^3\right), \quad \overline{W} \in \mathscr{W}
	$$
	with $N_0\ge 3$ and $W_0\in \mathscr{W}_1 \subset \subset \mathscr{W}$. Then there exist a local existing time $T_2>0$ and a unique classical solution $W \in C^1\left([0, T_2] \times \mathbb{R}^3 \right)$ of the Cauchy problem associated with \eqref{3.24-00} and the initial data $W(0)=W_0$ such that $W-\overline{W}$ belongs to $C\left([0, T_2] ; H^{N_0}\right) \cap C^1\left([0, T_2] ; H^{N_0-1}\right)$ and the following estimate holds
	\begin{align*}
		\|W-\overline{W}\|_{C\left([0, T_2] ; H^{N_0}\right) \cap C^1\left([0, T_2] ; H^{N_0-1}\right)}\le C_2,
	\end{align*}
	where $C_2$ depends on $\|W_0-\overline{W}\|_{H^{N_0}}$. Furthermore, the lifespan $T_2$ have the following lower bound
	\begin{align*}
		T_2\ge C_3 \Big(\|W_0-\overline{W}\|_{H^{N_0}}\Big)^{-1},
	\end{align*}
	where $C_3$ is independent of $\mathfrak{c}$ and $\|W_0-\overline{W}\|_{H^{N_0}}$. 
\end{Lemma}

\subsection{Newtonian limit from the relativistic Euler to the classical Euler}\label{sec3.3}

In this subsection, we focus on the Newtonian limit of the relativistic Euler equations. 
%
It follows from \eqref{3.24-00} and \eqref{4.8-0} that
\begin{align}\label{3.44-0}
	\mathbf{D}_0\partial_t(W-V)+\sum_{j=1}^3\mathbf{D}_j\partial_j (W-V)=\Upsilon,\quad (W-V)\Big|_{t=0}=0,
\end{align}
where 
\begin{align*}
	\Upsilon
	&=(\mathbf{B}_0-\mathbf{D}_0)\partial_t V+\sum_{j=1}^3(\mathbf{B}_j-\mathbf{D}_j)\partial_j V.
\end{align*}

\begin{Lemma}\label{lem3.3}
	There hold
	\begin{align}\label{3.24-01}
		\sigma^2=\frac{\partial \mathcal{P}}{\partial \rho}\Big|_{\eta}=\frac{5}{3}\theta
	\end{align}
	and 
	\begin{align}\label{3.25-01}
		a^2=\mathfrak{c}^2\frac{\partial P_0}{\partial e_0}\Big|_{S}=\frac{5}{3}T_0+O(\mathfrak{c}^{-2}).
	\end{align}
\end{Lemma}
\begin{proof}
	For \eqref{3.24-01}, it follows from \eqref{1.16-00} that
	\begin{align}\label{3.38-0}
		\mathcal{P}=\rho \theta=A_0^{\frac{2}{3}}\rho^{\frac{5}{3}}e^{\frac{2}{3}\eta},\quad A_0=(2\pi)^{-\frac{3}{2}}e^{-\frac{5}{2}},
	\end{align}
	which implies that
	\begin{align*}
		\frac{\partial \mathcal{P}}{\partial \rho}\Big|_{\eta}=\frac{5}{3}A_0^{\frac{2}{3}}\rho^{\frac{2}{3}}e^{\frac{2}{3}\eta}= \frac{5\mathcal{P}}{3\rho}=\frac{5}{3}\theta.
	\end{align*}

	For \eqref{3.25-01}, it follows from \eqref{3.32-0}--\eqref{3.34-0} that
	\begin{align*}
		P_0&=4 \pi e^4 e^{-S} \mathfrak{c}^5 \frac{K_2(\gamma)}{\gamma^2} \exp \left(\gamma \frac{K_1(\gamma)}{K_2(\gamma)}\right), \\
		e_0&=P_0\Big(\gamma \frac{K_1(\gamma)}{K_2(\gamma)}+3\Big).
	\end{align*}
	It is clear that
	\begin{align*}
		\frac{\partial P_0}{\partial \gamma}\Big|_{S}=\frac{\partial P_0}{\partial e_0}\Big|_{S}\cdot \frac{\partial e_0}{\partial \gamma}\Big|_{S}.
	\end{align*} 
	It follows from \cite[(3.32)]{Speck} that
	\begin{align*}
		\Big(\frac{\partial P_0}{\partial e_0}\Big|_{S}\Big)^{-1}=\frac{\frac{\partial e_0}{\partial \gamma}\Big|_{S}}{\frac{\partial P_0}{\partial \gamma}\Big|_{S}}=\gamma \frac{K_1(\gamma)}{K_2(\gamma)}+3+\frac{\gamma\left(\frac{K_1(\gamma)}{K_2(\gamma)}\right)^2+4 \frac{K_1(\gamma)}{K_2(\gamma)}-\gamma}{\gamma\left(\frac{K_1(\gamma)}{K_2(\gamma)}\right)^2+3 \frac{K_1(\gamma)}{K_2(\gamma)}-\gamma-\frac{4}{\gamma}} .
	\end{align*}
	Using the asymptotic expansions of $K_2(\gamma)$ and $K_3(\gamma)$ in Lemma \ref{lem2.1-00}, one has
	\begin{align}\label{3.54-0}
		\frac{K^2_3(\gamma)}{K^2_2(\gamma)}-1&=\frac{\frac{5}{\gamma}+\frac{115}{4\gamma^2}+\frac{2205}{32\gamma^3}+\frac{10395}{128\gamma^4}+O(\gamma^{-5})}{1+\frac{15}{4\gamma}+\frac{165}{32\gamma^2}+\frac{315}{128\gamma^3}+O(\gamma^{-4})}\nonumber\\
		&=\frac{5}{\gamma}+\frac{10}{\gamma^2}+\frac{45}{8\gamma^3}-\frac{15}{4\gamma^4}+O(\gamma^{-5})
	\end{align}
	and
	\begin{align}\label{3.55-0}
		\frac{K_3(\gamma)}{K_2(\gamma)}-1&=\frac{\frac{5}{2\gamma}+\frac{105}{16\gamma^2}+\frac{945}{256\gamma^3}+O(\gamma^{-4})}{1+\frac{15}{8\gamma}+\frac{105}{128\gamma^2}+O(\gamma^{-3})}=\frac{5}{2\gamma}+\frac{15}{8\gamma^2}-\frac{15}{8\gamma^{3}}+O(\gamma^{-4}).
	\end{align}
	Then one has
	\begin{align}\label{3.56-0}
		&\left(\frac{K_3(\gamma)}{K_2(\gamma)}\right)^2-\frac{5}{\gamma} \frac{K_3(\gamma)}{K_2(\gamma)}-1\nonumber\\
		&=\Big[\left(\frac{K_3(\gamma)}{K_2(\gamma)}\right)^2-1-\frac{5}{\gamma}\Big]-\frac{5}{\gamma}\Big(\frac{K_3(\gamma)}{K_2(\gamma)}-1\Big)\nonumber\\
		&=\frac{10}{\gamma^2}+\frac{45}{8\gamma^3}-\frac{15}{4\gamma^4}+O(\gamma^{-5})-\frac{5}{\gamma}\Big(\frac{5}{2\gamma}+\frac{15}{8\gamma^2}-\frac{15}{8\gamma^{3}}+O(\gamma^{-4})\Big)\nonumber\\
		&=-\frac{5}{2\gamma^{2}}-\frac{15}{4\gamma^{3}}+\frac{45}{8\gamma^4}+O(\gamma^{-5}).
	\end{align}
	Applying $\frac{K_1(\gamma)}{K_2(\gamma)}=\frac{K_3(\gamma)}{K_2(\gamma)}-\frac{4}{\gamma}$, we have
	\begin{align*}
		\Big(\gamma\frac{\partial P_0}{\partial e_0}\Big|_{S}\Big)^{-1}&= \frac{K_1(\gamma)}{K_2(\gamma)}+\frac{3}{\gamma}+\frac{\left(\frac{K_1(\gamma)}{K_2(\gamma)}\right)^2+\frac{4}{\gamma} \frac{K_1(\gamma)}{K_2(\gamma)}-1}{\left(\frac{K_1(\gamma)}{K_2(\gamma)}\right)^2+\frac{3}{\gamma} \frac{K_1(\gamma)}{K_2(\gamma)}-1-\frac{4}{\gamma^2}}\cdot \frac{1}{\gamma}\nonumber\\
		&=\frac{K_3(\gamma)}{K_2(\gamma)}+\frac{\frac{K_3(\gamma)}{K_2(\gamma)}}{\left(\frac{K_3(\gamma)}{K_2(\gamma)}\right)^2-\frac{5}{\gamma} \frac{K_3(\gamma)}{K_2(\gamma)}-1}\cdot \frac{1}{\gamma^2}\nonumber\\
		&=1+O(\gamma^{-1})+\frac{1+O(\gamma^{-1})}{-\frac{5}{2}+O(\gamma^{-1})}=\frac{3}{5}+O(\gamma^{-1}),
	\end{align*}
	which implies that
	\begin{align*}
		a^2=T_0\gamma\frac{\partial P_0}{\partial e_0}\Big|_{S}=T_0\Big(\frac{5}{3}+O(\gamma^{-1})\Big)=\frac{5}{3}T_0+O(\mathfrak{c}^{-2}).
	\end{align*}
	Therefore the proof is completed.
\end{proof}

\begin{Lemma}\label{lem3.4}
	It holds that
    \begin{align}\label{3.38-01}
    	|n_0-\rho|+|T_0-\theta|&\le  C|W-V|+\frac{C}{\mathfrak{c}^2},
    \end{align}
	where the constant $C$ depends on $\sup_{0\le t\le T}\|V(t)-\overline{V}\|_{H^{N_0}}$ and $\sup_{0\le t\le T}\|W(t)-\overline{W}\|_{H^{N_0}}$ and is independent of $\mathfrak{c}$. 
\end{Lemma}
\begin{proof}
	It follows from \eqref{3.38-0} that
	\begin{align}\label{3.61}
		\rho=(2\pi \mathcal{P})^{\frac{3}{5}}e^{1-\frac{2}{5}\eta}.
	\end{align}
	Since  $K_2(\gamma)=\sqrt{\frac{\pi}{2\gamma}}e^{-\gamma}(1+O(\gamma^{-1}))$, it follows from \eqref{3.34-0} that
	\begin{align*}
		n_0&=4 \pi e^{4-S} \mathfrak{c}^3 \frac{K_2(\gamma)}{\gamma} \exp \left(\gamma \frac{K_1(\gamma)}{K_2(\gamma)}\right)\nonumber\\
		&=(2\pi)^{\frac{3}{2}} \Big(\frac{P_0}{n_0}\Big)^{\frac{3}{2}}e^{-S}  \exp \left(\gamma \frac{K_3(\gamma)}{K_2(\gamma)}-\gamma\right)(1+O(\gamma^{-1})),
	\end{align*}
	which yields immediately that
	\begin{align}\label{3.62}
		n_0&=(2\pi)^{\frac{3}{5}} P_0^{\frac{3}{5}}e^{-\frac{2}{5}S}  \exp \left(\frac{2}{5}\gamma \frac{K_3(\gamma)}{K_2(\gamma)}-\frac{2}{5}\gamma\right)(1+O(\gamma^{-1}))\nonumber\\
		&=(2\pi)^{\frac{3}{5}} P_0^{\frac{3}{5}}e^{1-\frac{2}{5}S}  \exp \left(O(\gamma^{-1})\right)(1+O(\gamma^{-1}))\nonumber\\
		&=(2\pi P_0)^{\frac{3}{5}} e^{1-\frac{2}{5}S}+O(\gamma^{-1}).
	\end{align}
	Using \eqref{3.61}--\eqref{3.62}, one has
	\begin{align}\label{3.63}
		|n_0-\rho|&\le C|W-V|+\frac{C}{\mathfrak{c}^2}.
	\end{align}
	
	For the estimate of $|T_0-\theta|$ in \eqref{3.38-01}, a direct calculation shows that
	\begin{align*}
		T_0-\theta=\frac{P_0}{n_0}-\frac{\mathcal{P}}{\rho}=\frac{1}{n_0}(P_0-\mathcal{P})+\frac{P_0}{n_0\rho}(\rho-n_0),
	\end{align*}
	which, together with \eqref{3.63}, yields that
	\begin{align*}
		|T_0-\theta|&\le C|W-V|+\frac{C}{\mathfrak{c}^2}.
	\end{align*}
	Therefore the proof is completed.
\end{proof}

Since $n_0=n_0(P_0,S)$ and $\rho=\rho(\mathcal{P},\eta)$, to consider the Newtonian limit of the relativistic Euler equations, we still need to  control the following functions
\begin{align*}
	\frac{\partial n_0}{\partial P_0}\Big|_{S}-\frac{\partial \rho}{\partial \mathcal{P}}\Big|_{\eta},\quad \frac{\partial n_0}{\partial S}\Big|_{P_0}-\frac{\partial \rho}{\partial \eta}\Big|_{\mathcal{P}}.
\end{align*} 
For simplicity of notations, we replace $\frac{\partial n_0}{\partial P_0}\Big|_{S}$ with $\frac{\partial n_0}{\partial P_0}$ and the remaining notations can be understood in the same way.

\begin{Lemma}\label{lem3.5}
	It holds that 
	\begin{align}\label{3.69-0}
		\Big|\frac{\partial n_0}{\partial P_0} -\frac{\partial \rho}{\partial \mathcal{P}} \Big|+\Big|\frac{\partial n_0}{\partial S} -\frac{\partial \rho}{\partial \eta} \Big|+\Big|\frac{\partial T_0}{\partial P_0} -\frac{\partial \theta}{\partial \mathcal{P}} \Big|+\Big|\frac{\partial T_0}{\partial S} -\frac{\partial \theta}{\partial \eta} \Big|\le C|W-V|+\frac{C}{\mathfrak{c}^{2}},
	\end{align}
    where the constant $C$ depends on $\sup_{0\le t\le T}\|V(t)-\overline{V}\|_{H^{N_0}}$ and $\sup_{0\le t\le T}\|W(t)-\overline{W}\|_{H^{N_0}}$ and is independent of $\mathfrak{c}$.
\end{Lemma}
\begin{proof}
	Using \eqref{3.61}, one has
	\begin{align}\label{3.72-0}
		\frac{\partial \rho}{\partial \mathcal{P}}=\frac{3\rho}{5\mathcal{P}}=\frac{3}{5\theta},\quad \frac{\partial \rho}{\partial \eta}=-\frac{2\mathcal{P}}{5\theta}=-\frac{2}{5}\rho.
	\end{align}
	Since $\gamma=\frac{\mathfrak{c}^2}{T_0}=\mathfrak{c}^2\frac{n_0}{P_0}$ and 
	\begin{align*}
		n_0&=4 \pi e^{-S} \mathfrak{c}^3 \frac{K_2(\gamma)}{\gamma} \exp \left(\gamma \frac{K_3(\gamma)}{K_2(\gamma)}\right),
	\end{align*}
	it holds that
	\begin{align}\label{3.74-0}
		\frac{\partial n_0}{\partial P_0}&=4 \pi e^{-S} \mathfrak{c}^3\frac{d}{d\gamma}\Big[\frac{K_2(\gamma)}{\gamma} \exp \left(\gamma \frac{K_3(\gamma)}{K_2(\gamma)}\right)\Big]\cdot \frac{\partial \gamma}{\partial P_0}\nonumber\\
		&=4 \pi e^{-S} \mathfrak{c}^3\Big[\frac{K'_2(\gamma)\gamma-K_2(\gamma)}{\gamma^2}+\frac{K_2(\gamma)}{\gamma}\Big(\frac{K_3(\gamma)}{K_2(\gamma)}+\gamma \frac{K'_3(\gamma)K_2(\gamma)-K_3(\gamma)K'_2(\gamma)}{K^2_2(\gamma)}\Big) \Big] \nonumber\\
		&\qquad \times \exp \left(\gamma \frac{K_3(\gamma)}{K_2(\gamma)}\right)\cdot \frac{\mathfrak{c}^2}{P_0}\Big(\frac{\partial n_0}{\partial P_0}-\frac{n_0}{P_0}\Big)\nonumber\\
		&=\gamma^2\Big(\frac{K^2_3(\gamma)}{K^2_2(\gamma)}-\frac{5}{\gamma}\frac{K_3(\gamma)}{K_2(\gamma)}-1+\frac{1}{\gamma^2}\Big)\Big(\frac{\partial n_0}{\partial P_0}-\frac{n_0}{P_0}\Big)\nonumber\\
		&=\varphi(\gamma)\Big(\frac{\partial n_0}{\partial P_0}-\frac{n_0}{P_0}\Big),
	\end{align}
	where we have denoted
	\begin{align*}
		\varphi(\gamma):=\gamma^2\Big(\frac{K^2_3(\gamma)}{K^2_2(\gamma)}-\frac{5}{\gamma}\frac{K_3(\gamma)}{K_2(\gamma)}-1+\frac{1}{\gamma^2}\Big).
	\end{align*}
	It follows from \eqref{3.56-0} that
	\begin{align}\label{3.76-0}
		\varphi(\gamma)=-\frac{3}{2}+O(\gamma^{-1}).
	\end{align}
	Substituting \eqref{3.76-0} into \eqref{3.74-0}, one gets 
	\begin{align}\label{3.78-0}
		\frac{\partial n_0}{\partial P_0}=\frac{3}{5T_0}+O(\gamma^{-1}).
	\end{align}
	Similarly, one has
	\begin{align*}
		\frac{\partial n_0}{\partial S}&=-n_0+4 \pi e^{-S} \mathfrak{c}^3\frac{d}{d\gamma}\Big[\frac{K_2(\gamma)}{\gamma} \exp \left(\gamma \frac{K_3(\gamma)}{K_2(\gamma)}\right)\Big]\cdot \frac{\partial \gamma}{\partial S}\nonumber\\
		&=-n_0+\varphi(\gamma)\frac{\partial n_0}{\partial S}=-n_0+\Big(-\frac{3}{2}+O(\gamma^{-1})\Big)\frac{\partial n_0}{\partial S},
	\end{align*}
	which yields that
	\begin{align}\label{3.80-0}
		\frac{\partial n_0}{\partial S}=-\frac{2}{5}n_0+O(\gamma^{-1}).
	\end{align}
	It follows from \eqref{3.72-0}, \eqref{3.78-0} and \eqref{3.80-0} that
	\begin{align}\label{3.81-0}
		\Big|\frac{\partial n_0}{\partial P_0}-\frac{\partial \rho}{\partial \mathcal{P}}\Big|+\Big|\frac{\partial n_0}{\partial s}-\frac{\partial \rho}{\partial \eta}\Big|&\le C\Big|T_0-\theta|+ C\Big|n_0-\rho|+\frac{C}{\mathfrak{c}^2}\le C|W-V|+\frac{C}{\mathfrak{c}^2}.
	\end{align}

	Next, we consider $\big|\frac{\partial T_0}{\partial P_0}-\frac{\partial \theta}{\partial \mathcal{P}}\big|$ and $\big|\frac{\partial T_0}{\partial S}-\frac{\partial \theta}{\partial \eta}\big|$ in \eqref{3.69-0}. Noting $T_0=\frac{P_0}{n_0}$ and $\theta=\frac{\mathcal{P}}{\rho}$, we have
	\begin{align}\label{3.82-0}
		\frac{\partial T_0}{\partial P_0}=\frac{1}{n_0}-\frac{T_0}{n_0}\frac{\partial n_0}{\partial P_0},\quad \frac{\partial \theta}{\partial \mathcal{P}}=\frac{1}{\rho}-\frac{\theta}{\rho}\frac{\partial \rho}{\partial \mathcal{P}}
	\end{align}
	and
	\begin{align}\label{3.83-0}
		\frac{\partial T_0}{\partial S}= -\frac{T_0}{n_0}\frac{\partial n_0}{\partial S},\quad \frac{\partial \theta}{\partial \eta}=-\frac{\theta}{\rho}\frac{\partial \rho}{\partial \eta}.
	\end{align}
	Hence it is clear that
	\begin{align*}
		\Big|\frac{\partial T_0}{\partial P_0}-\frac{\partial \theta}{\partial \mathcal{P}}\Big|\le C\Big|n_0-\rho\Big|+C\Big|T_0-\theta\Big|+C\Big|\frac{\partial n_0}{\partial P_0}-\frac{\partial \rho}{\partial \mathcal{P}}\Big|
	\end{align*}
	and
	\begin{align*}
		\Big|\frac{\partial T_0}{\partial S}-\frac{\partial \theta}{\partial \eta}\Big|\le C\Big|n_0-\rho\Big|+C\Big|T_0-\theta\Big|+C\Big|\frac{\partial n_0}{\partial S}-\frac{\partial \rho}{\partial \eta}\Big|,
	\end{align*}
	which, together with \eqref{3.81-0}, yield \eqref{3.69-0}. Therefore the proof is completed.
\end{proof}

\begin{Lemma}\label{lem3.6}
	There hold
	\begin{align}\label{3.86-0}
		\Big|\frac{\partial^2 n_0}{\partial P^2_0} -\frac{\partial^2 \rho}{\partial \mathcal{P}^2} \Big|+\Big|\frac{\partial^2 n_0}{\partial S^2} -\frac{\partial^2 \rho}{\partial \eta^2} \Big|+\Big|\frac{\partial^2 n_0}{\partial P_0 \partial S} -\frac{\partial^2 \rho}{\partial \mathcal{P} \partial \eta} \Big|\le C|W-V|+\frac{C}{\mathfrak{c}^{2}}
	\end{align}
	and
	\begin{align}\label{3.87-0}
		\Big|\frac{\partial^2 T_0}{\partial P^2_0} -\frac{\partial^2 \theta}{\partial \mathcal{P}^2} \Big|+\Big|\frac{\partial^2 T_0}{\partial S^2} -\frac{\partial^2 \theta}{\partial \eta^2} \Big|+\Big|\frac{\partial^2 T_0}{\partial P_0 \partial S} -\frac{\partial^2 \theta}{\partial \mathcal{P} \partial \eta} \Big|\le C|W-V|+\frac{C}{\mathfrak{c}^{2}},
	\end{align}
    where the constant $C$ depends on $\sup_{0\le t\le T}\|V(t)-\overline{V}\|_{H^{N_0}}$ and $\sup_{0\le t\le T}\|W(t)-\overline{W}\|_{H^{N_0}}$ and is independent of $\mathfrak{c}$.
\end{Lemma}
\begin{proof}
	It follows from \eqref{3.72-0} that
	\begin{align}\label{3.62-10}
		\frac{\partial^2 \rho}{\partial \mathcal{P}^2}=-\frac{6}{25}\frac{1}{\mathcal{P}\theta},\quad \frac{\partial^2 \rho}{\partial \mathcal{P} \partial \eta}=-\frac{6}{25}\frac{1}{\theta},\quad \frac{\partial^2 \rho}{\partial \eta^2}=\frac{4}{25}\rho.
	\end{align}
	Using \eqref{3.74-0}, one has
	\begin{align}\label{3.91-0}
		\frac{\partial^2 n_0}{\partial P_0^2}&=\varphi'(\gamma)\frac{\mathfrak{c}^2}{P_0}\Big(\frac{\partial n_0}{\partial P_0}-\frac{n_0}{P_0}\Big)\Big(\frac{\partial n_0}{\partial P_0}-\frac{n_0}{P_0}\Big)+\varphi(\gamma)\Big(\frac{\partial^2 n_0}{\partial P_0^2}-\frac{1}{P_0}\frac{\partial n_0}{\partial P_0}+\frac{n_0}{P_0^2}\Big)\nonumber\\
		&=\gamma \varphi'(\gamma)\frac{1}{n_0\varphi^2(\gamma)}\Big(\frac{\partial n_0}{\partial P_0}\Big)^2+\varphi(\gamma)\Big(\frac{\partial^2 n_0}{\partial P_0^2}-\frac{1}{P_0}\frac{\partial n_0}{\partial P_0}+\frac{1}{P_0 T_0}\Big).
	\end{align}
	Noting
	\begin{align}\label{3.92-0}
		\Big(\frac{K_3(\gamma)}{K_2(\gamma)}\Big)'=
		\frac{K^2_3(\gamma)}{K^2_2(\gamma)}-\frac{5}{\gamma}\frac{K_3(\gamma)}{K_2(\gamma)}-1,
	\end{align}
	one has
	\begin{align}\label{3.65-01}
		\varphi'(\gamma)&=\frac{d}{d\gamma}\Big\{ \gamma^2\Big(\frac{K^2_3(\gamma)}{K^2_2(\gamma)}-\frac{5}{\gamma}\frac{K_3(\gamma)}{K_2(\gamma)}-1+\frac{1}{\gamma^2}\Big) \Big\}\nonumber\\
		&=2\gamma\Big[\frac{K^2_3(\gamma)}{K^2_2(\gamma)}-1\Big]+2\gamma^2\Big[\frac{K_3(\gamma)}{K_2(\gamma)}-1\Big]\Big(\frac{K_3(\gamma)}{K_2(\gamma)}\Big)'\nonumber\\
		&\qquad +(2\gamma^2-5\gamma)\Big(\frac{K_3(\gamma)}{K_2(\gamma)}\Big)'-5 \frac{K_3(\gamma)}{K_2(\gamma)}:=\sum_{j=1}^4 \mathcal{R}_j.
	\end{align}
	Applying \eqref{3.54-0}, \eqref{3.55-0}, \eqref{3.56-0} and \eqref{3.92-0}, one can obtain
	\begin{align}
		\mathcal{R}_1&=2\gamma\Big[\frac{K^2_3(\gamma)}{K^2_2(\gamma)}-1\Big]=2\gamma\Big(\frac{5}{\gamma}+\frac{10}{\gamma^2}+\frac{45}{8\gamma^3}+O(\gamma^{-4})\Big)=10+\frac{20}{\gamma}+\frac{45}{4\gamma^2}+O(\gamma^{-3}),\label{3.66-01}\\
		\mathcal{R}_2&=2\gamma^2\Big[\frac{K_3(\gamma)}{K_2(\gamma)}-1\Big]\Big(\frac{K_3(\gamma)}{K_2(\gamma)}\Big)'\nonumber\\
		&=2\gamma^2\Big(\frac{5}{2\gamma}+\frac{15}{8\gamma^2}+O(\gamma^{-3})\Big)\Big(-\frac{5}{2\gamma^{2}}-\frac{15}{4\gamma^{3}}+O(\gamma^{-4})\Big)=-\frac{25}{2\gamma}-\frac{225}{8\gamma^2}+O(\gamma^{-3}),\label{3.67-01}\\
		\mathcal{R}_3&=(2\gamma^2-5\gamma)\Big(\frac{K_3(\gamma)}{K_2(\gamma)}\Big)'=(2\gamma^2-5\gamma)\Big(-\frac{5}{2\gamma^{2}}-\frac{15}{4\gamma^{3}}+\frac{45}{8\gamma^4}+O(\gamma^{-5})\Big)\nonumber\\
		&=-5+\frac{5}{\gamma}+\frac{30}{\gamma^2}+O(\gamma^{-3}),\label{3.68-01}\\
		\mathcal{R}_4&=-5\Big[\frac{K_3(\gamma)}{K_2(\gamma)}-1\Big]-5=-5\Big(\frac{5}{2\gamma}+\frac{15}{8\gamma^2}+O(\gamma^{-3})\Big)-5\nonumber\\
		&=-5-\frac{25}{2\gamma}-\frac{75}{8\gamma^2}+O(\gamma^{-3}).\label{3.69-001}
	\end{align}
	Hence it follows from \eqref{3.65-01}--\eqref{3.69-001} that
	\begin{align*}
		\varphi'(\gamma)=\frac{15}{4\gamma^2}+O(\gamma^{-3}),
	\end{align*}
	which implies that 
	\begin{align}\label{3.99-0}
		\gamma \varphi'(\gamma)\frac{1}{\varphi^2(\gamma)}=O(\gamma^{-1}).
	\end{align}
	Since $\frac{\partial n_0}{\partial P_0}=\frac{3}{5T_0}+O(\gamma^{-1})$ and $\varphi(\gamma)=-\frac{3}{2}+O(\gamma^{-1})$, it follows from \eqref{3.91-0} and \eqref{3.99-0} that
	\begin{align*}
		(1-\varphi(\gamma))\frac{\partial^2 n_0}{\partial P_0^2}=(-\frac{3}{2}+O(\gamma^{-1}))\Big\{-\frac{1}{P_0}\Big(\frac{3}{5T_0}+O(\gamma^{-1})\Big)+\frac{1}{P_0 T_0}\Big\}+O(\gamma^{-1}),
	\end{align*}
	which implies that
	\begin{align}\label{3.73-10}
		\frac{\partial^2 n_0}{\partial P_0^2}=-\frac{6}{25}\frac{1}{P_0T_0}+O(\gamma^{-1}).
	\end{align}
	Similarly, we can obtain that
	\begin{align}\label{3.74-10}
		\frac{\partial^2 n_0}{\partial P_0 \partial S}=-\frac{6}{25}\frac{1}{T_0}+O(\gamma^{-1}),\quad \frac{\partial^2 n_0}{\partial S^2}=\frac{4}{25}n_0+O(\gamma^{-1}).
	\end{align}
    Hence we conclude \eqref{3.86-0} from \eqref{3.62-10}, \eqref{3.73-10}-\eqref{3.74-10} and Lemma \ref{lem3.4}.
    
    Using \eqref{3.82-0} and \eqref{3.83-0}, one has
    \begin{align}
    		\frac{\partial^2 T_0}{\partial S^2}&=\Big(-\frac{1}{n_0}\frac{\partial T_0}{\partial S}+\frac{T_0}{n_0^2}\frac{\partial n_0}{\partial S}\Big)\frac{\partial n_0}{\partial S}-\frac{T_0}{n_0}\frac{\partial^2 n_0}{\partial S^2},\label{3.75-10}\\
    		\frac{\partial^2 T_0}{\partial S \partial P_0 }&=\Big(-\frac{1}{n_0}\frac{\partial T_0}{\partial P_0}+\frac{T_0}{n_0^2}\frac{\partial n_0}{\partial P_0}\Big)\frac{\partial n_0}{\partial S}-\frac{T_0}{n_0}\frac{\partial^2 n_0}{\partial S \partial P_0},\label{3.76-10}\\
    		\frac{\partial^2 T_0}{\partial P_0^2 }&=\Big(-\frac{1}{n_0^2}\frac{\partial n_0}{\partial P_0}-\frac{1}{n_0}\frac{\partial T_0}{\partial P_0}+\frac{T_0}{n_0^2}\frac{\partial n_0}{\partial P_0}\Big)\frac{\partial n_0}{\partial P_0}-\frac{T_0}{n_0}\frac{\partial^2 n_0}{\partial P_0^2},\label{3.77-10}
    \intertext{and}
    	\frac{\partial^2 \theta_0}{\partial \eta^2}&=\Big(-\frac{1}{\rho}\frac{\partial \theta}{\partial \eta}+\frac{\theta}{\rho^2}\frac{\partial \rho}{\partial \eta}\Big)\frac{\partial \rho}{\partial \eta}-\frac{\theta}{\rho}\frac{\partial^2 \rho}{\partial \eta^2},\label{3.78-10}\\
    	\frac{\partial^2 \theta}{\partial \eta \partial \mathcal{P} }&=\Big(-\frac{1}{\rho}\frac{\partial \theta}{\partial \mathcal{P}}+\frac{\theta}{\rho^2}\frac{\partial \rho}{\partial \mathcal{P}}\Big)\frac{\partial \rho}{\partial \eta}-\frac{\theta}{\rho}\frac{\partial^2 \rho}{\partial \eta \partial \mathcal{P}},\label{3.79-10}\\
    	\frac{\partial^2 \theta}{\partial \mathcal{P}^2 }&=\Big(-\frac{1}{\rho^2}\frac{\partial \rho}{\partial \mathcal{P}}-\frac{1}{\rho}\frac{\partial \theta}{\partial \mathcal{P}}+\frac{\theta}{\rho^2}\frac{\partial \rho}{\partial \mathcal{P}}\Big)\frac{\partial \rho}{\partial \mathcal{P}}-\frac{\theta}{\rho}\frac{\partial^2 \rho}{\partial \mathcal{P}^2}. \label{3.80-10}
    \end{align}
    Thus \eqref{3.87-0} follows from \eqref{3.75-10}--\eqref{3.80-10}, Lemmas \ref{lem3.4}--\ref{lem3.5} and \eqref{3.86-0}. Therefore the proof is completed.
\end{proof}
    By similar arguments as in Lemmas \ref{lem3.5}--\ref{lem3.6}, we can obtain the following lemma whose proof is omitted for brevity of presentation.
\begin{Lemma}\label{lem3.7}
	There hold
	\begin{align*}
		|\partial_i (a^2-\sigma^2)|\le C|W-V|+C|\partial_i(W-V)|+\frac{C}{\mathfrak{c}^2}, \quad i=1,2,3,
	\end{align*}
	and 
	\begin{align*}
		|\partial_{ij} (a^2-\sigma^2)|\le C|W-V|+C|\nabla_x(W-V)|+C|\partial_{ij}(W-V)|+\frac{C}{\mathfrak{c}^2},\quad i,j=1,2,3,
	\end{align*}
     where the constant $C$ depends on $\sup_{0\le t\le T}\|V(t)-\overline{V}\|_{H^{N_0}}$ and $\sup_{0\le t\le T}\|W(t)-\overline{W}\|_{H^{N_0}}$ and is independent of $\mathfrak{c}$.
\end{Lemma}

We are now in a position to show the Newtonian limit from the relativistic Euler equations to the classical Euler equations. 

\begin{Proposition}\label{thm-retoce}
	Assume $\overline{V}=\overline{W}$. Suppose that $V=(P_0,u,S)$ is the unique smooth solution in Lemma \ref{thm-re} and $W=(\mathcal{P},\mathfrak{u},\eta)$ is the unique smooth solution in Lemma 
	\ref{thm-ce} with the same initial data. Let $T=\min\{T_1,T_2\}$, then it holds that
	\begin{align*}
		\sup_{0\le t\le T}\|(W-V)(t)\|_{L^{\infty}}\le \frac{C}{\mathfrak{c}^2},
	\end{align*}
	where the constant $C$ depends on $\sup_{0\le t\le T}\|V(t)-\overline{V}\|_{H^{N_0}}$ and $\sup_{0\le t\le T}\|W(t)-\overline{W}\|_{H^{N_0}}$ and is independent of $\mathfrak{c}$.
\end{Proposition}
\begin{proof}
Using Lemmas \ref{lem3.3}--\ref{lem3.4}, we have
\begin{align*}
	|\mathbf{B}_{\alpha}-\mathbf{D}_{\alpha}|\le C|W-V|+\frac{C}{\mathfrak{c}^2},\quad \alpha=0,1,2,3.
\end{align*}
Denote $\mathcal{U}(t):=\langle \mathbf{D}_0(W-V), W-V\rangle (t)$. It follows from \eqref{3.44-0} that
\begin{align*}
	\frac{d}{dt}\mathcal{U}(t)&=\frac{d}{dt}\langle \mathbf{D}_0(W-V), W-V\rangle\nonumber\\
	&=\langle(\partial_t \mathbf{D}_0+\sum_{j=1}^3\partial_j \mathbf{D}_j)(W-V),W-V\rangle+2\langle\Upsilon,W-V\rangle\nonumber\\
	&\le C\|W-V\|^2_{2}+\frac{C}{\mathfrak{c}^2}\|W-V\|_{2}\nonumber\\
	&\le C\mathcal{U}(t)+\frac{C}{\mathfrak{c}^2}\sqrt{\mathcal{U}(t)}.
\end{align*}
Applying Gr\"{o}nwall's inequality, one obtains
\begin{align*}
	\sup_{0\le t\le T}\mathcal{U}(t)\le \frac{C}{\mathfrak{c}^4},
\end{align*}
where the constant $C$ depends on $T$, $\sup_{0\le t\le T}\|V(t)-\overline{V}\|_{H^{N_0}}$ and $\sup_{0\le t\le T}\|W(t)-\overline{W}\|_{H^{N_0}}$ and is independent of $\mathfrak{c}$. 
Hence we get
\begin{align}\label{3.134-0}
	\sup_{0\le t\le T}\|(W-V)(t)\|_{2}\le \frac{C}{\mathfrak{c}^2}.
\end{align}
Similarly, by using Lemmas \ref{lem3.5}--\ref{lem3.7} and the energy method, we can obtain
\begin{align*}
	\sup_{0\le t\le T}\|\nabla_x(W-V)(t)\|_{2}+\sup_{0\le t\le T}\|\nabla^2_x(W-V)(t)\|_{2}\le \frac{C}{\mathfrak{c}^2},
\end{align*}
which, together with \eqref{3.134-0}, yields that 
\begin{align*}
	\sup_{0\le t\le T}\|(W-V)(t)\|_{\infty}\le C\sup_{0\le t\le T}\|(W-V)(t)\|_{H^2}\le \frac{C}{\mathfrak{c}^2}.
\end{align*}
Therefore the proof is completed.
\end{proof}

Based on Lemmas \ref{thm-re}--\ref{thm-ce} and the diffeomorphisms $(n_0,T_0) \leftrightarrow (P_0,S)$, $(\rho,\theta)\leftrightarrow (\mathcal{P},\eta)$ of $(0,\infty)\times (0,\infty)$, there exist positive constants $\bar{C}_0$, $\bar{c}_j\ (j=1,2,3,4)$ which are independent of $\mathfrak{c}$, such that for any $(t,x)\in [0,T]\times \mathbb{R}^3$, there holds
\begin{align}\label{3.90-10}
	|u(t,x)|\le \bar{C}_0,\quad 0<4\bar{c}_1 \le T_0(t,x)\le \frac{1}{4\bar{c}_1},\quad  0<4\bar{c}_2 \le \theta(t,x)\le \frac{1}{4\bar{c}_2}
\end{align}
and 
\begin{align}\label{3.91-10}
	0<\bar{c}_3\le n_0(t,x)\le \frac{1}{\bar{c}_3},\quad 0<\bar{c}_4\le \rho(t,x)\le \frac{1}{\bar{c}_4}.
\end{align}
\section{Uniform-in-$\mathfrak{c}$ estimates on the linearized collision operators}
    We first present a useful lemma which is very similar to \cite[Lemma 3.1]{Glassey1}. Since the proof is similar, we omit the details here for brevity.
    \begin{Lemma}\emph{(\cite{Glassey1})}\label{lem4.1-00}
    	Denote 
    	\begin{align*}
    		\boldsymbol{\ell}_1:=\mathfrak{c}\frac{p^0+q^0}{2},\quad \boldsymbol{j}_1:=\mathfrak{c}\frac{|p\times q|}{g},
    	\end{align*}
    	then there hold\\
    	\item
    	$\displaystyle
    	(i)\ \frac{\sqrt{|p \times q|^{2}+\mathfrak{c}^2|p-q|^{2}}}{\sqrt{p^{0} q^{0}}} \leq g \leq|p-q| \ \text { and } \ g^2<s \leq 4p^{0} q^{0}.$ \\
    	\item
    	$\displaystyle
    	(ii)\ v_{\phi}=\frac{\mathfrak{c}}{4}\frac{g \sqrt{s}}{p^{0} q^{0}} \le \min\Big\{\mathfrak{c}, \frac{|p-q|}{2}\Big\}.$ \\
    	\item
    	$\displaystyle
    	(iii)\ \boldsymbol{\ell}_1^{2}-\boldsymbol{j}_1^{2}=\frac{s\mathfrak{c}^2}{4g^2}|p-q|^2 =\frac{\mathfrak{c}^2g^{2}+4\mathfrak{c}^4}{4 g^{2}}|p-q|^{2}\geq \mathfrak{c}^4+\frac{\mathfrak{c}^2}{4}|p-q|^{2}, \quad p\neq q.$
    	\item
    	$\displaystyle
    	(iv)\
    	\lim_{\mathfrak{c}\rightarrow\infty}\frac{g}{|p-q|}=\lim_{\mathfrak{c}\rightarrow\infty}\frac{s}{4\mathfrak{c}^2}=\lim_{\mathfrak{c}\rightarrow\infty}\frac{\boldsymbol{\ell}_1}{\mathfrak{c}^2}=	\lim_{\mathfrak{c}\rightarrow\infty}\frac{\boldsymbol{\ell}_1^{2}-\boldsymbol{j}_1^{2}}{\mathfrak{c}^4}=1,\quad p\neq q.$
    \end{Lemma}
    \begin{Lemma}\label{lem4.2-00}
        Recall $\bar{C}_0$ in \eqref{3.90-10} and $\bar{p}$ in \eqref{3.3-0}. For $\mathfrak{c}$ suitably large, there hold
       \begin{gather}
       	\frac{1}{2}|p-q|\le |\bar{p}-\bar{q}|\le \frac{3}{2}|p-q|, \quad p\in \mathbb{R}^3,\label{2.1-00}\\
       	\frac{1}{2}|p^0|\le |\bar{p}^0|\le \frac{3}{2}|p^0|,\quad p\in \mathbb{R}^3,\label{2.2-00}\\
       	\frac{|p|}{2}-\bar{C}_0\le |\bar{p}|\le \frac{3|p|}{2}+\bar{C}_0, \quad p\in \mathbb{R}^3, \label{2.3-00}\\
       	\frac{1}{2}\le \operatorname{det}\Big(\frac{\partial \bar{p}}{\partial p}\Big)\le \frac{3}{2}, \quad p\in \mathbb{R}^3. \label{2.4-00}
       \end{gather}
   \end{Lemma}

   \begin{proof}
   		It follows from \eqref{3.3-0} that 
   	\begin{align}\label{3.21-00}
   		\bar{p}-\bar{q}=p-q+\big(\frac{u^0}{\mathfrak{c}}-1\big)\frac{u\cdot (p-q)}{|u|^2}u-\frac{p^0-q^0}{\mathfrak{c}}u
   	\end{align}
   and 
   \begin{align}\label{2.6-00}
   	    \bar{p}^0=\frac{u^0}{\mathfrak{c}}p^0-\frac{u\cdot p}{\mathfrak{c}}=p^0+\Big(\frac{u^0}{\mathfrak{c}}-1\Big)p^0-\frac{u\cdot p}{\mathfrak{c}}.
   \end{align}
   	For $\frac{\bar{C}_0}{\mathfrak{c}}\le \frac{1}{4}$, it holds that
   	\begin{align*}
   		\Big|\big(\frac{u^0}{\mathfrak{c}}-1\big)\frac{u\cdot (p-q)}{|u|^2}u\Big|+\Big|\frac{p^0-q^0}{\mathfrak{c}}u\Big|&\le \frac{|u|^2}{\mathfrak{c}(u^0+\mathfrak{c})}|p-q|+\frac{|u|}{\mathfrak{c}}|p-q|\le \frac{1}{2}|p-q|
   	\end{align*}
   and
   \begin{align*}
   	\Big|\Big(\frac{u^0}{\mathfrak{c}}-1\Big)p^0-\frac{u\cdot p}{\mathfrak{c}}\Big|\le \frac{|u|^2}{\mathfrak{c}(u^0+\mathfrak{c})}p^0+\frac{|u|}{\mathfrak{c}}p^0\le \frac{1}{2}p^0,
   \end{align*}
   	which, together with \eqref{3.21-00} and \eqref{2.6-00}, yield \eqref{2.1-00} and \eqref{2.2-00}. 
   	 
   	Observing
   	\begin{align}\label{4.9-10}
   		\bar{p}_i=p_i+\Big(\frac{u^0}{\mathfrak{c}}-1\Big)\frac{u_i}{|u|^2}\sum_{j=1}^3u_jp_j-\frac{u_i}{\mathfrak{c}}p^0,
   	\end{align}
   we have
   \begin{align*}
   	\Big|\big(\frac{u^0}{\mathfrak{c}}-1\big) \frac{u \cdot p}{|u|^2}u\Big|+\Big|\frac{p^0}{\mathfrak{c}}u\Big|\le
   	\frac{|u|^2}{\mathfrak{c}(u^0+\mathfrak{c})}|p|+\frac{|u|}{\mathfrak{c}}|p|+|u|\le \frac{|p|}{2}+\bar{C}_0, 
   \end{align*}
   which, together with \eqref{4.9-10}, implies \eqref{2.3-00}.

   It follows from \eqref{4.9-10} that
   	\begin{align}\label{4.11-00}
   		\frac{\partial \bar{p}_i}{\partial p_j}=\delta_{ij}+\Big(\frac{u^0}{\mathfrak{c}}-1\Big)\frac{u_i}{|u|^2}u_j-\frac{u_i}{\mathfrak{c}}\frac{p_j}{p^0}.
   	\end{align}
   	For $\mathfrak{c}$ suitably large, it is clear that
   	\begin{align*}
   		\Big|\Big(\frac{u^0}{\mathfrak{c}}-1\Big)\frac{u_i}{|u|^2}u_j-\frac{u_i}{\mathfrak{c}}\frac{p_j}{p^0}\Big|\le \frac{|u|^2}{\mathfrak{c}(u^0+\mathfrak{c})}+\frac{|u|}{\mathfrak{c}}\le \frac{1}{16},
   	\end{align*}
   	which, together with \eqref{4.11-00}, implies \eqref{2.4-00}. 
    Therefore the proof is completed.
   \end{proof}
	\begin{Lemma}\label{lem4.3-00}
	Recall $\bar{c}_1$ in \eqref{3.90-10}. Then there hold
	\begin{align}\label{2.33}
		k_{\mathfrak{c}1}(p,q)\lesssim |\bar{p}-\bar{q}|e^{-2\bar{c}_1|\bar{p}|-2\bar{c}_1|\bar{q}|}\lesssim |p-q|e^{-\bar{c}_1|p|-\bar{c}_1|q|}
	\end{align}
	and
	\begin{align}\label{2.34-1}
		k_{\mathfrak{c}2}(p,q)\lesssim \Big[\frac{1}{\mathfrak{c}}+\frac{1}{|\bar{p}-\bar{q}|}\Big]e^{-2\bar{c}_1|\bar{p}-\bar{q}|} \lesssim  \Big[\frac{1}{\mathfrak{c}}+\frac{1}{|p-q|}\Big]e^{-\bar{c}_1|p-q|}.
	\end{align}
    Moreover, it holds that
    \begin{align}\label{2.21-0}
    	k_{\mathfrak{c}2}(p,q)\lesssim
    	\frac{1}{|\bar{p}-\bar{q}|}e^{-\bar{c}_1|\bar{p}-\bar{q}|}\lesssim \frac{1}{|p-q|}e^{-\frac{\bar{c}_1}{2}|p-q|}.
    \end{align}
\end{Lemma}
\begin{proof}
	For any $p\in \mathbb{R}^3$, it is clear that
	\begin{align*}
		\mathfrak{c}^2-\mathfrak{c}p^0=\mathfrak{c}^2(1-\sqrt{1+\frac{|p|^2}{\mathfrak{c}^2}})=-\frac{|p|^2}{1+\sqrt{1+\frac{|p|^2}{\mathfrak{c}^2}}},
	\end{align*}
	which yields
	\begin{align}\label{2.36-00}
		-\frac{|p|^2}{2}\le \mathfrak{c}^2-\mathfrak{c}p^0\le -\frac{|p|^2}{1+\sqrt{1+|p|^2}}=-\sqrt{1+|p|^2}+1\le -|p|+1.
	\end{align}
	It follows from \eqref{1.34-0}, Lemmas \ref{lem4.1-00}--\ref{lem4.2-00} and \eqref{2.36-00} that
	\begin{align*}
		k_{\mathfrak{c}1}(p,q)
		&\lesssim |p-q|\exp{\Big(\frac{\mathfrak{c}^2+u^{\mu}p_{\mu}}{2T_0}\Big)}\exp{\Big(\frac{\mathfrak{c}^2+u^{\mu}q_{\mu}}{2T_0}\Big)}\nonumber\\
		&= |p-q|\exp{\Big(\frac{\mathfrak{c}^2-\mathfrak{c}\bar{p}^{0}}{2T_0}\Big)}\exp{\Big(\frac{\mathfrak{c}^2-\mathfrak{c}\bar{q}^{0}}{2T_0}\Big)}\nonumber\\
		&\lesssim |p-q|\exp{\Big(-\frac{|\bar{p}|+|\bar{q}|}{2T_0}\Big)}\lesssim |p-q|\exp{\Big(-\frac{|p|+|q|}{4T_0}\Big)}
	\end{align*}
    and \eqref{2.33} follows.
	For \eqref{2.34-1}, observing $J_2(\bar{\boldsymbol{\ell}},\bar{\boldsymbol{j}})\le J_1(\bar{\boldsymbol{\ell}},\bar{\boldsymbol{j}})$, we have from \eqref{1.44-0}--\eqref{2.12} that 
	\begin{align*}
		k_{\mathfrak{c}2}(p,q)\lesssim \mathfrak{c}\frac{s^{3 / 2}}{g p^{0} q^{0}}e^{\frac{\mathfrak{c}^2}{T_0}}J_1(\bar{\boldsymbol{\ell}},\bar{\boldsymbol{j}})&\lesssim \mathfrak{c}\frac{s^{3 / 2}}{g p^{0} q^{0}}\frac{\bar{\boldsymbol{\ell}}}{\bar{\boldsymbol{\ell}}^{2}-\bar{\boldsymbol{j}}^{2}}\left[1+\frac{1}{\sqrt{\bar{\boldsymbol{\ell}}^{2}-\bar{\boldsymbol{j}^{2}}}}\right] e^{\frac{\mathfrak{c}^2-\sqrt{\boldsymbol{\ell}^{2}-\boldsymbol{j}^{2}}}{T_0}}\nonumber\\
		&\lesssim \mathfrak{c}\frac{s^{3 / 2}}{g p^{0} q^{0}}\frac{\bar{\boldsymbol{\ell}}}{\bar{\boldsymbol{\ell}}^{2}-\bar{\boldsymbol{j}}^{2}} e^{\frac{\mathfrak{c}^2-\sqrt{\boldsymbol{\ell}^{2}-\boldsymbol{j}^{2}}}{T_0}}.
	\end{align*}
	It follows from Lemma \ref{lem4.1-00} that
	\begin{align*}
		\mathfrak{c}^2-\sqrt{\boldsymbol{\ell}^2-\boldsymbol{j}^2}&\le \mathfrak{c}^2-\sqrt{\mathfrak{c}^4+\frac{\mathfrak{c}^2}{4}|\bar{p}-\bar{q}|^2}\le -\frac{\mathfrak{c}^2}{4} \frac{|\bar{p}-\bar{q}|^2}{\mathfrak{c}^2+\sqrt{\mathfrak{c}^4+\frac{\mathfrak{c}^2}{4}|\bar{p}-\bar{q}|^2}}\nonumber\\ 
		&= -\frac{1}{4} \frac{|\bar{p}-\bar{q}|^2}{1+\sqrt{1+\frac{1}{4\mathfrak{c}^2}|\bar{p}-\bar{q}|^2}}\le -\frac{1}{4} \frac{|\bar{p}-\bar{q}|^2}{1+\sqrt{1+\frac{1}{4}|\bar{p}-\bar{q}|^2}}\nonumber\\
		&=-\sqrt{1+\frac{1}{4}|\bar{p}-\bar{q}|^2}+1\le -\frac{|\bar{p}-\bar{q}|}{2}+1,
	\end{align*}
	then we have
	\begin{align}\label{2.40}
		k_{\mathfrak{c}2}(p,q)\lesssim& \mathfrak{c}\frac{s^{3 / 2}}{g p^{0} q^{0}} 
		\frac{\mathfrak{c}(\bar{p}^0+\bar{q}^0)}{2T_0}\frac{1}{\frac{s\mathfrak{c}^2|\bar{p}-\bar{q}|^2}{4g^2T_0^2}}e^{-\frac{|\bar{p}-\bar{q}|}{2T_0}}
		\lesssim\frac{s^{1 / 2}(\bar{p}^0+\bar{q}^0)}{p^{0} q^{0}}
		\frac{g}{|\bar{p}-\bar{q}|^2}e^{-\frac{|\bar{p}-\bar{q}|}{2T_0}}\nonumber\\
		\lesssim&\frac{\sqrt{g^2+4\mathfrak{c}^2}(\bar{p}^0+\bar{q}^0)}{p^{0} q^{0}}
		\frac{1}{|\bar{p}-\bar{q}|}e^{-\frac{|\bar{p}-\bar{q}|}{2T_0}}
		\lesssim\frac{(|\bar{p}-\bar{q}|+\mathfrak{c})(\bar{p}^0+\bar{q}^0)}{p^{0} q^{0}}
		\frac{1}{|\bar{p}-\bar{q}|}e^{-\frac{|\bar{p}-\bar{q}|}{2T_0}}\nonumber\\
		\lesssim & \Big[\frac{1}{p^0}+\frac{1}{q^0}+\frac{1}{|p-q|}\Big(\frac{\mathfrak{c}}{p^0}+\frac{\mathfrak{c}}{q^0}\Big)\Big]e^{-\frac{|p-q|}{4T_0}}
		\lesssim  \Big[\frac{1}{\mathfrak{c}}+\frac{1}{|p-q|}\Big]e^{-\bar{c}_1|p-q|},
	\end{align}
	where we used the fact that both $s$ and $g$ are Lorentz invariant, i.e., 
	\begin{align*}
		s(p,q)=s(\bar{p},\bar{q}),\quad g(p,q)=g(\bar{p},\bar{q}).
	\end{align*}
    Moreover, it follows from the fourth inequality of \eqref{2.40} that
    \begin{align*}
    	k_{\mathfrak{c}2}(p,q)&\lesssim \frac{(|\bar{p}-\bar{q}|+\mathfrak{c})(\bar{p}^0+\bar{q}^0)}{p^{0} q^{0}}
    	\frac{1}{|\bar{p}-\bar{q}|}e^{-\frac{|\bar{p}-\bar{q}|}{2T_0}}
    	\lesssim (|\bar{p}-\bar{q}|+1)
    	\frac{1}{|\bar{p}-\bar{q}|}e^{-\frac{|\bar{p}-\bar{q}|}{2T_0}}\nonumber\\
    	&\lesssim
    	\frac{1}{|\bar{p}-\bar{q}|}e^{-\frac{|\bar{p}-\bar{q}|}{4T_0}} \lesssim \frac{1}{|p-q|}e^{-\frac{\bar{c}_1}{2}|p-q|}. 
    \end{align*}
	Therefore the proof is completed.
\end{proof}
    It follows from Lemmas \ref{lem4.1-00}--\ref{lem4.2-00} and \eqref{2.36-00} that
    \begin{align}\label{4.14-20}
    	\mathbf{M}_{\mathfrak{c}}(t,x,p)\lesssim e^{-2\bar{c}_1|p|}.
    \end{align}

Based on Lemma \ref{lem4.3-00}, we can show the following key estimates.
\begin{Lemma} \label{lem4.4-00}
	Let $\ell \ge 0,\, 0\le \varpi\le \frac{\bar{c}_1}{4}$ and recall $w_{\ell}(p)$ in \eqref{1.46-00}. Then there hold
	\begin{align}\label{2.41}
		\int_{\mathbb{R}^3}k_{\mathfrak{c}1}(p,q)\frac{w_{\ell}(p)e^{\varpi |p|}}{w_{\ell}(q)e^{\varpi |q|}}dq&\lesssim \frac{1}{1+|p|}
	\end{align}
	and
	\begin{align}\label{2.42}
		\int_{\mathbb{R}^3}k_{\mathfrak{c}2}(p,q)\frac{w_{\ell}(p)e^{\varpi |p|}}{w_{\ell}(q)e^{\varpi |q|}}dq&\lesssim
		\left\{
		\begin{aligned}
			&\frac{1}{1+|p|}, \quad |p|\le \mathfrak{c},\\
			&\frac{1}{\mathfrak{c}}, \quad |p|\ge \mathfrak{c}.
		\end{aligned}
		\right.
	\end{align}
\end{Lemma}
\begin{proof}
	For $\ell> 0$, notice that
	\begin{align*}
		\frac{(1+|p|^2)^{\frac{\ell}{2}}e^{\varpi |p|}}{(1+|q|^2)^{\frac{\ell}{2}}e^{\varpi |q|}}\lesssim (1+|p-q|^2)^{\frac{\ell}{2}}e^{\varpi |p-q|},
	\end{align*}
	which can be absorbed by the exponential part $e^{-\frac{\bar{c}_1}{3}|p-q|}$ of $k_{\mathfrak{c}1}(p,q)$ and $k_{\mathfrak{c}2}(p,q)$. In the following, we mainly focus on the case of $\ell=0$.
	
	For \eqref{2.41}, it follows from \eqref{2.33} that
	\begin{align*}
		\int_{\mathbb{R}^3}k_{\mathfrak{c}1}(p,q)dq\lesssim \int_{\mathbb{R}^3}|p-q|e^{-\bar{c}_1|p|-\bar{c}_1|q|}dq\lesssim \int_{\mathbb{R}^3}e^{-\frac{\bar{c}_1}{2}|p|-\frac{\bar{c}_1}{2}|q|}dq\lesssim e^{-\frac{\bar{c}_1}{2}|p|}\lesssim \frac{1}{1+|p|}.
	\end{align*}
	Next we consider \eqref{2.42}. If $|p|\le 3\bar{C}_0$, we have from Lemmas \ref{lem4.2-00}--\ref{lem4.3-00} that
	\begin{align}\label{2.45-00}
		\int_{\mathbb{R}^3}k_{\mathfrak{c}2}(p,q)dq\lesssim \int_{\mathbb{R}^3}\frac{1}{|p-q|}e^{-\frac{\bar{c}_1}{2}|p-q|}d\bar{q}\lesssim 1\lesssim \frac{1}{1+|p|}.
	\end{align}
	For $|p|\ge 3\bar{C}_0$, we divide the proof into three cases. \\
	\noindent \textit{Case 1:} $|\bar{p}-\bar{q}|\ge \mathfrak{c}^{\frac{1}{8}}$. It holds that
	\begin{align}\label{2.47}
		\int_{|\bar{p}-\bar{q}|\ge\mathfrak{c}^{\frac{1}{8}}}k_{\mathfrak{c}2}(p,q)dq
		&\lesssim \int_{|\bar{p}-\bar{q}|\ge\mathfrak{c}^{\frac{1}{8}}}\frac{1}{|\bar{p}-\bar{q}|}e^{-\bar{c}_1|\bar{p}-\bar{q}|}d\bar{q}\nonumber\\
		&\lesssim e^{-\frac{\bar{c}_1}{2} \mathfrak{c}^{\frac{1}{8}}} \int_{|\bar{p}-\bar{q}|\ge\mathfrak{c}^{\frac{1}{8}}}\frac{1}{|\bar{p}-\bar{q}|}e^{-\frac{\bar{c}_1}{2}|\bar{p}-\bar{q}|}d\bar{q}\nonumber\\
		&\lesssim e^{-\frac{\bar{c}_1}{2} \mathfrak{c}^{\frac{1}{8}}}\lesssim \frac{1}{1+\mathfrak{c}}.
	\end{align}
	\noindent \textit{Case 2:} $|\bar{p}-\bar{q}|\le \mathfrak{c}^{\frac{1}{8}}$ and $|p|\le \mathfrak{c}$. For any $p,q\in \mathbb{R}^3$, it is direct to check that
	\begin{align*}
		g^2-|p-q|^2&=2p^0q^0-2\mathfrak{c}^2-2p\cdot q-|p-q|^2\nonumber\\
		&=2\frac{\mathfrak{c}^2(|p|^2+|q|^2)+|p|^2|q|^2}{p^0q^0+\mathfrak{c}^2}-(|p|^2+|q|^2)\nonumber\\
		&=\frac{2}{p^0q^0+\mathfrak{c}^2}\Big\{\mathfrak{c}^2(|p|^2+|q|^2)+|p|^2|q|^2-\frac{1}{2}(|p|^2+|q|^2)(p^0q^0+\mathfrak{c}^2)\Big\}\nonumber\\
		&=\frac{1}{p^0q^0+\mathfrak{c}^2}\{(\mathfrak{c}^2-p^0q^0)(|p|^2+|q|^2)+2|p|^2|q|^2\}\nonumber\\
		&= \frac{1}{p^0q^0+\mathfrak{c}^2}\Big\{-\frac{\mathfrak{c}^2(|p|^2+|q|^2)+|p|^2|q|^2}{\mathfrak{c}^2+p^0q^0}(|p|^2+|q|^2)+2|p|^2|q|^2\Big\}\nonumber\\
		&=-\frac{(q^0|p|^2-p^0|q|^2)^2}{(p^0q^0+\mathfrak{c}^2)^2},
	\end{align*}
	which, together with Lemma \ref{lem4.1-00}, yields that
	\begin{align}\label{2.49}
		\mathfrak{c}^2-\sqrt{\boldsymbol{\ell}^2-\boldsymbol{j}^2}&=\frac{\mathfrak{c}^4-(\boldsymbol{\ell}^2-\boldsymbol{j}^2)}{\mathfrak{c}^2+\sqrt{\boldsymbol{\ell}^2-\boldsymbol{j}^2}}=\frac{\mathfrak{c}^4-\frac{s\mathfrak{c}^2}{4g^2}|\bar{p}-\bar{q}|^2}{\mathfrak{c}^2+\frac{\sqrt{s}\mathfrak{c}}{2g}|\bar{p}-\bar{q}|}=\frac{\mathfrak{c}^2-\frac{s}{4g^2}|\bar{p}-\bar{q}|^2}{1+\sqrt{\frac{s}{4\mathfrak{c}^2}}\frac{|\bar{p}-\bar{q}|}{g}}\nonumber\\
		&=\frac{\mathfrak{c}^2-\frac{g^2+4\mathfrak{c}^2}{4g^2}|\bar{p}-\bar{q}|^2}{1+\sqrt{\frac{s}{4\mathfrak{c}^2}}\frac{|\bar{p}-\bar{q}|}{g}}=\frac{1}{1+\sqrt{\frac{s}{4\mathfrak{c}^2}}\frac{|\bar{p}-\bar{q}|}{g}}\Big\{-\frac{1}{4}|\bar{p}-\bar{q}|^2+\frac{\mathfrak{c}^2}{g^2}(g^2-|\bar{p}-\bar{q}|^2)\Big\}\nonumber\\
		&=-\frac{1}{4}\frac{|\bar{p}-\bar{q}|^2}{1+\sqrt{\frac{s}{4\mathfrak{c}^2}}\frac{|\bar{p}-\bar{q}|}{g}}-\frac{1}{1+\sqrt{\frac{s}{4\mathfrak{c}^2}}\frac{|\bar{p}-\bar{q}|}{g}}\frac{\mathfrak{c}^2}{g^2}\frac{(\bar{q}^0|\bar{p}|^2-\bar{p}^0|\bar{q}|^2)^2}{(\bar{p}^0\bar{q}^0+\mathfrak{c}^2)^2}.
	\end{align}
	Hence, it follows from \eqref{2.40} and \eqref{2.49} that
	\begin{align}\label{2.50}
		&\int_{|\bar{p}-\bar{q}|\le \mathfrak{c}^{\frac{1}{8}}}k_{\mathfrak{c}2}(p,q)dq\nonumber\\
		&\lesssim \int_{|\bar{p}-\bar{q}|\le\mathfrak{c}^{\frac{1}{8}} }\frac{|\bar{p}-\bar{q}|+1}{|\bar{p}-\bar{q}|}e^{-\bar{c}_1|\bar{p}-\bar{q}|}e^{\frac{1}{2T_0}(\mathfrak{c}^2-\sqrt{\boldsymbol{\ell}^2-\boldsymbol{j}^2})}dq\nonumber\\
		&\lesssim  \int_{|\bar{p}-\bar{q}|\le \mathfrak{c}^{\frac{1}{8}}}\frac{1}{|\bar{p}-\bar{q}|}e^{-\frac{\bar{c}_1}{2}|\bar{p}-\bar{q}|}e^{M(\bar{p},\bar{q})}
		\exp{\Big(-\frac{1}{8T_0}\frac{|\bar{p}-\bar{q}|^2}{1+\sqrt{\frac{s}{4\mathfrak{c}^2}}\frac{|\bar{p}-\bar{q}|}{g}}\Big)}d\bar{q}\nonumber\\
		&\lesssim  \int_{|\bar{p}-\bar{q}|\le \mathfrak{c}^{\frac{1}{8}}}\frac{1}{|\bar{p}-\bar{q}|}e^{-\frac{\bar{c}_1}{2}|\bar{p}-\bar{q}|}
		e^{M(\bar{p},\bar{q})}d\bar{q},
	\end{align}
	where we made a change of variables $q\rightarrow \bar{q}$ and 
	\begin{align}\label{2.51}
		M(\bar{p},\bar{q}):=-\frac{1}{1+\sqrt{\frac{s}{4\mathfrak{c}^2}}\frac{|\bar{p}-\bar{q}|}{g}}\frac{\mathfrak{c}^2}{g^2}\frac{(\bar{q}^0|\bar{p}|^2-\bar{p}^0|\bar{q}|^2)^2}{(\bar{p}^0\bar{q}^0+\mathfrak{c}^2)^2}\frac{1}{2T_0}.
	\end{align}
	Noting $|\bar{p}-\bar{q}|\le \mathfrak{c}^{\frac{1}{8}}$ and $|p|\le \mathfrak{c}$, one has $|\bar{p}|\lesssim \mathfrak{c}$ and so 
	\begin{align}\label{2.52}
		\bar{p}^0=\sqrt{\mathfrak{c}^2+|\bar{p}|^2}\lesssim \mathfrak{c},\quad 	\bar{q}^0=\sqrt{\mathfrak{c}^2+|\bar{q}|^2}\le \sqrt{\mathfrak{c}^2+2|\bar{p}-\bar{q}|^2+2|\bar{p}|^2}\lesssim \mathfrak{c},
	\end{align}
	which yields that
	\begin{align*}
		(1+\sqrt{\frac{s}{4\mathfrak{c}^2}}\frac{|\bar{p}-\bar{q}|}{g})g^2(\bar{p}^0\bar{q}^0+\mathfrak{c}^2)^2&\le (1+\sqrt{1+\frac{g^2}{4\mathfrak{c}^2}})|\bar{p}-\bar{q}|^2(\bar{p}^0\bar{q}^0+\mathfrak{c}^2)^2\nonumber\\
		&\lesssim \mathfrak{c}^4|\bar{p}-\bar{q}|^2.
	\end{align*}
	A direct calculation shows that
	\begin{align}\label{2.54-00}
		\bar{p}^0|\bar{q}|^2-\bar{q}^0|\bar{p}|^2&=\mathfrak{c}(|\bar{q}|^2-|\bar{p}|^2)+(\bar{p}^0-\mathfrak{c})|\bar{q}|^2-(\bar{q}^0-\mathfrak{c})|\bar{p}|^2\nonumber\\
		&=\mathfrak{c}(|\bar{q}|^2-|\bar{p}|^2)+\frac{|\bar{p}|^2|\bar{q}|^2}{\bar{p}^0+\mathfrak{c}}-\frac{|\bar{p}|^2|\bar{q}|^2}{\bar{q}^0+\mathfrak{c}}\nonumber\\
		&=\mathfrak{c}(|\bar{q}|^2-|\bar{p}|^2)+\frac{|\bar{p}|^2|\bar{q}|^2}{(\bar{p}^0+\mathfrak{c})(\bar{q}^0+\mathfrak{c})}(\bar{q}^0-\bar{p}^0)\nonumber\\
		&=\mathfrak{c}(|\bar{q}|^2-|\bar{p}|^2)+\frac{|\bar{p}|^2|\bar{q}|^2(|\bar{q}|^2-|\bar{p}|^2)}{(\bar{p}^0+\mathfrak{c})(\bar{q}^0+\mathfrak{c})(\bar{p}^0+\bar{q}^0)}.
	\end{align}
	Thus, in view of \eqref{3.90-10} and \eqref{2.52}--\eqref{2.54-00}, there exists a positive constant $\alpha_0$ which is independent of $\mathfrak{c}$ such that
	\begin{align}\label{2.55}
		M(\bar{p},\bar{q})&\le  -\alpha_0\frac{1}{\mathfrak{c}^2}\frac{1}{|\bar{p}-\bar{q}|^2}\left(\mathfrak{c}(|\bar{q}|^2-|\bar{p}|^2)+\frac{|\bar{p}|^2|\bar{q}|^2(|\bar{q}|^2-|\bar{p}|^2)}{(\bar{p}^0+\mathfrak{c})(\bar{q}^0+\mathfrak{c})(\bar{p}^0+\bar{q}^0)}\right)^2\nonumber\\
		&\le  -\alpha_0\frac{(|\bar{q}|^2-|\bar{p}|^2)^2}{|\bar{p}-\bar{q}|^2}.
	\end{align}
	Combining \eqref{2.50} and \eqref{2.55}, one has that
	\begin{align*}
		\int_{|\bar{p}-\bar{q}|\le \mathfrak{c}^{\frac{1}{8}}}k_{\mathfrak{c}2}(p,q)dq
		&\lesssim \int_{|\bar{p}-\bar{q}|\le \mathfrak{c}^{\frac{1}{8}}}\frac{1}{|\bar{p}-\bar{q}|}e^{-\frac{\bar{c}_1}{2}|\bar{p}-\bar{q}|}
		e^{-\alpha_0\frac{(|\bar{q}|^2-|\bar{p}|^2)^2}{|\bar{p}-\bar{q}|^2}}d\bar{q}.
	\end{align*}
	By taking similar arguments as in \cite[Lemma 3.3.1]{Glassey} (see also Case 3 below), we obtain
	\begin{align}\label{2.57}
		\int_{|\bar{p}-\bar{q}|\le \mathfrak{c}^{\frac{1}{8}}}k_{\mathfrak{c}2}(p,q)dq\lesssim \frac{1}{1+|\bar{p}|}\lesssim \frac{1}{1+|p|}, \quad \text{for}\ |p|\le \mathfrak{c}.
	\end{align}
	\noindent \textit{Case 3:} $|\bar{p}-\bar{q}|\le \mathfrak{c}^{\frac{1}{8}}$ and $|p|\ge \mathfrak{c}$. It is clear that $|\bar{p}|\gtrsim \mathfrak{c}$. Noting
	\begin{align*}
		|\bar{q}|\le |\bar{q}-\bar{p}|+|\bar{p}|\lesssim |\bar{p}|,\quad |\bar{q}|\ge |\bar{p}|-|\bar{p}-\bar{q}|\gtrsim|\bar{p}|,
	\end{align*}
	then we have
	\begin{align*}
		|\bar{p}|\cong |\bar{q}|, \quad  \bar{p}^0\cong \bar{q}^0.
	\end{align*}
	Hence it is clear that
	\begin{align}\label{2.60}
		(1+\sqrt{\frac{s}{4\mathfrak{c}^2}}\frac{|\bar{p}-\bar{q}|}{g})g^2(\bar{p}^0\bar{q}^0+\mathfrak{c}^2)^2&\le (1+\sqrt{1+\frac{g^2}{4\mathfrak{c}^2}})|\bar{p}-\bar{q}|^2(\bar{p}^0\bar{q}^0+\mathfrak{c}^2)^2\nonumber\\
		&\lesssim |\bar{p}-\bar{q}|^2(\mathfrak{c}^2+|\bar{p}|^2)^2.
	\end{align}
	For $|\bar{p}|\gtrsim \mathfrak{c}$, it holds that
	\begin{align*}
		\mathfrak{c}+\frac{|\bar{p}|^2|\bar{q}|^2}{(\bar{p}^0+\mathfrak{c})(\bar{q}^0+\mathfrak{c})(\bar{p}^0+\bar{q}^0)}\cong \mathfrak{c}+\frac{|\bar{p}|^4}{(\bar{p}^0)^3}\cong \mathfrak{c}+\frac{|\bar{p}|^4}{(\mathfrak{c}^2+|\bar{p}|^2)^{\frac{3}{2}}}\cong \mathfrak{c}+|\bar{p}|,
	\end{align*}
	which, together with \eqref{2.54-00}, yields that
	\begin{align}\label{2.63}
		(\bar{p}^0|\bar{q}|^2-\bar{q}^0|\bar{p}|^2)^2\cong (|\bar{q}|^2-|\bar{p}|^2)^2(\mathfrak{c}^2+|\bar{p}|^2).
	\end{align}
	Combining \eqref{2.51}, \eqref{2.60} and \eqref{2.63}, for some positive constant $\alpha_1$ which is independent of $\mathfrak{c}$, we have
	\begin{align}\label{2.64}
		M(\bar{p},\bar{q})\le -\alpha_1 \frac{\mathfrak{c}^2}{\mathfrak{c}^2+|\bar{p}|^2}\frac{(|\bar{q}|^2-|\bar{p}|^2)^2}{|\bar{p}-\bar{q}|^2}.
	\end{align}
	Hence, for $|p|\ge \mathfrak{c}$, it follows from \eqref{2.50} and \eqref{2.64} that
	\begin{align*}
		\int_{|\bar{p}-\bar{q}|\le \mathfrak{c}^{\frac{1}{8}}}k_{\mathfrak{c}2}(p,q)dq&\lesssim \int_{|\bar{p}-\bar{q}|\le \mathfrak{c}^{\frac{1}{8}}}\frac{1}{|\bar{p}-\bar{q}|}e^{-\frac{\bar{c}_1}{2}|\bar{p}-\bar{q}|}
		e^{M(\bar{p},\bar{q})}d\bar{q}\nonumber\\
		&\lesssim \int_{|\bar{p}-\bar{q}|\le \mathfrak{c}^{\frac{1}{8}}}\frac{1}{|\bar{p}-\bar{q}|}e^{-\frac{\bar{c}_1}{2}|\bar{p}-\bar{q}|}
		e^{-\alpha_1\frac{\mathfrak{c}^2}{\mathfrak{c}^2+|\bar{p}|^2}\frac{(|\bar{q}|^2-|\bar{p}|^2)^2}{|\bar{p}-\bar{q}|^2}}d\bar{q}.
	\end{align*}
	Following the arguments as in \cite[Lemma 3.3.1]{Glassey}, we can make a change of variables
	\begin{align*}
		|\bar{p}-\bar{q}|=r, \quad (\bar{q}-\bar{p})\cdot \bar{p}=|\bar{p}|r\cos\theta, \quad 0\le r<\infty, \ 0\le \theta \le \pi,
	\end{align*}
	which yields that
	\begin{align*}
		|\bar{q}|^2=|\bar{q}-\bar{p}|^2+|\bar{p}|^2+2(\bar{q}-\bar{p})\cdot \bar{p}=r^2+|\bar{p}|^2+2r|\bar{p}|\cos\theta.
	\end{align*}
	Denoting $\alpha_2^2:=\alpha_1\frac{\mathfrak{c}^2}{\mathfrak{c}^2+|\bar{p}|^2}$ and $u=\alpha_2(r+2|\bar{p}|\cos\theta)$, one has
	\begin{align*}
		\int_{|\bar{p}-\bar{q}|\le \mathfrak{c}^{\frac{1}{4}}}k_{\mathfrak{c}2}(p,q)dq&\lesssim \int_{0}^{\infty}re^{-\frac{\bar{c}_1}{2}r}dr\int_{0}^{\pi}e^{-\alpha_2^2(r+2|\bar{p}|\cos\theta)^2}\sin\theta d\theta\nonumber\\
		&\lesssim \frac{1}{\alpha_2|\bar{p}|}\int_{-\infty}^{\infty}e^{-u^2}du\lesssim \frac{\sqrt{\mathfrak{c}^2+|\bar{p}|^2}}{\mathfrak{c}|\bar{p}|}\nonumber\\
		&\lesssim \frac{1}{\mathfrak{c}},
	\end{align*}
	which, together with \eqref{2.45-00}, \eqref{2.47}, \eqref{2.57}, yields \eqref{2.42}. Therefore the proof is completed.
\end{proof}
By similar arguments as in Lemma \ref{lem4.4-00}, one can also obtain
\begin{Lemma}\label{lem2.8}
	There hold
	\begin{align*}
		\int_{\mathbb{R}^3}k_{\mathfrak{c}1}^2(p,q)\Big(\frac{w_{\ell}(p)}{w_{\ell}(q)}\Big)^2dq\lesssim \frac{1}{1+|p|}
	\end{align*}
	and
	\begin{align*}
		\int_{\mathbb{R}^3}k_{\mathfrak{c}2}^2(p,q)\Big(\frac{w_{\ell}(p)}{w_{\ell}(q)}\Big)^2dq\lesssim
		\left\{
		\begin{aligned}
			&\frac{1}{1+|p|}, \quad |p|\le \mathfrak{c},\\
			&\frac{1}{\mathfrak{c}}, \quad |p|\ge \mathfrak{c}.
		\end{aligned}
		\right.
	\end{align*}
\end{Lemma}
Recall $k_{\mathfrak{c}}(p,q)=k_{\mathfrak{c}2}(p,q)-k_{\mathfrak{c}1}(p,q)$ in \eqref{2.9-20} and denote
\begin{align*}
	k_{\mathfrak{c} w}(p,q):=k_{\mathfrak{c}}(p,q)\frac{w_{\ell}(p)}{w_{\ell}(q)}.
\end{align*}
By similar arguments as in Lemma \ref{lem4.4-00}, one can also obtain
\begin{align*} 
	\int_{\mathbb{R}^3}k_{\mathfrak{c} w}(p,q)e^{\frac{\bar{c}_1}{4}|p-q|}dq\lesssim
	\left\{
	\begin{aligned}
		&\frac{1}{1+|p|}, \quad \text{for}\ |p|\le \mathfrak{c},\\
		&\frac{1}{\mathfrak{c}}, \quad \text{for}\  |p|\ge \mathfrak{c}.
	\end{aligned}
	\right.
\end{align*}
Next we estimate the collision frequency $\nu_{\mathfrak{c}}(p)$.
\begin{Lemma}\label{lem2.9}
	It holds that
	\begin{align}\label{2.72-10}
		\nu_{\mathfrak{c}}(p)\cong
		\left\{
		\begin{aligned}
			&1+|p|, \quad |p|\le \mathfrak{c},\\
			&\mathfrak{c}, \quad |p|\ge \mathfrak{c}.
		\end{aligned}
		\right.
	\end{align}
\end{Lemma}
\begin{proof}
	Recall
	\begin{align*}
		\nu_{\mathfrak{c}}(p)&=\int_{\mathbb{R}^3}\int_{\mathbb{S}^2}\frac{\mathfrak{c}}{4}\frac{g\sqrt{s}}{p^0q^0}\mathbf{M}_{\mathfrak{c}}(q)d\omega dq.
	\end{align*}
	Since the proof is complicated, we split it into four cases.
	
	\smallskip
	
	\noindent \textit{Case 1:} $|q|\ge \mathfrak{c}^{\frac{1}{8}}$. Using Lemma \ref{lem4.1-00} and \eqref{4.14-20}, one has
	\begin{align}\label{2.73}
		&\int_{|q|\ge \mathfrak{c}^{\frac{1}{8}}}\int_{\mathbb{S}^2}\frac{\mathfrak{c}}{4}\frac{g\sqrt{s}}{p^0q^0}\mathbf{M}_{\mathfrak{c}}(q)d\omega dq \lesssim \int_{|q|\ge \mathfrak{c}^{\frac{1}{8}}}\mathfrak{c}e^{-2\bar{c}_1|q|}dq 
		\lesssim e^{-\bar{c}_1\mathfrak{c}^{\frac{1}{8}}}.
	\end{align}
	\noindent \textit{Case 2:} $|q|\le \mathfrak{c}^{\frac{1}{8}}$ and $|p|\le \mathfrak{c}^{\frac{3}{8}}$. It holds that
	\begin{align}\label{2.74}
		&\int_{|q|\le \mathfrak{c}^{\frac{1}{8}}}\int_{\mathbb{S}^2}\frac{\mathfrak{c}g\sqrt{s}}{4p^0q^0}\frac{n_0\gamma}{4\pi \mathfrak{c}^3K_2(\gamma)}\exp{\Big(\frac{u^{\mu}q_{\mu}}{T_0}\Big)}d\omega dq \nonumber\\
		&=\int_{|q|\le \mathfrak{c}^{\frac{1}{8}}}\int_{\mathbb{S}^2}\frac{\mathfrak{c}g\sqrt{s}}{4p^0q^0}\frac{n_0}{(2\pi T_0)^{\frac{3}{2}}}(1+O(\gamma^{-1}))\exp{\Big(\frac{\mathfrak{c}^2-\mathfrak{c}\bar{q}^0}{T_0}\Big)}d\omega dq \nonumber\\
		&=\frac{n_0}{(2\pi T_0)^{\frac{3}{2}}}\int_{|q|\le \mathfrak{c}^{\frac{1}{8}}}\int_{\mathbb{S}^2}\frac{\mathfrak{c}g\sqrt{s}}{4p^0q^0}\exp{\Big(\frac{\mathfrak{c}^2-\mathfrak{c}\bar{q}^0}{T_0}\Big)}d\omega dq\cdot O(\gamma^{-1}) \nonumber\\
		&\qquad +\frac{n_0}{(2\pi T_0)^{\frac{3}{2}}}\int_{|q|\le \mathfrak{c}^{\frac{1}{8}}}\int_{\mathbb{S}^2}\Big(\frac{\mathfrak{c}g\sqrt{s}}{4p^0q^0}-\frac{|p-q|}{2}\Big)\exp{\Big(\frac{\mathfrak{c}^2-\mathfrak{c}\bar{q}^0}{T_0}\Big)}d\omega dq\nonumber\\
		&\qquad +\frac{n_0}{(2\pi T_0)^{\frac{3}{2}}}\int_{|q|\le \mathfrak{c}^{\frac{1}{8}}}\int_{\mathbb{S}^2}\frac{|p-q|}{2}\exp{\Big(\frac{\mathfrak{c}^2-\mathfrak{c}\bar{q}^0}{T_0}\Big)}d\omega dq\nonumber\\
		&:=\mathcal{H}_1+\mathcal{H}_2+\mathcal{H}_3.
	\end{align}
	It is clear that
	\begin{align}\label{2.75}
		|\mathcal{H}_1|\lesssim \frac{1}{\mathfrak{c}^2}\int_{|q|\le \mathfrak{c}^{\frac{1}{8}}} |p-q|e^{-2\bar{c}_1|q|}dq\cong \frac{1+|p|}{\mathfrak{c}^2}\lesssim \mathfrak{c}^{-\frac{13}{8}}.
	\end{align}
	Using Lemma \ref{lem4.2-00}, \eqref{3.90-10} and \eqref{2.36-00}, we have
	\begin{align}\label{2.76}
		\mathcal{H}_3 \gtrsim \int_{|\bar{q}|\le \frac{1}{2}\mathfrak{c}^{\frac{1}{8}}}|\bar{p}-\bar{q}|\exp{\Big(-\frac{|\bar{q}|^2}{8\bar{c}_1}\Big)} d\bar{q} \cong 1+|\bar{p}|\gtrsim1+|p|.
	\end{align}
	For $\mathcal{H}_2$, notice that
	\begin{align}\label{2.77}
		g^2&=2p^0q^0-2p\cdot q-2\mathfrak{c}^2=|p-q|^2+2p^0q^0-2\mathfrak{c}^2-|p|^2-|q|^2\nonumber\\
		&=|p-q|^2+\frac{4(|p|^2+\mathfrak{c}^2)(|q|^2+\mathfrak{c}^2)-(2\mathfrak{c}^2+|p|^2+|q|^2)^2}{2p^0q^0+(2\mathfrak{c}^2+|p|^2+|q|^2)}\nonumber\\	&=|p-q|^2-\frac{(|p|^2-|q|^2)^2}{2p^0q^0+(2\mathfrak{c}^2+|p|^2+|q|^2)},
	\end{align}
	then one has
	\begin{align*}
		\frac{\mathfrak{c}g\sqrt{s}}{4p^0q^0}-\frac{|p-q|}{2}&=\frac{1}{4p^0q^0}\{\mathfrak{c}g\sqrt{s}-2p^0q^0|p-q|\}\nonumber\\
		&=\frac{\mathfrak{c}^2g^2(g^2+4\mathfrak{c}^2)-4|p-q|^2(|p|^2+\mathfrak{c}^2)(|q|^2+\mathfrak{c}^2)}{4p^0q^0(\mathfrak{c}g\sqrt{s}+2p^0q^0|p-q|)}\nonumber\\
		&=\frac{4\mathfrak{c}^4(g^2-|p-q|^2)+\mathfrak{c}^2g^4- 4|p-q|^2\{|p|^2|q|^2+\mathfrak{c}^2(|p|^2+|q|^2)\}}{4p^0q^0(\mathfrak{c}g\sqrt{s}+2p^0q^0|p-q|)}\nonumber\\
		&\lesssim O(\mathfrak{c}^{-\frac{7}{8}}),
	\end{align*}
	which implies that
	\begin{align}\label{2.79}
		|\mathcal{H}_2|\lesssim \int_{|q|\le \mathfrak{c}^{\frac{1}{8}}} \mathfrak{c}^{-\frac{7}{8}}e^{-2\bar{c}_1|q|}dq\lesssim \mathfrak{c}^{-\frac{7}{8}}.
	\end{align}
	It follows from \eqref{2.74}--\eqref{2.76} and \eqref{2.79} that
	\begin{align}\label{2.80}
		\int_{|q|\le \mathfrak{c}^{\frac{1}{8}}}\int_{\mathbb{S}^2}\frac{\mathfrak{c}}{4}\frac{g\sqrt{s}}{p^0q^0}\mathbf{M}_{\mathfrak{c}}(q)d\omega dq \cong 1+|p|.
	\end{align}
	\noindent \textit{Case 3:} $|q|\le \mathfrak{c}^{\frac{1}{8}}$ and $\mathfrak{c}\ge |p|\ge \mathfrak{c}^{\frac{3}{8}}$.
	It follows from Lemma \ref{lem4.1-00} that
	\begin{align*}
		g\ge \frac{\mathfrak{c}|p-q|}{\sqrt{p^0q^0}}\gtrsim \frac{\mathfrak{c}|p|}{\mathfrak{c}}=|p|
	\end{align*}
	and
	\begin{align*}
		g\le |p-q|\lesssim |p|,
	\end{align*}
	which yields that $g\cong |p|$. Thus we have
	\begin{align}\label{2.84}
		\int_{|q|\le \mathfrak{c}^{\frac{1}{8}}}\int_{\mathbb{S}^2}\frac{\mathfrak{c}}{4}\frac{g\sqrt{s}}{p^0q^0}\mathbf{M}_{\mathfrak{c}}(q)d\omega dq
		\cong  \int_{|q|\le \mathfrak{c}^{\frac{1}{8}}}\int_{\mathbb{S}^2}|p|\exp{\Big(\frac{\mathfrak{c}^2-\mathfrak{c}\bar{q}^0}{T_0}\Big)}d\omega dq
		\cong 1+|p|.
	\end{align}
	\noindent \textit{Case 4:} $|q|\le \mathfrak{c}^{\frac{1}{8}}$ and $|p|\ge \mathfrak{c}$. It is obvious that
	\begin{align}\label{2.85}
		\int_{|q|\le \mathfrak{c}^{\frac{1}{8}}}\int_{\mathbb{S}^2}\frac{\mathfrak{c}}{4}\frac{g\sqrt{s}}{p^0q^0}\mathbf{M}_{\mathfrak{c}}(q)d\omega dq
		&\lesssim  \int_{|q|\le \mathfrak{c}^{\frac{1}{8}}}\mathfrak{c}e^{-2\bar{c}_1|q|} dq\lesssim \mathfrak{c}.
	\end{align}
	On the other hand, since $|p|\ge \mathfrak{c}$, one has
	\begin{align*}
		g\ge \frac{\mathfrak{c}|p-q|}{\sqrt{p^0q^0}}\gtrsim \frac{\mathfrak{c}|p|}{(|p|^2+\mathfrak{c}^2)^{\frac{1}{4}}\sqrt{\mathfrak{c}}}\gtrsim \sqrt{\mathfrak{c}|p|}.
	\end{align*}
	Thus we have
	\begin{align}\label{2.87}
		\int_{|q|\le \mathfrak{c}^{\frac{1}{8}}}\int_{\mathbb{S}^2}\frac{\mathfrak{c}}{4}\frac{g\sqrt{s}}{p^0q^0}\mathbf{M}_{\mathfrak{c}}(q)d\omega dq
		&\gtrsim \int_{|q|\le \mathfrak{c}^{\frac{1}{8}}}\frac{\sqrt{\mathfrak{c}|p|}\sqrt{\mathfrak{c}^2+\mathfrak{c}|p|}}{p^0}\exp{\Big(\frac{\mathfrak{c}^2-\mathfrak{c}\bar{q}^0}{T_0}\Big)} dq \gtrsim \mathfrak{c}.
	\end{align}
	It follows from \eqref{2.85} and \eqref{2.87} that
	\begin{align}\label{2.88}
		\int_{|q|\le \mathfrak{c}^{\frac{1}{8}}}\int_{\mathbb{S}^2}\frac{\mathfrak{c}}{4}\frac{g\sqrt{s}}{p^0q^0}\mathbf{M}_{\mathfrak{c}}(q)d\omega dq
		&\cong \mathfrak{c}.
	\end{align}
	Combining \eqref{2.73}, \eqref{2.80}, \eqref{2.84} and \eqref{2.88},  we conclude \eqref{2.72-10}. Therefore the proof is completed.
\end{proof}
\begin{remark}
	By similar arguments as in Lemma \ref{lem2.9}, we can obtain
	\begin{align}\label{2.90}
		\int_{\mathbb{R}^{3}}\int_{\mathbb{S}^2} v_{\phi} \mathbf{M}_{\mathfrak{c}}^{\alpha}(q)d \omega d q&\cong \nu_{\mathfrak{c}}(p), \quad \text{for}\ \alpha>0.
	\end{align}
\end{remark}

\subsection{Uniform-in-$\mathfrak{c}$ coercivity estimate on $\mathbf{L}_{\mathfrak{c}}$} In this subsection, we shall derive a uniform-in-$\mathfrak{c}$ coercivity estimate for the linearized relativistic collision operator $\mathbf{L}_{\mathfrak{c}}$.
For later use, we denote
\begin{align}
k_1(p,q)&:=2\pi|p-q|\frac{\rho}{(2\pi \theta)^{\frac{3}{2}}}e^{-\frac{|p-\mathfrak{u}|^2}{4\theta}-\frac{|q-\mathfrak{u}|^2}{4\theta}},\label{2.97}\\
k_2(p,q)&:=\frac{2}{|p-q|}\frac{\rho}{\sqrt{2\pi \theta}}e^{-\frac{|p-q|^2}{8\theta}-\frac{(|p-\mathfrak{u}|^2-|q-\mathfrak{u}|^2)^2}{8\theta|p-q|^2}},\label{2.98} 
\end{align}
which are indeed the corresponding kernels of Newtonian Boltzmann equation.

\begin{Lemma}\label{lem2.13}
It holds that
\begin{align}\label{2.106}
	\int_{\mathbb{R}^3}|k_{\mathfrak{c}1}(p,q)-k_1(p,q)|dq&\lesssim \mathfrak{c}^{-\frac{3}{2}}, \quad p\in \mathbb{R}^3.
\end{align}
\end{Lemma}
\begin{proof} We remark that throughout the proof, we make no attempt to be optimal in our estimates. We split the proof into three cases.\\
\noindent \textit{Case 1.} $|p|\ge \mathfrak{c}^{\frac{1}{8}}$. It follows from \eqref{3.90-10}, \eqref{2.97} and  Lemma \ref{lem4.2-00} that
\begin{align}\label{2.107}
	&\int_{\mathbb{R}^3}|k_{\mathfrak{c}1}(p,q)-k_1(p,q)|dq\nonumber\\
	&\lesssim  \int_{\mathbb{R}^3}|p-q|e^{-\bar{c}_1|p|-\bar{c}_1|q|}dq+\int_{\mathbb{R}^3}|p-q|e^{-\frac{|p|^2}{8\theta}-\frac{|q|^2}{8\theta}}dq\nonumber\\
	&\lesssim e^{-\frac{\bar{c}_1}{2}\mathfrak{c}^{\frac{1}{8}}}+e^{-\frac{\bar{c}_2}{4}\mathfrak{c}^{\frac{1}{4}}}\lesssim \mathfrak{c}^{-\frac{3}{2}}.
\end{align}
\noindent \textit{Case 2.} $|p|\le \mathfrak{c}^{\frac{1}{8}}$ and $|q|\ge \mathfrak{c}^{\frac{1}{8}}$. Similar to \eqref{2.107}, one has
\begin{align}\label{2.108}
	\int_{|q|\ge \mathfrak{c}^{\frac{1}{8}}}|k_{\mathfrak{c}1}(p,q)-k_1(p,q)|dq&\lesssim  \int_{|q|\ge \mathfrak{c}^{\frac{1}{8}}}|p-q|e^{-\bar{c}_1|p|-\bar{c}_1|q|}dq+\int_{|q|\ge \mathfrak{c}^{\frac{1}{8}}}|p-q|e^{-\frac{|p|^2}{8\theta}-\frac{|q|^2}{8\theta}}dq\nonumber\\
	&\lesssim e^{-\frac{\bar{c}_1}{2}\mathfrak{c}^{\frac{1}{8}}}+e^{-\frac{\bar{c}_2}{4}\mathfrak{c}^{\frac{1}{4}}}\lesssim \mathfrak{c}^{-\frac{3}{2}}.
\end{align}
\noindent \textit{Case 3.} $|p|\le \mathfrak{c}^{\frac{1}{8}}$ and $|q|\le \mathfrak{c}^{\frac{1}{8}}$. Recall that
\begin{align*}
	k_{\mathfrak{c}1}(p,q)&=\frac{\pi\mathfrak{c}g\sqrt{s}}{p^0q^0}\frac{n_0}{4\pi \mathfrak{c}T_0K_2(\gamma)}\exp{\Big(\frac{u^{\mu}p_{\mu}}{2T_0}\Big)}\exp{\Big(\frac{u^{\mu}q_{\mu}}{2T_0}\Big)}\nonumber\\
	&=\frac{\pi\mathfrak{c}g\sqrt{s}}{p^0q^0}\frac{n_0}{(2\pi T_0)^{\frac{3}{2}}}(1+O(\gamma^{-1}))\exp{\Big(\frac{\mathfrak{c}^2+u^{\mu}p_{\mu}}{2T_0}\Big)}\exp{\Big(\frac{\mathfrak{c}^2+u^{\mu}q_{\mu}}{2T_0}\Big)}\nonumber\\
	&=\frac{\pi\mathfrak{c}g\sqrt{s}}{p^0q^0}\frac{n_0}{(2\pi T_0)^{\frac{3}{2}}}(1+O(\gamma^{-1}))\exp{\Big(\frac{\mathfrak{c}^2-\mathfrak{c}\bar{p}^0}{2T_0}\Big)}\exp{\Big(\frac{\mathfrak{c}^2-\mathfrak{c}\bar{q}^0}{2T_0}\Big)}.
\end{align*}
Then we have
\begin{align}\label{4.45-00}
	&|k_{\mathfrak{c}1}(p,q)-k_1(p,q)|\nonumber\\
	&\le 
	\frac{\pi\mathfrak{c}g\sqrt{s}}{p^0q^0}\frac{n_0}{(2\pi T_0)^{\frac{3}{2}}}\exp{\Big(\frac{\mathfrak{c}^2-\mathfrak{c}\bar{p}^0}{2T_0}\Big)}\exp{\Big(\frac{\mathfrak{c}^2-\mathfrak{c}\bar{q}^0}{2T_0}\Big)} \cdot O(\gamma^{-1})\nonumber\\
	&\qquad +\Big|\frac{\pi\mathfrak{c}g\sqrt{s}}{p^0q^0}-2\pi|p-q|\Big|\frac{n_0}{(2\pi T_0)^{\frac{3}{2}}}\exp{\Big(\frac{\mathfrak{c}^2-\mathfrak{c}\bar{p}^0}{2T_0}\Big)}\exp{\Big(\frac{\mathfrak{c}^2-\mathfrak{c}\bar{q}^0}{2T_0}\Big)}\nonumber\\
	&\qquad +2\pi|p-q|\Big|\frac{n_0}{(2\pi T_0)^{\frac{3}{2}}}-\frac{\rho}{(2\pi \theta)^{\frac{3}{2}}}\Big|\exp{\Big(\frac{\mathfrak{c}^2-\mathfrak{c}\bar{p}^0}{2T_0}\Big)}\exp{\Big(\frac{\mathfrak{c}^2-\mathfrak{c}\bar{q}^0}{2T_0}\Big)}\nonumber\\
	&\qquad +
	 2\pi|p-q|\frac{\rho}{(2\pi \theta)^{\frac{3}{2}}}e^{-\frac{|p-\mathfrak{u}|^2}{4\theta}-\frac{|q-\mathfrak{u}|^2}{4\theta}}\Big|\exp{\Big(\frac{|p-\mathfrak{u}|^2}{4\theta}+\frac{|q-\mathfrak{u}|^2}{4\theta}+\frac{\mathfrak{c}^2-\mathfrak{c}\bar{p}^0}{2T_0}+\frac{\mathfrak{c}^2-\mathfrak{c}\bar{q}^0}{2T_0}\Big)}-1\Big|\nonumber\\
	 &:=\mathcal{D}_1+\mathcal{D}_2+\mathcal{D}_3+\mathcal{D}_4.
\end{align}
It is clear that
\begin{align*} 
	|\mathcal{D}_1|\lesssim \frac{|p-q|}{\mathfrak{c}^2}e^{-\bar{c}_1|p|-\bar{c}_1|q|},
\end{align*}
which implies that 
\begin{align}\label{2.58-00}
	\int_{|q|\le \mathfrak{c}^{\frac{1}{8}}}|\mathcal{D}_1(q,p)|dq\lesssim \mathfrak{c}^{-2}.
\end{align}
For $\mathcal{D}_2$, we notice that
\begin{align}\label{2.111}
	\frac{\mathfrak{c}\sqrt{s}}{2p^0q^0}-1&=\frac{\mathfrak{c}\sqrt{s}-2p^0q^0}{2p^0q^0}
	=\frac{\mathfrak{c}^2g^2-4\mathfrak{c}^2(|p|^2+|q|^2)-4|p|^2|q|^2}{2p^0q^0(\mathfrak{c}\sqrt{s}+2p^0q^0)}\lesssim O(\mathfrak{c}^{-\frac{3}{2}}).
\end{align}
It follows from \eqref{2.77} that
\begin{align}\label{2.112}
	g^2-|p-q|^2&=-\frac{(|p|^2-|q|^2)^2}{2p^0q^0+(2\mathfrak{c}^2+|p|^2+|q|^2)}\lesssim O(\mathfrak{c}^{-\frac{3}{2}}),
\end{align}
which yields that
\begin{align}\label{2.113}
	|g-|p-q||=\frac{|g^2-|p-q|^2|}{g+|p-q|}\lesssim \frac{O(\mathfrak{c}^{-\frac{3}{2}})}{g+|p-q|}\lesssim\frac{O(\mathfrak{c}^{-\frac{3}{2}})}{|p-q|}.
\end{align}
Using \eqref{2.111} and \eqref{2.113}, one has
\begin{align*}
	\Big|\frac{\mathfrak{c}}{2}\frac{g\sqrt{s}}{p^0q^0}-|p-q|\Big|&\le |g(\frac{\mathfrak{c}\sqrt{s}}{2p^0q^0}-1)|+|g-|p-q||\nonumber\\
	&\lesssim (g+\frac{1}{|p-q|})\mathfrak{c}^{-\frac{3}{2}}\lesssim (|p-q|+\frac{1}{|p-q|})\mathfrak{c}^{-\frac{3}{2}},
\end{align*}
which implies that
\begin{align}\label{2.62-00}
	\int_{|q|\le \mathfrak{c}^{\frac{1}{8}}}|\mathcal{D}_2(q,p)|dq&\lesssim \int_{|q|\le \mathfrak{c}^{\frac{1}{8}}}\Big|\frac{\mathfrak{c}}{2}\frac{g\sqrt{s}}{p^0q^0}-|p-q|\Big|e^{-\bar{c}_1|p|-\bar{c}_1|q|}dq\nonumber\\
	&\lesssim  \mathfrak{c}^{-\frac{3}{2}}\int_{|q|\le \mathfrak{c}^{\frac{1}{8}}} \Big(|p-q|+\frac{1}{|p-q|}\Big)e^{-\bar{c}_1|p|-\bar{c}_1|q|}\lesssim \mathfrak{c}^{-\frac{3}{2}}.
\end{align}
For $\mathcal{D}_3$, it follows from Proposition \ref{thm-retoce} that
\begin{align}\label{2.64-0}
	\Big|\frac{n_0}{(2\pi T_0)^{\frac{3}{2}}}-\frac{\rho}{(2\pi \theta)^{\frac{3}{2}}}\Big|\lesssim |T_0-\theta|+|n_0-\rho|\lesssim \mathfrak{c}^{-2},
\end{align}
which yields that
\begin{align}\label{2.64-00}
	\int_{|q|\le \mathfrak{c}^{\frac{1}{8}}}|\mathcal{D}_3(q,p)|dq
	&\lesssim  \mathfrak{c}^{-2}\int_{|q|\le \mathfrak{c}^{\frac{1}{8}}} |p-q|e^{-\bar{c}_1|p|-\bar{c}_1|q|}\lesssim \mathfrak{c}^{-2}.
\end{align}
For $\mathcal{D}_4$, a direct calculation shows that
\begin{align}\label{4.68-00}
	\frac{|p-\mathfrak{u}|^2}{4\theta}+\frac{\mathfrak{c}^2-\mathfrak{c}\bar{p}^0}{2T_0}&=\frac{|p-\mathfrak{u}|^2}{4\theta T_0}(T_0-\theta)+\frac{1}{4T_0}(|p-\mathfrak{u}|^2+2\mathfrak{c}^2-2\mathfrak{c}\bar{p}^0)\nonumber\\
	&=\frac{|p-\mathfrak{u}|^2}{4\theta T_0}(T_0-\theta)+\frac{1}{4T_0}\Big[\frac{|p|^4}{(p^0+\mathfrak{c})^2}+2p\cdot (u-\mathfrak{u})+(|\mathfrak{u}|^2-|u|^2)\Big]\nonumber\\
	&\qquad +\frac{1}{4T_0}\Big[\frac{|u|^4}{(u^0+\mathfrak{c})^2}-2\frac{|p|^2|u|^2}{(u^0+\mathfrak{c})(p^0+\mathfrak{c})}\Big],
\end{align}
which implies that 
\begin{align}\label{4.69-00}
	\Big|\frac{|p-\mathfrak{u}|^2}{4\theta}+\frac{\mathfrak{c}^2-\mathfrak{c}\bar{p}^0}{2T_0}\Big|\lesssim \mathfrak{c}^{-\frac{3}{2}}.
\end{align}
Similarly, one has
\begin{align*}
	\Big|\frac{|q-\mathfrak{u}|^2}{4\theta}+\frac{\mathfrak{c}^2-\mathfrak{c}\bar{q}^0}{2T_0}\Big|\lesssim \mathfrak{c}^{-\frac{3}{2}}.
\end{align*}
Thus we have
\begin{align}\label{2.67-00}
	\int_{|q|\le \mathfrak{c}^{\frac{1}{8}}}|\mathcal{D}_4(q,p)|dq\lesssim \mathfrak{c}^{-\frac{3}{2}}\int_{|q|\le \mathfrak{c}^{\frac{1}{8}}}|p-q|e^{-\frac{|p|^2}{8\theta}-\frac{|q|^2}{8\theta}}dq \lesssim \mathfrak{c}^{-\frac{3}{2}}.
\end{align}
Combining \eqref{4.45-00}, \eqref{2.58-00}, \eqref{2.62-00}, \eqref{2.64-00} and \eqref{2.67-00}, we have that
\begin{align}\label{2.116}
	\int_{|q|\le \mathfrak{c}^{\frac{1}{8}}}|k_{\mathfrak{c}1}(p,q)-k_1(p,q)|dq&\lesssim \mathfrak{c}^{-\frac{3}{2}}, \quad   |p|\le \mathfrak{c}^{\frac{1}{8}}.
\end{align}
Hence, we conclude \eqref{2.106} from \eqref{2.107}, \eqref{2.108} and \eqref{2.116}. Therefore the proof is completed.
\end{proof}

\begin{Lemma}\label{lem2.14}
It holds that
\begin{align*}
	\int_{\mathbb{R}^3}|k_{\mathfrak{c}2}(p,q)-k_2(p,q)|dq&\lesssim \mathfrak{c}^{-\frac{3}{8}}, \quad p\in \mathbb{R}^3.
\end{align*}
\end{Lemma}
\begin{proof}
Since the proof is complicated,	we split the proof into three cases.\\
\noindent \textit{Case 1.} $|p-q|\ge \mathfrak{c}^{\frac{1}{8}}$. It follows from \eqref{3.90-10}, \eqref{2.21-0} and \eqref{2.98} that
\begin{align} \label{2.119}
	\int_{|p-q|\ge \mathfrak{c}^{\frac{1}{8}}}|k_{\mathfrak{c}2}(p,q)-k_2(p,q)|dq&\lesssim  \int_{|p-q|\ge \mathfrak{c}^{\frac{1}{8}}}\frac{1}{|p-q|}e^{-\frac{\bar{c}_1}{2}|p-q|}dq+\int_{|p-q|\ge \mathfrak{c}^{\frac{1}{8}}}\frac{1}{|p-q|}e^{-\frac{|p-q|^2}{8\theta}}dq\nonumber\\
	&\lesssim e^{-\frac{\bar{c}_1}{4}\mathfrak{c}^{\frac{1}{8}}}+e^{-\frac{\bar{c}_2}{4}\mathfrak{c}^{\frac{1}{4}}}\lesssim \mathfrak{c}^{-\frac{3}{8}}.
\end{align}
\noindent \textit{Case 2.} $|p-q|\le \mathfrak{c}^{\frac{1}{8}}$ and $|p|\ge \mathfrak{c}^{\frac{3}{8}}$. By Lemma \ref{lem4.4-00} and similar arguments for $k_2(p,q)$ as in \cite[Lemma 3.3.1]{Glassey}, one has
\begin{align}\label{2.120}
	\int_{|p-q|\le \mathfrak{c}^{\frac{1}{8}}}|k_{\mathfrak{c}2}(p,q)-k_2(p,q)|dq
	\lesssim \frac{1}{\mathfrak{c}}+\frac{1}{1+|p|}+\frac{1}{1+|p-\mathfrak{u}|}\lesssim \mathfrak{c}^{-\frac{3}{8}}.
\end{align}
\noindent \textit{Case 3.} $|p-q|\le \mathfrak{c}^{\frac{1}{8}}$ and $|p|\le \mathfrak{c}^{\frac{3}{8}}$. In this case, we have $|q|\lesssim \mathfrak{c}^{\frac{3}{8}}$. Recall that
\begin{align*}
	k_{\mathfrak{c}2}(p, q)=\frac{\mathfrak{c}\pi s^{\frac{3}{2}}}{4gp^0 q^0}\frac{n_0}{(2\pi T_0)^{\frac{3}{2}}}(1+O(\gamma^{-1}))\frac{\bar{\boldsymbol{\ell}}\sqrt{\bar{\boldsymbol{\ell}}^{2}-\bar{\boldsymbol{j}}^{2}}+\bar{\boldsymbol{\ell}}+(\bar{\boldsymbol{\ell}}^{2}-\bar{\boldsymbol{j}}^{2})}{(\bar{\boldsymbol{\ell}}^{2}-\bar{\boldsymbol{j}}^{2})^{\frac{3}{2}}}e^{\frac{\mathfrak{c}^2-\sqrt{\boldsymbol{\ell}^2-\boldsymbol{j}^2}}{T_0}}.
\end{align*}
Then one has
\begin{align}\label{2.74-0}
	&|k_{\mathfrak{c}2}(p,q)-k_2(p,q)|\nonumber\\
	&\le \frac{\mathfrak{c}\pi s^{\frac{3}{2}}}{4gp^0 q^0}\frac{n_0}{(2\pi T_0)^{\frac{3}{2}}}\frac{\bar{\boldsymbol{\ell}}\sqrt{\bar{\boldsymbol{\ell}}^{2}-\bar{\boldsymbol{j}}^{2}}+\bar{\boldsymbol{\ell}}+(\bar{\boldsymbol{\ell}}^{2}-\bar{\boldsymbol{j}}^{2})}{(\bar{\boldsymbol{\ell}}^{2}-\bar{\boldsymbol{j}}^{2})^{\frac{3}{2}}}e^{\frac{\mathfrak{c}^2-\sqrt{\boldsymbol{\ell}^2-\boldsymbol{j}^2}}{T_0}}\cdot O(\gamma^{-1})\nonumber\\
	&\qquad +\frac{\mathfrak{c}\pi s^{\frac{3}{2}}}{4gp^0 q^0}\Big|\frac{n_0}{(2\pi T_0)^{\frac{3}{2}}}-\frac{\rho}{(2\pi \theta)^{\frac{3}{2}}}\Big|\frac{\bar{\boldsymbol{\ell}}\sqrt{\bar{\boldsymbol{\ell}}^{2}-\bar{\boldsymbol{j}}^{2}}+\bar{\boldsymbol{\ell}}+(\bar{\boldsymbol{\ell}}^{2}-\bar{\boldsymbol{j}}^{2})}{(\bar{\boldsymbol{\ell}}^{2}-\bar{\boldsymbol{j}}^{2})^{\frac{3}{2}}}e^{\frac{\mathfrak{c}^2-\sqrt{\boldsymbol{\ell}^2-\boldsymbol{j}^2}}{T_0}}\nonumber\\
	 &\qquad+ \frac{4\pi\rho}{(2\pi \theta)^{\frac{3}{2}}}\Big|\frac{\mathfrak{c} s^{\frac{3}{2}}}{16gp^0 q^0}\frac{\bar{\boldsymbol{\ell}}\sqrt{\bar{\boldsymbol{\ell}}^{2}-\bar{\boldsymbol{j}}^{2}}+\bar{\boldsymbol{\ell}}+(\bar{\boldsymbol{\ell}}^{2}-\bar{\boldsymbol{j}}^{2})}{(\bar{\boldsymbol{\ell}}^{2}-\bar{\boldsymbol{j}}^{2})^{\frac{3}{2}}}-\frac{\theta}{|p-q|}\Big|e^{\frac{\mathfrak{c}^2-\sqrt{\boldsymbol{\ell}^2-\boldsymbol{j}^2}}{T_0}}\nonumber\\
	 &\qquad+\frac{2}{|p-q|}\frac{\rho}{\sqrt{2\pi \theta}}e^{-\frac{|p-q|^2}{8\theta}-\frac{(|p-\mathfrak{u}|^2-|q-\mathfrak{u}|^2)^2}{8\theta|p-q|^2}}\Big|e^{\frac{\mathfrak{c}^2-\sqrt{\boldsymbol{\ell}^2-\boldsymbol{j}^2}}{T_0}+\frac{|p-q|^2}{8\theta}+\frac{(|p-\mathfrak{u}|^2-|q-\mathfrak{u}|^2)^2}{8\theta|p-q|^2}}-1\Big|\nonumber\\
	 &:=\mathcal{E}_1+\mathcal{E}_2+\mathcal{E}_3+\mathcal{E}_4.
\end{align}
It follows from \eqref{2.21-0} that 
\begin{align}\label{2.75-0}
	\int_{|p-q|\le \mathfrak{c}^{\frac{1}{8}}}|\mathcal{E}_1|dq\lesssim \frac{1}{\mathfrak{c}^2}\int_{|p-q|\le \mathfrak{c}^{\frac{1}{8}}} \frac{1}{|p-q|}e^{-\frac{\bar{c}_1}{2}|p-q|}dq\lesssim \frac{1}{\mathfrak{c}^2}.
\end{align}
By \eqref{2.64-0}, one has
\begin{align}\label{2.76-0}
	\int_{|p-q|\le \mathfrak{c}^{\frac{1}{8}}}|\mathcal{E}_2|dq\lesssim \frac{1}{\mathfrak{c}^2}\int_{|p-q|\le \mathfrak{c}^{\frac{1}{8}}} \frac{1}{|p-q|}e^{-\frac{\bar{c}_1}{2}|p-q|}dq\lesssim \frac{1}{\mathfrak{c}^2}.
\end{align}

We next focus on $\mathcal{E}_3$. It holds that
\begin{align}
	&\frac{\mathfrak{c} s^{\frac{3}{2}}}{16gp^0 q^0}\frac{\bar{\boldsymbol{\ell}}\sqrt{\bar{\boldsymbol{\ell}}^{2}-\bar{\boldsymbol{j}}^{2}}+\bar{\boldsymbol{\ell}}+(\bar{\boldsymbol{\ell}}^{2}-\bar{\boldsymbol{j}}^{2})}{(\bar{\boldsymbol{\ell}}^{2}-\bar{\boldsymbol{j}}^{2})^{\frac{3}{2}}}-\frac{\theta}{|p-q|}\nonumber\\
	&=\frac{1}{2\mathfrak{c}^2}\frac{g^2T_0^3}{p^0q^0|\bar{p}-\bar{q}|^3}\Big(\bar{\boldsymbol{\ell}}\sqrt{\bar{\boldsymbol{\ell}}^{2}-\bar{\boldsymbol{j}}^{2}}+\bar{\boldsymbol{\ell}}+(\bar{\boldsymbol{\ell}}^{2}-\bar{\boldsymbol{j}}^{2})\Big)-\frac{\theta}{|p-q|}\nonumber\\
	&=\frac{1}{2\mathfrak{c}^2}\frac{g^2}{p^0q^0|\bar{p}-\bar{q}|^3}\Big(\frac{\mathfrak{c}^2\sqrt{s}(\bar{p}^0+\bar{q}^0)}{4g}|\bar{p}-\bar{q}|T_0+\frac{\mathfrak{c}^2s}{4g^2}|\bar{p}-\bar{q}|^2T_0+\frac{\mathfrak{c}}{2}(\bar{p}^0+\bar{q}^0)T_0^2\Big)-\frac{\theta}{|p-q|}\nonumber\\
	&=\Big\{\frac{1}{2}\frac{g^2}{\bar{p}^0\bar{q}^0|\bar{p}-\bar{q}|^3}\Big(\frac{\sqrt{s}(\bar{p}^0+\bar{q}^0)}{4g}|\bar{p}-\bar{q}|+\frac{s}{4g^2}|\bar{p}-\bar{q}|^2\Big)\theta-\frac{\theta}{|\bar{p}-\bar{q}|}\Big\}+\frac{g^2(\bar{p}^0+\bar{q}^0)}{4\mathfrak{c}p^0q^0|\bar{p}-\bar{q}|^3}T_0^2\nonumber\\
	&\qquad +\frac{1}{2}\frac{g^2}{p^0q^0|\bar{p}-\bar{q}|^3}\Big(\frac{\sqrt{s}(\bar{p}^0+\bar{q}^0)}{4g}|\bar{p}-\bar{q}|+\frac{s}{4g^2}|\bar{p}-\bar{q}|^2\Big)(T_0-\theta)+\Big(\frac{\theta}{|\bar{p}-\bar{q}|}-\frac{\theta}{|p-q|}\Big)\nonumber\\
	&\qquad 
	+\frac{1}{2}\frac{g^2}{|\bar{p}-\bar{q}|^3}\Big(\frac{1}{p^0q^0}-\frac{1}{\bar{p}^0\bar{q}^0}\Big)\Big(\frac{\sqrt{s}(\bar{p}^0+\bar{q}^0)}{4g}|\bar{p}-\bar{q}|+\frac{s}{4g^2}|\bar{p}-\bar{q}|^2\Big)\theta\nonumber\\
	&:=\mathcal{E}_{31}+\mathcal{E}_{32}+\mathcal{E}_{33}+\mathcal{E}_{34}+\mathcal{E}_{35}.
\end{align}
A direct calculation shows that
\begin{align}\label{2.77-0}
	|\mathcal{E}_{32}|+|\mathcal{E}_{33}|+|\mathcal{E}_{34}|+|\mathcal{E}_{35}|\lesssim \frac{1}{|p-q|}\mathfrak{c}^{-\frac{13}{8}}.
\end{align}
For $\mathcal{E}_{31}$, one has
\begin{align*}
	\frac{\mathcal{E}_{31}}{\theta}|\bar{p}-\bar{q}|&=\frac{1}{2}\frac{g^2}{\bar{p}^0\bar{q}^0|\bar{p}-\bar{q}|^2}\Big(\frac{\sqrt{s}(\bar{p}^0+\bar{q}^0)}{4g}|\bar{p}-\bar{q}|+\frac{s}{4g^2}|\bar{p}-\bar{q}|^2\Big)-1\nonumber\\
	&=\frac{1}{2}\Big(\frac{\sqrt{s}g(\bar{p}^0+\bar{q}^0)}{4\bar{p}^0\bar{q}^0|\bar{p}-\bar{q}|}-1\Big)+\frac{1}{2}\Big(\frac{s}{4\bar{p}^0\bar{q}^0}-1\Big)\nonumber\\
	&:=\mathcal{E}_{311}+\mathcal{E}_{312}.
\end{align*}
For $\mathcal{E}_{312}$, we notice that
\begin{align*}
	\frac{s}{4\bar{p}^0\bar{q}^0}-1&=\frac{s-4\bar{p}^0\bar{q}^0}{4\bar{p}^0\bar{q}^0}=\frac{(g^2+4\mathfrak{c}^2)^2-16(\mathfrak{c}^2+|\bar{p}|^2)(\mathfrak{c}^2+|\bar{q}|^2)}{4\bar{p}^0\bar{q}^0(s+4\bar{p}^0\bar{q}^0)}\nonumber\\
	&=\frac{g^4+8g^2\mathfrak{c}^2-16\mathfrak{c}^2(|\bar{p}|^2+|\bar{q}|^2)-16|\bar{p}|^2|\bar{q}|^2}{4\bar{p}^0\bar{q}^0(s+4\bar{p}^0\bar{q}^0)}\nonumber\\
	&\lesssim O(\mathfrak{c}^{-\frac{5}{4}}).
\end{align*}
For $\mathcal{E}_{311}$, it is clear that
\begin{align*}
	\frac{\sqrt{s}g(\bar{p}^0+\bar{q}^0)}{4\bar{p}^0\bar{q}^0|\bar{p}-\bar{q}|}-1&=\frac{\sqrt{s}g(\bar{p}^0+\bar{q}^0)-4\bar{p}^0\bar{q}^0|\bar{p}-\bar{q}|}{4\bar{p}^0\bar{q}^0|\bar{p}-\bar{q}|}\nonumber\\
	&=\frac{\sqrt{s}g-2\bar{q}^0|\bar{p}-\bar{q}|}{4\bar{q}^0|\bar{p}-\bar{q}|}+\frac{\sqrt{s}g-2\bar{p}^0|\bar{p}-\bar{q}|}{4p^0|\bar{p}-\bar{q}|}.
\end{align*}
Due to \eqref{2.112}, one has
\begin{align*}
	\frac{\sqrt{s}g-2\bar{q}^0|\bar{p}-\bar{q}|}{4\bar{q}^0|\bar{p}-\bar{q}|}=&\frac{(g^2+4\mathfrak{c}^2)g^2-4(|\bar{q}|^2+\mathfrak{c}^2)|\bar{p}-\bar{q}|^2}{4\bar{q}^0|\bar{p}-\bar{q}|(\sqrt{s}g+2\bar{q}^0|\bar{p}-\bar{q}|)}\nonumber\\
	=&\frac{4\mathfrak{c}^2(g^2-|\bar{p}-\bar{q}|^2)+g^4-4|\bar{q}|^2|\bar{p}-\bar{q}|^2}{4\bar{q}^0|\bar{p}-\bar{q}|(\sqrt{s}g+2\bar{q}^0|\bar{p}-\bar{q}|)}\nonumber\\
	=&-\frac{4\mathfrak{c}^2(|\bar{p}|^2-|\bar{q}|^2)^2}{4\bar{q}^0|\bar{p}-\bar{q}|(\sqrt{s}g+2\bar{q}^0|\bar{p}-\bar{q}|)(2\bar{p}^0\bar{q}^0+(2\mathfrak{c}^2+|\bar{p}|^2+|\bar{q}|^2))}\nonumber\\
	&\quad +\frac{g^4-4|\bar{q}|^2|\bar{p}-\bar{q}|^2}{4\bar{q}^0|\bar{p}-\bar{q}|(\sqrt{s}g+2\bar{q}^0|\bar{p}-\bar{q}|)}\nonumber\\
	\lesssim &O(\mathfrak{c}^{-\frac{5}{4}})
\end{align*}
and
\begin{align*}
	\frac{\sqrt{s}g-2\bar{p}^0|\bar{p}-\bar{q}|}{4\bar{p}^0|\bar{p}-\bar{q}|}\lesssim O(\mathfrak{c}^{-\frac{5}{4}}).
\end{align*}
Thus we can obtain
\begin{align}\label{2.129}
	|\mathcal{E}_{31}|\lesssim \frac{1}{|p-q|}\mathfrak{c}^{-\frac{5}{4}}.
\end{align}
Combining \eqref{2.77-0} and \eqref{2.129}, one obtains, for $|p|\le \mathfrak{c}^{\frac{3}{8}}$, that
\begin{align}\label{2.81-0}
	\int_{|p-q|\le \mathfrak{c}^{\frac{1}{8}}}|\mathcal{E}_3|dq\lesssim \mathfrak{c}^{-\frac{5}{4}} \int_{|p-q|\le \mathfrak{c}^{\frac{1}{8}}}\frac{1}{|p-q|}e^{-\frac{\bar{c}_1}{2}|p-q|}dq\lesssim \mathfrak{c}^{-\frac{5}{4}}.
\end{align}

Next, we consider $\mathcal{E}_4$. It follows from \eqref{2.49} and \eqref{2.77} that
\begin{align*}
	&\frac{\mathfrak{c}^2-\sqrt{\boldsymbol{\ell}^2-\boldsymbol{j}^2}}{T_0}+\frac{|p-q|^2}{8\theta}+\frac{(|p-\mathfrak{u}|^2-|q-\mathfrak{u}|^2)^2}{8\theta|p-q|^2}\nonumber\\
	&=\frac{1}{T_0}\Big[\mathfrak{c}^2-\sqrt{\boldsymbol{\ell}^2-\boldsymbol{j}^2}+\frac{|\bar{p}-\bar{q}|^2}{8}+\frac{(|\bar{p}|^2-|\bar{q}|^2)^2}{8|\bar{p}-\bar{q}|^2}\Big]+\Big[\frac{|p-q|^2}{8\theta}-\frac{|\bar{p}-\bar{q}|^2}{8T_0}\Big]\nonumber\\
	&\qquad +\Big[\frac{(|p-\mathfrak{u}|^2-|q-\mathfrak{u}|^2)^2}{8\theta|p-q|^2}-\frac{(|\bar{p}|^2-|\bar{q}|^2)^2}{8T_0|\bar{p}-\bar{q}|^2}\Big]\nonumber\\
	&:=\mathcal{G}_1+\mathcal{G}_2+\mathcal{G}_3.
\end{align*}
For $\mathcal{G}_2$, using \eqref{3.3-0} and Proposition \ref{thm-retoce}, one has
\begin{align*}
	|\mathcal{G}_2|\lesssim \frac{1}{\theta}\Big||p-q|^2-|\bar{p}-\bar{q}|^2\Big|+|\bar{p}-\bar{q}|^2\frac{|\theta-T_0|}{\theta T_0}\lesssim \mathfrak{c}^{-\frac{7}{8}}.
\end{align*}
For $\mathcal{G}_3$, it holds that
\begin{align*}
	\mathcal{G}_3
	&=\frac{1}{8\theta}(|p-\mathfrak{u}|^2-|q-\mathfrak{u}|^2)^2\Big[\frac{1}{|p-q|^2}-\frac{1}{|\bar{p}-\bar{q}|^2}\Big]\nonumber\\
	&\qquad +\frac{1}{8\theta}\frac{1}{|\bar{p}-\bar{q}|^2}\Big[(|p-\mathfrak{u}|^2-|q-\mathfrak{u}|^2)^2-(|\bar{p}|^2-|\bar{q}|^2)^2\Big]\nonumber\\
	&\qquad +\frac{(|\bar{p}|^2-|\bar{q}|^2)^2}{|\bar{p}-\bar{q}|^2}\frac{T_0-\theta}{8\theta T_0}\nonumber\\
	&:=\mathcal{G}_{31}+\mathcal{G}_{32}+\mathcal{G}_{33}.
\end{align*}
Applying \eqref{3.3-0} and Proposition \ref{thm-retoce} again, we have
\begin{align*}
	|\mathcal{G}_{31}|\lesssim \mathfrak{c}^{-\frac{7}{8}},\quad |\mathcal{G}_{32}|\lesssim \mathfrak{c}^{-\frac{5}{4}}, \quad 	|\mathcal{G}_{33}|\lesssim  \mathfrak{c}^{-\frac{5}{4}},
\end{align*}
which yields that
\begin{align*}
		|\mathcal{G}_3|\lesssim \mathfrak{c}^{-\frac{7}{8}}.
\end{align*}
We finally focus on $\mathcal{G}_1$. It holds that
\begin{align}\label{2.131}
	&\frac{|\bar{p}-\bar{q}|^2}{8}+\frac{(|\bar{p}|^2-|\bar{q}|^2)^2}{8|\bar{p}-\bar{q}|^2}+\mathfrak{c}^2-\sqrt{\boldsymbol{\ell}^{2}-\boldsymbol{j}^{2}}\nonumber\\
	&=\Big[\frac{(|\bar{p}|^2-|\bar{q}|^2)^2}{8|\bar{p}-\bar{q}|^2}-\frac{1}{1+\sqrt{\frac{s}{4\mathfrak{c}^2}}\frac{|\bar{p}-\bar{q}|}{g}}\frac{\mathfrak{c}^2}{g^2}\frac{(|\bar{p}|^2-|\bar{q}|^2)^2}{2\bar{p}^0\bar{q}^0+(2\mathfrak{c}^2+|\bar{p}|^2+|\bar{q}|^2)}\Big]\nonumber\\
	&\quad +\Big[\frac{|\bar{p}-\bar{q}|^2}{8}-\frac{1}{4}\frac{|\bar{p}-\bar{q}|^2}{1+\sqrt{\frac{s}{4\mathfrak{c}^2}}\frac{|\bar{p}-\bar{q}|}{g}}\Big]\nonumber\\
	&:=\mathcal{G}_{11}+\mathcal{G}_{12}.
\end{align}
For $\mathcal{G}_{12}$, we have from \eqref{2.77} that
\begin{align}\label{2.132}
	\mathcal{G}_{12}=&\frac{|\bar{p}-\bar{q}|^2}{8}\Big(1-\frac{2}{1+\sqrt{\frac{s}{4\mathfrak{c}^2}}\frac{|\bar{p}-\bar{q}|}{g}}\Big)\nonumber\\
	=&\frac{|\bar{p}-\bar{q}|^2}{8}\frac{\sqrt{\frac{s}{4\mathfrak{c}^2}}\frac{|\bar{p}-\bar{q}|}{g}-1}{\sqrt{\frac{s}{4\mathfrak{c}^2}}\frac{|\bar{p}-\bar{q}|}{g}+1}=\frac{|\bar{p}-\bar{q}|^2}{8}\frac{\sqrt{s}|\bar{p}-\bar{q}|-2\mathfrak{c}g}{\sqrt{s}|\bar{p}-\bar{q}|+2\mathfrak{c}g}\nonumber\\
	=&\frac{|\bar{p}-\bar{q}|^2}{8}\frac{g^2|\bar{p}-\bar{q}|^2+4\mathfrak{c}^2(|\bar{p}-\bar{q}|^2-g^2)}{(\sqrt{s}|\bar{p}-\bar{q}|+2\mathfrak{c}g)^2}\nonumber\\	=&\frac{|\bar{p}-\bar{q}|^2}{8}\Big\{\frac{g^2|\bar{p}-\bar{q}|^2}{(\sqrt{s}|\bar{p}-\bar{q}|+2\mathfrak{c}g)^2}+\frac{4\mathfrak{c}^2}{(\sqrt{s}|\bar{p}-\bar{q}|+2\mathfrak{c}g)^2}\frac{(|\bar{p}|^2-|\bar{q}|^2)^2}{2\bar{p}^0\bar{q}^0+(2\mathfrak{c}^2+|\bar{p}|^2+|\bar{q}|^2)}\Big\}\nonumber\\
	\lesssim &O(\mathfrak{c}^{-1}).
\end{align}
For $\mathcal{G}_{11}$, it can be written as
\begin{align}\label{2.133}
	&\frac{(|\bar{p}|^2-|\bar{q}|^2)^2}{8|\bar{p}-\bar{q}|^2}\Big\{1-\frac{1}{1+\sqrt{\frac{s}{4\mathfrak{c}^2}}\frac{|\bar{p}-\bar{q}|}{g}}\frac{8|\bar{p}-\bar{q}|^2}{g^2}\frac{\mathfrak{c}^2}{2\bar{p}^0\bar{q}^0+(2\mathfrak{c}^2+|\bar{p}|^2+|\bar{q}|^2)}\Big\}\nonumber\\
	&=\frac{(|\bar{p}|^2-|\bar{q}|^2)^2}{8|\bar{p}-\bar{q}|^2}\frac{(2\mathfrak{c}g^2+\sqrt{s}|\bar{p}-\bar{q}|g)(2\bar{p}^0\bar{q}^0+(2\mathfrak{c}^2+|\bar{p}|^2+|\bar{q}|^2))-16\mathfrak{c}^3|\bar{p}-\bar{q}|^2}{(2\mathfrak{c}g^2+\sqrt{s}|\bar{p}-\bar{q}|g)(2\bar{p}^0\bar{q}^0+(2\mathfrak{c}^2+|\bar{p}|^2+|\bar{q}|^2))}\nonumber\\
	&=\frac{(|\bar{p}|^2-|\bar{q}|^2)^2}{8|\bar{p}-\bar{q}|^2}\frac{(4\mathfrak{c}^3g^2-4\mathfrak{c}^3|\bar{p}-\bar{q}|^2)+(4\mathfrak{c}g^2\bar{p}^0\bar{q}^0-4\mathfrak{c}^3|\bar{p}-\bar{q}|^2)+(2\sqrt{s}|\bar{p}-\bar{q}|g\mathfrak{c}^2-4\mathfrak{c}^3|\bar{p}-\bar{q}|^2)}{(2\mathfrak{c}g^2+\sqrt{s}|\bar{p}-\bar{q}|g)(2\bar{p}^0\bar{q}^0+(2\mathfrak{c}^2+|\bar{p}|^2+|\bar{q}|^2))}\nonumber\\
	&\quad+\frac{(|\bar{p}|^2-|\bar{q}|^2)^2}{8|\bar{p}-\bar{q}|^2}\frac{(2\sqrt{s}\bar{p}^0\bar{q}^0|\bar{p}-\bar{q}|g-4\mathfrak{c}^3|\bar{p}-\bar{q}|^2)+(2\mathfrak{c}g^2+\sqrt{s}|\bar{p}-\bar{q}|g)(|\bar{p}|^2+|\bar{q}|^2)}{(2\mathfrak{c}g^2+\sqrt{s}|\bar{p}-\bar{q}|g)(2\bar{p}^0\bar{q}^0+(2\mathfrak{c}^2+|\bar{p}|^2+|\bar{q}|^2))}\nonumber\\
	&:=\frac{(|\bar{p}|^2-|\bar{q}|^2)^2}{8|\bar{p}-\bar{q}|^2}(\mathcal{G}_{111}+\mathcal{G}_{112}+\mathcal{G}_{113}+\mathcal{G}_{114}+\mathcal{G}_{115}).
\end{align}
We have from \eqref{2.77} that
\begin{align}
	\mathcal{G}_{111}&=\frac{4\mathfrak{c}^3g^2-4\mathfrak{c}^3|\bar{p}-\bar{q}|^2}{(2\mathfrak{c}g^2+\sqrt{s}|\bar{p}-\bar{q}|g)(2\bar{p}^0\bar{q}^0+(2\mathfrak{c}^2+|\bar{p}|^2+|\bar{q}|^2))}\nonumber\\
	&=\frac{-4\mathfrak{c}^3}{(2\mathfrak{c}g^2+\sqrt{s}|\bar{p}-\bar{q}|g)(2\bar{p}^0\bar{q}^0+(2\mathfrak{c}^2+|\bar{p}|^2+|\bar{q}|^2))}\frac{(|\bar{p}|^2-|\bar{q}|^2)^2}{2\bar{p}^0\bar{q}^0+(2\mathfrak{c}^2+|\bar{p}|^2+|\bar{q}|^2)}\nonumber\\
	&\lesssim O(\mathfrak{c}^{-\frac{5}{4}}),\label{2.134}\\
\intertext{where we have used  $g^2\bar{p}^0\bar{q}^0\ge \mathfrak{c}^2|\bar{p}-\bar{q}|^2$. Similarly, one has} 
	\mathcal{G}_{112}&=\frac{4\mathfrak{c}g^2\bar{p}^0\bar{q}^0-4\mathfrak{c}^3|\bar{p}-\bar{q}|^2}{(2\mathfrak{c}g^2+\sqrt{s}|\bar{p}-\bar{q}|g)(2\bar{p}^0\bar{q}^0+(2\mathfrak{c}^2+|\bar{p}|^2+|\bar{q}|^2))}\nonumber\\
	&=\frac{-4\mathfrak{c}\bar{p}^0\bar{q}^0(|\bar{p}-\bar{q}|^2-g^2)+4\mathfrak{c}|\bar{p}-\bar{q}|^2(\bar{p}^0\bar{q}^0-\mathfrak{c}^2)}{(2\mathfrak{c}g^2+\sqrt{s}|\bar{p}-\bar{q}|g)(2\bar{p}^0\bar{q}^0+(2\mathfrak{c}^2+|\bar{p}|^2+|\bar{q}|^2))}\nonumber\\
	&\lesssim O(\mathfrak{c}^{-\frac{5}{4}}),\label{2.135}\\
	\mathcal{G}_{113}&=\frac{2\sqrt{s}|\bar{p}-\bar{q}|g\mathfrak{c}^2-4\mathfrak{c}^3|\bar{p}-\bar{q}|^2}{(2\mathfrak{c}g^2+\sqrt{s}|\bar{p}-\bar{q}|g)(2\bar{p}^0\bar{q}^0+(2\mathfrak{c}^2+|\bar{p}|^2+|\bar{q}|^2))}\nonumber\\
	&=\frac{-2\mathfrak{c}^2|\bar{p}-\bar{q}|}{(2\mathfrak{c}g^2+\sqrt{s}|\bar{p}-\bar{q}|g)(2\bar{p}^0\bar{q}^0+(2\mathfrak{c}^2+|\bar{p}|^2+|\bar{q}|^2))}(2\mathfrak{c}|\bar{p}-\bar{q}|-\sqrt{s}g)\nonumber\\
	&=\frac{-2\mathfrak{c}^2|\bar{p}-\bar{q}|}{(2\mathfrak{c}g^2+\sqrt{s}|\bar{p}-\bar{q}|g)(2\bar{p}^0\bar{q}^0+(2\mathfrak{c}^2+|\bar{p}|^2+|\bar{q}|^2))}\frac{4\mathfrak{c}^2(|\bar{p}-\bar{q}|^2-g^2)-g^4}{2\mathfrak{c}|\bar{p}-\bar{q}|+\sqrt{s}g}\nonumber\\
	&\lesssim O(\mathfrak{c}^{-\frac{5}{4}}),\label{2.136}\\
	\mathcal{G}_{114}&=\frac{2\sqrt{s}\bar{p}^0\bar{q}^0|\bar{p}-\bar{q}|g-4\mathfrak{c}^3|\bar{p}-\bar{q}|^2}{(2\mathfrak{c}g^2+\sqrt{s}|\bar{p}-\bar{q}|g)(2\bar{p}^0\bar{q}^0+(2\mathfrak{c}^2+|\bar{p}|^2+|\bar{q}|^2))}\nonumber\\
	&=\frac{-2|\bar{p}-\bar{q}|}{(2\mathfrak{c}g^2+\sqrt{s}|\bar{p}-\bar{q}|g)(2\bar{p}^0\bar{q}^0+(2\mathfrak{c}^2+|\bar{p}|^2+|\bar{q}|^2))}(2\mathfrak{c}^3|\bar{p}-\bar{q}|-\sqrt{s}\bar{p}^0\bar{q}^0g)\nonumber\\
	&=\frac{-2|\bar{p}-\bar{q}|}{(2\mathfrak{c}g^2+\sqrt{s}|\bar{p}-\bar{q}|g)(2\bar{p}^0\bar{q}^0+(2\mathfrak{c}^2+|\bar{p}|^2+|\bar{q}|^2))}\frac{4\mathfrak{c}^6|\bar{p}-\bar{q}|^2-s(\bar{p}^0)^2(\bar{q}^0)^2g^2}{2\mathfrak{c}^3|\bar{p}-\bar{q}|+\sqrt{s}\bar{p}^0\bar{q}^0g}\nonumber\\
	&=\frac{-2|\bar{p}-\bar{q}|}{(2\mathfrak{c}g^2+\sqrt{s}|\bar{p}-\bar{q}|g)(2\bar{p}^0\bar{q}^0+(2\mathfrak{c}^2+|\bar{p}|^2+|\bar{q}|^2))}\nonumber\\
	&\qquad \times \Big\{\frac{4\mathfrak{c}^6(|\bar{p}-\bar{q}|^2-g^2)-g^4(|\bar{p}|^2+\mathfrak{c}^2)(|\bar{q}|^2+\mathfrak{c}^2)}{2\mathfrak{c}^3|\bar{p}-\bar{q}|+\sqrt{s}\bar{p}^0\bar{q}^0g}-\frac{4\mathfrak{c}^2g^2(|\bar{p}|^2\mathfrak{c}^2+|\bar{q}|^2\mathfrak{c}^2+|\bar{p}|^2|\bar{q}|^2)}{2\mathfrak{c}^3|\bar{p}-\bar{q}|+\sqrt{s}\bar{p}^0\bar{q}^0g}\Big\} \nonumber\\
	&\lesssim O(\mathfrak{c}^{-\frac{5}{4}}) \label{2.137}\\
\intertext{and}
	\mathcal{G}_{115}&=\frac{(2\mathfrak{c}g^2+\sqrt{s}|\bar{p}-\bar{q}|g)(|\bar{p}|^2+|\bar{q}|^2)}{(2\mathfrak{c}g^2+\sqrt{s}|\bar{p}-\bar{q}|g)(2\bar{p}^0\bar{q}^0+(2\mathfrak{c}^2+|\bar{p}|^2+|\bar{q}|^2))}\lesssim O(\mathfrak{c}^{-\frac{5}{4}}).\label{2.138}
\end{align}
Combining \eqref{2.133}--\eqref{2.138}, we have
\begin{align*}
	|\mathcal{G}_{11}|\lesssim \mathfrak{c}^{-\frac{1}{2}},
\end{align*}
which, together with \eqref{2.131}--\eqref{2.132}, yields that
\begin{align}\label{2.140}
	\int_{|p-q|\le \mathfrak{c}^{\frac{1}{8}}}|\mathcal{E}_4|dq &\lesssim \mathfrak{c}^{-\frac{1}{2}}\int_{|p-q|\le \mathfrak{c}^{\frac{1}{8}}}\frac{1}{|p-q|} e^{-\frac{|p-q|^2}{8\theta}-\frac{(|p|^2-|q|^2)^2}{8\theta|p-q|^2}}dq
	\lesssim \mathfrak{c}^{-\frac{1}{2}}.
\end{align}
Combining \eqref{2.119}--\eqref{2.76-0}, \eqref{2.81-0} and \eqref{2.140}, one has
\begin{align*}
	\int_{\mathbb{R}^3}|k_{\mathfrak{c}2}(p,q)-k_2(p,q)|dq
	\lesssim \mathfrak{c}^{-\frac{3}{8}}, \quad p\in \mathbb{R}.
\end{align*}
Therefore the proof is completed.
\end{proof}

Denote $\overline{\mathbf{M}}_{\mathfrak{c}}$ as the local Maxwellian in the rest frame where $(u^0,u^1,u^2,u^3)^t=(\mathfrak{c},0,0,0)^t$:
\begin{align*}
	\overline{\mathbf{M}}_{\mathfrak{c}}(t,x,p):=\frac{n_0\gamma}{4\pi \mathfrak{c}^3K_2(\gamma)}\exp\Big\{\frac{-\mathfrak{c}p^0}{T_0}\Big\}.
\end{align*}
Define the third momentum
\begin{align*}
	T^{\alpha \beta \gamma}[\mathbf{M}_{\mathfrak{c}}]:=\int_{\mathbb{R}^3} \frac{p^\alpha p^\beta p^\gamma}{p^0} \mathbf{M}_{\mathfrak{c}} d p,\quad
	\overline{T}^{\alpha \beta \gamma}:=\int_{\mathbb{R}^3} \frac{p^\alpha p^\beta p^\gamma}{p^0} \overline{\mathbf{M}}_{\mathfrak{c}} d p.
\end{align*}
We first give the expression of $\overline{T}^{\alpha \beta \gamma}$ which can be proved directly and we omit the details here for brevity.
\begin{Lemma}\label{lem4.10}
	Let $i, j, k \in\{1,2,3\}$. For the third momentum $\overline{T}^{\alpha \beta \gamma}$ which corresponds to $T^{\alpha \beta \gamma}[\mathbf{M}_{\mathfrak{c}}]$ in the rest frame, there hold
	\begin{align*}
		\overline{T}^{000}&=\frac{n_0\mathfrak{c}^2\left[3 K_3(\gamma)+\gamma K_2(\gamma)\right]}{\gamma K_2(\gamma)}, \\
		\overline{T}^{0 i i}&=\overline{T}^{i i 0}=\overline{T}^{i 0 i}=\frac{n_0\mathfrak{c}^2 K_3(\gamma)}{\gamma K_2(\gamma)}, \\
		\overline{T}^{\alpha \beta \gamma}&=0, \quad \text { if }(\alpha, \beta, \gamma) \neq(0,0,0),(0, i, i),(i, i, 0),(i, 0, i) .
	\end{align*}
\end{Lemma}
Recalling the Lorentz transformation in \eqref{2.4-20} and observing
$$
T^{\alpha \beta \gamma}[\mathbf{M}_{\mathfrak{c}}]=\bar{\Lambda}_{\alpha^{\prime}}^\alpha \bar{\Lambda}_{\beta^{\prime}}^\beta \bar{\Lambda}_{\gamma^{\prime}}^\gamma \overline{T}^{\alpha^{\prime} \beta^{\prime} \gamma^{\prime}},
$$
we can obtain the expression of $T^{\alpha \beta \gamma}[\mathbf{M}_{\mathfrak{c}}]$ from Lemma \ref{lem4.10}.
\begin{Lemma}\label{lem4.11-0}
	For $i, j, k \in\{1,2,3\}$, there hold
	\begin{align*}
		T^{000}[\mathbf{M}_{\mathfrak{c}}] =& \frac{n_0}{\mathfrak{c}\gamma K_2(\gamma)}\left[\left(3 K_3(\gamma)+\gamma K_2(\gamma)\right)\left(u^0\right)^3+3 K_3(\gamma) u^0|u|^2\right],\\
		T^{00 i}[\mathbf{M}_{\mathfrak{c}}] =& \frac{n_0}{\mathfrak{c}\gamma K_2(\gamma)}\left[\left(5 K_3(\gamma)+\gamma K_2(\gamma)\right)\left(u^0\right)^2 u_i+K_3(\gamma)|u|^2 u_i\right], \\
		T^{0 i j}[\mathbf{M}_{\mathfrak{c}}]= & \frac{n_0}{\mathfrak{c}\gamma K_2(\gamma)}\left[\left(6 K_3(\gamma)+\gamma K_2(\gamma)\right) u^0 u_i u_j+\mathfrak{c}^2K_3(\gamma) u^0\delta_{i j}\right], \\
		T^{i j k}[\mathbf{M}_{\mathfrak{c}}]= & \frac{n_0}{\mathfrak{c}\gamma K_2(\gamma)}\Big[(6 K_3(\gamma)+\gamma K_2(\gamma)) u_i u_j u_k+\mathfrak{c}^2K_3(\gamma)\left(u_i \delta_{j k}+u_j \delta_{i k}+u_k \delta_{i j}\right)\Big].
	\end{align*}
\end{Lemma}

Since we have not found a direct reference which gives the orthonormal basis of $\mathcal{N}_{\mathfrak{c}}$, so we present details of calculation in the appendix for completeness though it is somehow routine. Indeed, the orthonormal basis of $\mathcal{N}_{\mathfrak{c}}$ for the relativistic Boltzmann equation has the form
\begin{align}\label{4.116-00}
	\chi^{\mathfrak{c}}_0=\mathfrak{a}_0\mathbf{\sqrt{M_{\mathfrak{c}}}},\quad  \chi^{\mathfrak{c}}_j=\frac{p_j-\mathfrak{a}_j}{\mathfrak{b}_j}\mathbf{\sqrt{M_{\mathfrak{c}}}}\ (j=1,2,3),\quad \chi^{\mathfrak{c}}_4=\frac{p^0/\mathfrak{c}+\sum_{i=1}^3\lambda_ip_i+\mathfrak{e}}{\zeta}\mathbf{\sqrt{M_{\mathfrak{c}}}},  
\end{align}
where $\mathfrak{a}_{\alpha}\ (\alpha=0,1,2,3)$, $\mathfrak{b}_j\ (j=1,2,3)$, $\lambda_i\ (i=1,2,3)$ and $\mathfrak{e}$ are all given in the appendix. 

In the following lemma, we shall show that, as $\mathfrak{c}\rightarrow \infty$,  the relativistic orthonormal basis in \eqref{4.116-00} converges to the following Newtonian orthonormal basis
\begin{align}\label{4.98-00}
	\chi_{0}=\frac{1}{\sqrt{\rho}}\sqrt{\mu},\quad \chi_{j}=\frac{p_j-\mathfrak{u}_j}{\sqrt{\rho \theta}}\sqrt{\mu}\ (j=1,2,3), \quad \chi_{4}=\frac{1}{\sqrt{6\rho}}\Big(\frac{|p-\mathfrak{u}|^2}{\theta}-3\Big)\sqrt{\mu},
\end{align}
where $\mu(t,x,p)$ is defined by \eqref{1.34-20}.

\begin{Lemma}\label{lem2.10-0}
	For any fixed $p\in \mathbb{R}^3$, it holds that
	\begin{align*}
		\lim_{\mathfrak{c}\rightarrow \infty}\chi^{\mathfrak{c}}_{\alpha}=\chi_{\alpha}, \quad \alpha=0,1,\cdots, 4.
	\end{align*}
\end{Lemma}  
\begin{proof}
	In view of Proposition \ref{thm-retoce}, one has
	\begin{align*}
		\lim_{\mathfrak{c}\rightarrow \infty}\mathbf{M}_{\mathfrak{c}}(p)=\mu(p),\quad 
		\lim_{\mathfrak{c}\rightarrow \infty}I^0=\lim_{\mathfrak{c}\rightarrow \infty} \frac{n_0u^0}{\mathfrak{c}}=\rho.
	\end{align*}
	Then we have
	\begin{align*}
		\lim_{\mathfrak{c}\rightarrow \infty} \mathfrak{a}_0=\lim_{\mathfrak{c}\rightarrow \infty} \frac{1}{\sqrt{I^0}}=\frac{1}{\sqrt{\rho}},
	\end{align*}
	which implies that $\lim_{\mathfrak{c}\rightarrow \infty}\chi^{\mathfrak{c}}_{0}=\chi_0$. 
	
	For $j=1,2,3$, a direct calculation shows that
	\begin{align}\label{4.119-00}
		\lim_{\mathfrak{c}\rightarrow \infty}T^{0 j}=\lim_{\mathfrak{c}\rightarrow \infty} \frac{n_0}{\mathfrak{c}}\frac{K_3(\gamma)}{K_2(\gamma)} u^{0}u_j=\rho \mathfrak{u}_j
	\end{align}
	and
	\begin{align*}
		\lim_{\mathfrak{c}\rightarrow \infty}T^{0 j j}&= \lim_{\mathfrak{c}\rightarrow \infty}\frac{n_0}{\mathfrak{c}\gamma K_2(\gamma)}\left[\left(6 K_3(\gamma)+\gamma K_2(\gamma)\right) u^0 u^2_j +\mathfrak{c}^2K_3(\gamma) u^0\right]=\rho \mathfrak{u}^2_j+\rho \theta.
	\end{align*}
	Thus one has
	\begin{align*}
		\lim_{\mathfrak{c}\rightarrow \infty} \mathfrak{a}_j=	\lim_{\mathfrak{c}\rightarrow \infty} \frac{T^{0j}}{I_0}=\mathfrak{u}_j
	\end{align*}
	and 
	\begin{align*}
		\lim_{\mathfrak{c}\rightarrow \infty} \mathfrak{b}_j=	\lim_{\mathfrak{c}\rightarrow \infty} \sqrt{T^{0jj}-\frac{(T^{0j})^2}{I^0}}=\sqrt{\rho \theta},
	\end{align*}
	which implies that $\lim_{\mathfrak{c}\rightarrow \infty}\chi^{\mathfrak{c}}_{j}=\chi_j$, $j=1,2,3$. 
	
	The proof for $\lim_{\mathfrak{c}\rightarrow \infty}\chi^{\mathfrak{c}}_{4}=\chi_4$ is much more complicated. It is clear that
	\begin{align*}
		\chi^{\mathfrak{c}}_4&=\frac{p^0+\mathfrak{c}\mathfrak{e}+\mathfrak{c}\sum_{i=1}^3\lambda_ip_i}{\mathfrak{c}\zeta}\mathbf{\sqrt{M_{\mathfrak{c}}}}\nonumber\\
		&=\frac{(p^0+\mathfrak{c}\mathfrak{e})(p^0-\mathfrak{c}\mathfrak{e})+\mathfrak{c}(p^0-\mathfrak{c}\mathfrak{e})\sum_{i=1}^3\lambda_ip_i}{\mathfrak{c}\zeta (p^0-\mathfrak{c}\mathfrak{e})}\mathbf{\sqrt{M_{\mathfrak{c}}}}\nonumber\\
		&=\frac{\operatorname{Num}}{\operatorname{Den}}\mathbf{\sqrt{M_{\mathfrak{c}}}}.
	\end{align*}
	We first calculate the numerator. 
	Denote $\hat{A}(\gamma):=\frac{K_3(\gamma)}{K_2(\gamma)}-\frac{6}{\gamma}-\frac{K_2(\gamma)}{K_3(\gamma)}$. It follows from Lemma \ref{lem2.1-00} that 
	\begin{align*}
		\hat{A}(\gamma)=-\frac{1}{\gamma}+O(\gamma^{-2}).
	\end{align*}
	Now we have
	\begin{align*}
		1+\mathfrak{e}&=1+\frac{\frac{1}{\gamma}-\frac{(u^0)^2}{\gamma T_0}\frac{K_3(\gamma)}{K_2(\gamma)}-\hat{A}(\gamma)\frac{|u|^2}{\gamma T_0}}{\frac{u^0}{\mathfrak{c}}-\hat{A}(\gamma)\frac{u^0|u|^2}{\mathfrak{c}T_0}}\nonumber\\
		&=\frac{\frac{u^0}{\mathfrak{c}}-\hat{A}(\gamma)\frac{u^0|u|^2}{\mathfrak{c}T_0}+\frac{1}{\gamma}-\frac{(u^0)^2}{\mathfrak{c}^2}\frac{K_3(\gamma)}{K_2(\gamma)}-\hat{A}(\gamma)\frac{|u|^2}{\gamma T_0}}{\frac{u^0}{\mathfrak{c}}-\hat{A}(\gamma)\frac{u^0|u|^2}{\mathfrak{c}T_0}}\nonumber\\
		&=\frac{1-\frac{u^0}{\mathfrak{c}}\frac{K_3(\gamma)}{K_2(\gamma)}-\hat{A}(\gamma)\frac{|u|^2}{T_0}+\frac{\mathfrak{c}}{\gamma u^0}-\hat{A}(\gamma)\frac{\mathfrak{c}|u|^2}{\gamma u^0 T_0}}{1-\hat{A}(\gamma)\frac{|u|^2}{T_0}}
	\end{align*}
	and 
	\begin{align*}
		\lambda_i=\frac{\hat{A}(\gamma)\frac{(u^0)^2}{\mathfrak{c}^2 T_0}u_i}{\frac{u^0}{\mathfrak{c}}-\hat{A}(\gamma)\frac{u^0|u|^2}{\mathfrak{c}T_0}}=\frac{\big(-\frac{1}{\gamma}+O(\gamma^{-2})\big)\frac{(u^0)^2}{\mathfrak{c}^2 T_0}u_i}{\frac{u^0}{\mathfrak{c}}-\hat{A}(\gamma)\frac{u^0|u|^2}{\mathfrak{c}T_0}}, \quad i=1,2,3,
	\end{align*}
	thus we obtain
	\begin{align*}
		&\lim_{\mathfrak{c}\rightarrow \infty} \mathfrak{e}=\lim_{\mathfrak{c}\rightarrow \infty} \frac{\frac{1}{\gamma}-\frac{(u^0)^2}{\mathfrak{c}^2}\frac{K_3(\gamma)}{K_2(\gamma)}-\hat{A}(\gamma)\frac{|u|^2}{\gamma T_0}}{\frac{u^0}{\mathfrak{c}}-\hat{A}(\gamma)\frac{u^0|u|^2}{\mathfrak{c}T_0}}=-1,\\
		&\lim_{\mathfrak{c}\rightarrow \infty} \gamma(1+\mathfrak{e})=-\frac{3}{2}+\frac{|\mathfrak{u}|^2}{2\theta},\\
		&\lim_{\mathfrak{c}\rightarrow \infty} \gamma \lambda_i=
		-\frac{\mathfrak{u}_i}{\theta}, \quad i=1,2,3,
	\end{align*}
	where we used the fact that
	\begin{align*}
		\lim_{\mathfrak{c}\rightarrow \infty} \gamma\Big(1-\frac{u^0}{\mathfrak{c}}\frac{K_3(\gamma)}{K_2(\gamma)}\Big)=\lim_{\mathfrak{c}\rightarrow \infty}\gamma\Big[-\frac{u^0}{\mathfrak{c}}\Big(\frac{K_3(\gamma)}{K_2(\gamma)}-1\Big)-\Big(\frac{u^0}{\mathfrak{c}}-1\Big)\Big]=-\frac{5}{2}-\frac{|\mathfrak{u}|^2}{2\theta}.
	\end{align*}
	Hence we get
	\begin{align}\label{5.18-0}
		\lim_{\mathfrak{c}\rightarrow \infty}\operatorname{Num}&=\lim_{\mathfrak{c}\rightarrow \infty}\Big(|p|^2+\gamma T_0(1-\mathfrak{e})(1+\mathfrak{e})+\gamma T_0\Big(\frac{p^0}{\mathfrak{c}}-\mathfrak{e}\Big)\sum_{i=1}^3\lambda_ip_i\Big)\nonumber\\
		&=|p|^2-3\theta+|\mathfrak{u}|^2-2\sum_{i=1}^3\mathfrak{u}_ip_i=|p-\mathfrak{u}|^2-3\theta.
	\end{align}

	We next consider the denominator. Notice that
	\begin{align*}
		\operatorname{Den}=\mathfrak{c} \zeta (p^0-\mathfrak{c}\mathfrak{e})=\mathfrak{c}^2\zeta   \Big(\frac{p^0}{\mathfrak{c}}-\mathfrak{e}\Big)
	\end{align*}
	and 
	\begin{align*}
		\lim_{\mathfrak{c}\rightarrow \infty}   \Big(\frac{p^0}{\mathfrak{c}}-\mathfrak{e}\Big)=2,
	\end{align*}
	then we focus on the quantity $\mathfrak{c}^2 \zeta=T_0\sqrt{\gamma^2\zeta^2}$. By the expression of $\zeta$ in the appendix , one has
	\begin{align}\label{5.21-0}
		\zeta^2&=\Big(\sum_{i,j=1}^3\lambda_i\lambda_jT^{0ij}\Big)+\Big(2\sum_{i=1}^3\lambda_i\mathfrak{e} T^{0i}+2\sum_{i=1}^3\frac{\lambda_i}{\mathfrak{c}}T^{00i}\Big)+\Big(\frac{T^{000}}{\mathfrak{c}^2}+\mathfrak{e}^2I^0+2\frac{\mathfrak{e}}{\mathfrak{c}}T^{00}\Big)\nonumber\\
		&:=\mathcal{I}_1+\mathcal{I}_2+\mathcal{I}_3.
	\end{align}
	It is easy to see that
	\begin{align*}
		\lim_{\mathfrak{c}\rightarrow \infty}T^{0 i j}&= \lim_{\mathfrak{c}\rightarrow \infty}\frac{n_0}{\mathfrak{c}\gamma K_2(\gamma)}\left[\left(6 K_3(\gamma)+\gamma K_2(\gamma)\right) u^0 u_i u_j +\mathfrak{c}^2K_3(\gamma) u^0\delta_{ij}\right]=\rho \mathfrak{u}_i \mathfrak{u}_j+\rho \theta \delta_{ij},
	\end{align*}
	which yields that
	\begin{align}\label{5.23-0}
		\lim_{\mathfrak{c}\rightarrow \infty}\gamma^2 \mathcal{I}_1&=\lim_{\mathfrak{c}\rightarrow \infty} \sum_{i,j=1}^3(\gamma\lambda_i)\cdot (\gamma\lambda_j) T^{0ij}\nonumber\\
		&=\sum_{i,j=1}^3\Big(-\frac{\mathfrak{u}_i}{\theta}\Big)\Big(-\frac{\mathfrak{u}_j}{\theta}\Big)(\rho \mathfrak{u}_i \mathfrak{u}_j+\rho \theta \delta_{ij})\nonumber\\
		&=\frac{\rho |\mathfrak{u}|^4}{\theta^2}+\frac{\rho |\mathfrak{u}|^2}{\theta}.
	\end{align}
	We notice that
	\begin{align*}
		\mathfrak{e} T^{0i}+\frac{T^{00i}}{\mathfrak{c}}=(\mathfrak{e}+1) T^{0i}+\Big(\frac{T^{00i}}{\mathfrak{c}}-T^{0i}\Big)
	\end{align*}
	and 
	\begin{align}\label{4.138-00}
		\frac{T^{00i}}{\mathfrak{c}}-T^{0i}
		&=\frac{n_0}{\mathfrak{c}^2\gamma K_2(\gamma)}\left[\left(5 K_3(\gamma)+\gamma K_2(\gamma)\right)\left(u^0\right)^2 u_i+K_3(\gamma)|u|^2 u_i\right]-\frac{n_0}{\mathfrak{c}}\frac{K_3(\gamma)}{K_2(\gamma)} u^{0}u_i\nonumber\\
		&=n_0u_i\Big\{\frac{5}{\gamma}\frac{K_3(\gamma)}{K_2(\gamma)}\Big(\frac{u^0}{\mathfrak{c}}\Big)^2+\frac{|u|^2}{\mathfrak{c}\gamma}\frac{K_3(\gamma)}{K_2(\gamma)}+\frac{u^{0}}{\mathfrak{c}}\Big(\frac{u^0}{\mathfrak{c}}-1\Big)+\frac{u^{0}}{\mathfrak{c}}\Big(1-\frac{K_3(\gamma)}{K_2(\gamma)}\Big) \Big\}.
	\end{align}
	Then it follows from \eqref{4.119-00} and \eqref{4.138-00} that 
	\begin{align*}
		\lim_{\mathfrak{c}\rightarrow \infty}\gamma (\mathfrak{e}+1) T^{0i}=\rho \mathfrak{u}_i\Big(-\frac{3}{2}+\frac{|\mathfrak{u}|^2}{2\theta}\Big)
	\end{align*}
	and 
	\begin{align*}
		\lim_{\mathfrak{c}\rightarrow \infty}\gamma\Big(\frac{T^{00i}}{\mathfrak{c}}-T^{0i}\Big)=\rho \mathfrak{u}_i\Big(5+0+\frac{|\mathfrak{u}|^2}{2\theta}-\frac{5}{2}\Big)=\rho \mathfrak{u}_i\Big(\frac{5}{2}+\frac{|\mathfrak{u}|^2}{2\theta}\Big).
	\end{align*}
	Hence one obtains
	\begin{align}\label{5.28-0}
		\lim_{\mathfrak{c}\rightarrow \infty}\gamma^2 \mathcal{I}_2&=2\lim_{\mathfrak{c}\rightarrow \infty} \sum_{i=1}^3 (\gamma \lambda_i)\cdot \gamma \Big(\frac{T^{00i}}{\mathfrak{c}}-T^{0i}\Big)+2\lim_{\mathfrak{c}\rightarrow \infty} \sum_{i=1}^3 (\gamma \lambda_i)\cdot \gamma (\mathfrak{e}+1) T^{0i}\nonumber\\
		&=2\sum_{i=1}^3 \Big(-\frac{\mathfrak{u}_i}{\theta}\Big)\cdot \Big[\rho \mathfrak{u}_i\Big(\frac{5}{2}+\frac{|\mathfrak{u}|^2}{2\theta}\Big)+\rho \mathfrak{u}_i\Big(-\frac{3}{2}+\frac{|\mathfrak{u}|^2}{2\theta}\Big)\Big]\nonumber\\
		&=-2\frac{\rho |\mathfrak{u}|^4}{\theta^2}-2\frac{\rho |\mathfrak{u}|^2}{\theta}.
	\end{align}

	We finally consider $\gamma^2\mathcal{I}_3$. It holds that
	\begin{align*}
		\frac{\mathcal{I}_3}{n_0}&=\frac{1}{\mathfrak{c}^2}\frac{T^{000}}{n_0}+\mathfrak{e}^2\frac{I^0}{n_0}+2\frac{\mathfrak{e}}{\mathfrak{c}}\frac{T^{00}}{n_0}\nonumber\\
		&=\frac{1}{\mathfrak{c}^3\gamma K_2(\gamma)}\left[\left(3 K_3(\gamma)+\gamma K_2(\gamma)\right)\left(u^0\right)^3+3 K_3(\gamma) u^0|u|^2\right]+\mathfrak{e}^2\frac{u^0}{\mathfrak{c}}\nonumber\\
		&\qquad +2\mathfrak{e}\Big(\frac{1}{\mathfrak{c}^2}\frac{K_3(\gamma)}{K_2(\gamma)} (u^{0})^2-\frac{1}{\gamma}\Big)\nonumber\\
		&=\frac{3}{\gamma}\frac{K_3(\gamma)}{K_2(\gamma)}\Big(\frac{u^0}{\mathfrak{c}}\Big)^3+\Big(\frac{u^0}{\mathfrak{c}}\Big)^3+\frac{3}{\gamma}\frac{K_3(\gamma)}{K_2(\gamma)}\frac{u^0|u|^2}{\mathfrak{c}^3}+\mathfrak{e}^2\frac{u^0}{\mathfrak{c}}+2\mathfrak{e}\frac{K_3(\gamma)}{K_2(\gamma)} \Big(\frac{u^0}{\mathfrak{c}}\Big)^2 -\mathfrak{e}\frac{2}{\gamma}\nonumber\\
		&=\frac{3}{\gamma}\frac{K_3(\gamma)}{K_2(\gamma)}\frac{u^0|u|^2}{\mathfrak{c}^3}+\frac{3}{\gamma}\frac{K_3(\gamma)}{K_2(\gamma)}\Big(\frac{u^0}{\mathfrak{c}}\Big)^2 \Big(\frac{u^0}{\mathfrak{c}}-1\Big)+\mathfrak{e}\Big(2\frac{u^0}{\mathfrak{c}}+2\Big)\Big(\frac{K_3(\gamma)}{K_2(\gamma)}-1\Big)\Big(\frac{u^0}{\mathfrak{c}}-1\Big)\nonumber\\
		&\qquad +2\mathfrak{e}\Big(\frac{K_3(\gamma)}{K_2(\gamma)}-1-\frac{5}{2\gamma}\Big)+\frac{u^0}{\mathfrak{c}}\Big(\frac{u^0}{\mathfrak{c}}-1\Big)^2+2\frac{u^0}{\mathfrak{c}}(1+\mathfrak{e})\Big(\frac{u^0}{\mathfrak{c}}-1\Big)\nonumber\\
		&\qquad +\frac{u^0}{\mathfrak{c}}(1+\mathfrak{e})^2+\frac{3}{\gamma}\Big[\Big(\frac{K_3(\gamma)}{K_2(\gamma)}-1\Big)\Big(\frac{u^0}{\mathfrak{c}}\Big)^2+\Big(\frac{u^0}{\mathfrak{c}}+1\Big)\Big(\frac{u^0}{\mathfrak{c}}-1\Big)+(1+\mathfrak{e})\Big].
	\end{align*}
	Thus we have
	\begin{align}\label{5.30-0}
		\lim_{\mathfrak{c}\rightarrow \infty}\gamma^2 \mathcal{I}_3=\frac{3}{2}\rho+\frac{\rho |\mathfrak{u}|^4}{\theta^2}+\frac{\rho |\mathfrak{u}|^2}{\theta}.
	\end{align}
	Combining \eqref{5.21-0}, \eqref{5.23-0}, \eqref{5.28-0} and \eqref{5.30-0}, we finally obtain
	\begin{align*}
		\lim_{\mathfrak{c}\rightarrow \infty}\gamma^2 \zeta^2=\frac{3}{2}\rho.
	\end{align*}
	Hence one obtains
	\begin{align*}
		\lim_{\mathfrak{c}\rightarrow \infty}\operatorname{Den}=\theta\sqrt{6\rho},
	\end{align*}
	which, together with \eqref{5.18-0}, yields that
	\begin{align*}
		\lim_{\mathfrak{c}\rightarrow \infty}\chi^{\mathfrak{c}}_{4}=\frac{|p-\mathfrak{u}|^2-3\theta}{\theta\sqrt{6\rho}}=\chi_4.
	\end{align*}
	Therefore the proof is completed.
\end{proof}

With above preparations, we shall prove the coercivity estimate for the linear operator $\mathbf{L}_{\mathfrak{c}}$.

\begin{Proposition}[Uniform coercivity estimate on $\mathbf{L}_{\mathfrak{c}}$] \label{prop4.13}
There exists a positive constant $\zeta_0>0$, which is independent of $\mathfrak{c}$, such that
\begin{align*}
	\langle \mathbf{L}_{\mathfrak{c}}g, g\rangle \geq \zeta_0\|\{\mathbf{I-P_{\mathfrak{c}}}\}g\|_{\nu_{\mathfrak{c}}}^{2}
\end{align*}
for any $g\in L_{\nu}^2(\mathbb{R}^3)$.
\end{Proposition}
\begin{proof}
It is clear that one only needs to show that there is a positive constant $\zeta_0>0$, which is independent of $\mathfrak{c}$, such that
\begin{align}\label{2.161}
	\langle \mathbf{L}_{\mathfrak{c}}g, g\rangle \geq \zeta_0\|g\|_{\nu_{\mathfrak{c}}}^{2}=\zeta_0
\end{align}
holds for any $\mathfrak{c}$ and any $g\in \mathcal{N}_{\mathfrak{c}}^{\perp}$ with $\|g\|_{\nu_{\mathfrak{c}}}=1$.

For any given $\mathfrak{c}$, the linearized Boltzmann collision operator $\mathbf{L}_{\mathfrak{c}}$ satisfies the well-known hypocoercivity (see \cite{Golse} for instance), i.e., there exists a positive constant $\alpha_{\mathfrak{c}}>0$, such that
\begin{align}\label{2.162}
	\langle \mathbf{L}_{\mathfrak{c}}g, g\rangle \geq \alpha_{\mathfrak{c}}\|g\|_{\nu_{\mathfrak{c}}}^{2}=\alpha_{\mathfrak{c}}
\end{align}
for any $g\in \mathcal{N}_{\mathfrak{c}}^{\perp}$ with $\|g\|_{\nu_{\mathfrak{c}}}=1$. Denote
\begin{align}\label{2.163}
	\zeta_{\mathfrak{c}}:=\inf_{\substack{g\in \mathcal{N}_{\mathfrak{c}}^{\perp}\\ \|g\|_{\nu_{\mathfrak{c}}}=1}}\langle \mathbf{L}_{\mathfrak{c}}g, g\rangle.
\end{align}
It follows from \eqref{2.162} that $\zeta_{\mathfrak{c}}\ge \alpha_{\mathfrak{c}}>0$ for any $\mathfrak{c}$. To prove \eqref{2.161}, it suffices to show that
\begin{align}\label{2.164}
	\inf_{\mathfrak{c}\ge 1}\zeta_{\mathfrak{c}}>0.
\end{align}
We prove \eqref{2.164} by contradiction.  Assume that \eqref{2.164} is not true, then there exists a sequence $\{\zeta_{\mathfrak{c}_n}\}$ such that
\begin{align}\label{2.165}
	\lim_{n \rightarrow \infty}\mathfrak{c}_n=\infty \quad \text{and}\quad \lim_{n\rightarrow \infty}\zeta_{\mathfrak{c}_n}=0.
\end{align}
For each $n$, owing to \eqref{2.163}, there exists $g_n\in \mathcal{N}^{\perp}_{\mathfrak{c}_n}$ with $\|g_n\|_{\nu_{\mathfrak{c}_n}}=1$, so that 		
\begin{align*}		
	\zeta_{\mathfrak{c}_n}\le \langle \mathbf{L}_{\mathfrak{c}_n}g_{n}, g_{n}\rangle<\zeta_{\mathfrak{c}_n}+\frac{1}{n},		
\end{align*}
which, together with \eqref{2.165}, yields that
\begin{align}\label{2.167}
	\lim_{n\rightarrow \infty}\langle \mathbf{L}_{\mathfrak{c}_n}g_{n}, g_{n}\rangle=0.
\end{align}
It is clear that $\{g_n\}^{\infty}_{n=1}$ is a bounded sequence in $L^2(\mathbb{R}^3)$. Since $L^{2}$ is a Hilbert space, based on the Eberlein-Smulian theorem, we have the weakly convergent sequence (up to extracting a subsequence with an abuse of notation) $g_{n} \rightharpoonup g$ in $L^{2}$. Moreover, for any fixed $N\ge 1$, one has
\begin{align*}
	\chi_{\{|p|\le N\}}\sqrt{\nu_{\mathfrak{c}_n}}g_n \rightharpoonup \chi_{\{|p|\le N\}}\sqrt{\nu}g \quad \text{in} \ L^2,
\end{align*}
where $\nu(p)=\lim_{\mathfrak{c}\rightarrow\infty} \nu_{\mathfrak{c}}(p)$. Hence, by the weak semi-continuity, for any fixed $N$, we have
\begin{align*}
	\|\chi_{\{|p|\le N\}}\sqrt{\nu}g\|_{2}\le \liminf\limits_{n\rightarrow \infty} \|\chi_{\{|p|\le N\}}\sqrt{\nu_{\mathfrak{c}_n}}g_n\|_{2}\le 1,
\end{align*}
which implies that
\begin{align}\label{2.169}
	\|\sqrt{\nu}g\|_{2}\le 1.
\end{align}
For later use, we denote
\begin{align*}
	\mathbf{L}f:=\nu f-\mathbf{K}f,
\end{align*}
where 
\begin{align}\label{4.93-20}
	\mathbf{K}f:=\int_{\mathbb{R}^3}k(p,q)f(q)dq=\int_{\mathbb{R}^3}[k_2(p,q)-k_1(p,q)]f(q)dq
\end{align}
with $k_1(p,q)$ and $k_2(p,q)$ defined in \eqref{2.97}--\eqref{2.98}. We also denote $\mathcal{N}$ as the null space of $\mathbf{L}$, that is,  $\mathcal{N}:=\operatorname{span}\{\chi_0, \chi_1, \chi_2, \chi_3, \chi_4\}$.
Clearly, we have
\begin{align}\label{2.170}
	0\le \left\langle \mathbf{L}_{\mathfrak{c}_{n}} g_{n}, g_{n}\right\rangle&=\left\|g_{n}\right\|_{\nu_{\mathfrak{c}_n}}^{2}-\left\langle (\mathbf{K}_{\mathfrak{c}_{n}}-\mathbf{K})g_{n}, g_{n}\right\rangle-\left\langle \mathbf{K}g_{n}, g_{n}\right\rangle\nonumber\\
	&=1-\left\langle (\mathbf{K}_{\mathfrak{c}_{n}}-\mathbf{K})g_{n}, g_{n}\right\rangle-\left\langle \mathbf{K}g_{n}, g_{n}\right\rangle.
\end{align}
Since $\mathbf{K}$ is a compact operator on $L^{2}$, it holds that
\begin{align*}
	\lim_{n \rightarrow \infty} \|\mathbf{K}g_{n}-\mathbf{K}g\|_{2}=0.
\end{align*}
Hence we have
\begin{align*}
	\langle \mathbf{K}g_{n},g_{n}\rangle-\langle \mathbf{K}g, g\rangle=\langle \mathbf{K}g_n-\mathbf{K}g, g_{n}\rangle+\langle \mathbf{K}g, g_{n}-g\rangle\rightarrow 0, \quad n\rightarrow \infty.
\end{align*}
It follows from Lemmas \ref{lem2.13}--\ref{lem2.14} that
\begin{align}\label{2.172}
	\left\langle (\mathbf{K}_{\mathfrak{c}_{n}}-\mathbf{K})g_{n},g_{n}\right\rangle \rightarrow 0, \quad n\rightarrow \infty.
\end{align}
Combining \eqref{2.167}, \eqref{2.170}-\eqref{2.172}, we have
\begin{align}\label{2.173}
	\langle \mathbf{K}g, g\rangle=1,
\end{align}
which, together with \eqref{2.169}, yields that
\begin{align*}
	0\le \langle\mathbf{L}g,g\rangle=\|g\|^2_{\nu}-\langle \mathbf{K}g, g\rangle\le 0.
\end{align*}
Thus we have $g\in \mathcal{N}$.

Next, we shall show that $g\in \mathcal{N}^{\perp}$. Recall $\chi^{\mathfrak{c}_n}_{\alpha}$, $\chi_{\alpha}$ defined in \eqref{4.116-00} (with $\mathfrak{c}$ replaced by $\mathfrak{c}_n$) and \eqref{4.98-00}. Notice that
\begin{align}\label{2.174}
	0=\langle g_n,\chi^{\mathfrak{c}_n}_{\alpha}\rangle=\langle g_n-g, \chi^{\mathfrak{c}_n}_{\alpha}-\chi_{\alpha}\rangle+\langle g_n-g, \chi_{\alpha}\rangle+\langle g,\chi^{\mathfrak{c}_n}_{\alpha}-\chi_{\alpha}\rangle+\langle g, \chi_{\alpha}\rangle, \quad \alpha=0,1,\cdots,4.
\end{align}
Using Lemma \ref{lem2.10-0} and $g_n \rightharpoonup g$ in $L^2$, we take the limit $n\rightarrow \infty$ in \eqref{2.174} to obtain
\begin{align*}
	\langle g,\chi_{\alpha}\rangle=0, \quad \alpha=0,1,\cdots,4,
\end{align*}
which implies that $g\in \mathcal{N}^{\perp}$. Since we also have $g\in \mathcal{N}$, one concludes that $g=0$, which contradicts with \eqref{2.173}.
Therefore the proof of Proposition \ref{prop4.13} is completed.
\end{proof}

\subsection{Uniform estimate on $\mathbf{L}_{\mathfrak{c}}^{-1}$}\label{sec4.2}
To apply the Hilbert expansion procedure, we need  uniform-in-$\mathfrak{c}$ estimate on $\mathbf{L}_{\mathfrak{c}}^{-1}$. The proof is inspired by \cite{Jiang}.
\begin{Lemma}\label{lem5.1}
	For any fixed $0\le \lambda<1$, it holds that 
	\begin{align*}
		\mathbf{M}_{\mathfrak{c}}^{-\frac{\lambda}{2}}(p)|\mathbf{K}_{\mathfrak{c}}f(p)|\lesssim \|\mathbf{M}_{\mathfrak{c}}^{-\frac{\lambda}{2}}f\|_{2}, \quad p\in \mathbb{R}^3.
	\end{align*}
\end{Lemma}
\begin{proof}
	It follows from \eqref{1.34-0} and Lemma \ref{lem4.3-00} that
	\begin{align*}
		\mathbf{M}_{\mathfrak{c}}^{-\frac{\lambda}{2}}(p)|\mathbf{K}_{\mathfrak{c}1}f(p)|
		&\lesssim \int_{\mathbb{R}^3} |p-q|\mathbf{M}_{\mathfrak{c}}^{\frac{1-\lambda}{2}}(p) \mathbf{M}_{\mathfrak{c}}^{\frac{1}{2}}(q)|f(q)|dq\nonumber\\
		&\lesssim \int_{\mathbb{R}^3} |p-q|e^{-(1-\lambda)\bar{c}_1|p|}e^{-\bar{c}_1|q|}|f(q)|dq\nonumber\\
		&\lesssim \int_{\mathbb{R}^3} e^{-\frac{\bar{c}_1}{2}|q|}|f(q)|dq \lesssim \|f\|_{2}.
	\end{align*}
	Using \eqref{2.6-10}, one has
	\begin{align}\label{2.31-0}
		&\mathbf{M}_{\mathfrak{c}}^{-\frac{\lambda}{2}}(p)|\mathbf{K}_{\mathfrak{c}2}f(p)|\nonumber\\
		&\le \frac{\mathfrak{c}}{p^{0}} \int_{\mathbb{R}^{3}} \frac{d q}{q^{0}} \int_{\mathbb{R}^{3}} \frac{d q^{\prime}}{q^{\prime 0}} \int_{\mathbb{R}^{3}} \frac{d p^{\prime}}{p^{\prime 0}} W\left(p, q \mid p^{\prime}, q^{\prime}\right) \mathbf{M}_{\mathfrak{c}}^{\frac{1+\lambda}{2}}(q)\mathbf{M}_{\mathfrak{c}}^{\frac{1-\lambda}{2}}(q') |\mathbf{M}_{\mathfrak{c}}^{-\frac{\lambda}{2}}(p')f(p')|\nonumber\\
		&=\frac{\mathfrak{c}}{p^{0}} \int_{\mathbb{R}^{3}} \frac{d q}{q^{0}} \int_{\mathbb{R}^{3}} \frac{d q^{\prime}}{q^{\prime 0}} \int_{\mathbb{R}^{3}} \frac{d p^{\prime}}{p^{\prime 0}} \bar{s}  \delta^{(4)}\left(p^{\mu}+p^{\prime \mu}-q^{\mu}-q^{\prime \mu}\right) \mathbf{M}_{\mathfrak{c}}^{\frac{1+\lambda}{2}}(p')\mathbf{M}_{\mathfrak{c}}^{\frac{1-\lambda}{2}}(q') |\mathbf{M}_{\mathfrak{c}}^{-\frac{\lambda}{2}}(q)f(q)|\nonumber\\
		&\lesssim \int_{\mathbb{R}^{3}} \xi(p,q)|\mathbf{M}_{\mathfrak{c}}^{-\frac{\lambda}{2}}(q)f(q)|d q,
	\end{align}
	where we exchanged $p'$ and $q$ in the last second step with 
	\begin{align*}
		\xi(p,q):=\frac{\mathfrak{c}}{p^{0}q^{0}} \int_{\mathbb{R}^{3}} \frac{d q^{\prime}}{q^{\prime 0}} \int_{\mathbb{R}^{3}} \frac{d p^{\prime}}{p^{\prime 0}} \bar{s}\delta^{(4)}\left(p^{\mu}+p^{\prime \mu}-q^{\mu}-q^{\prime \mu}\right) \mathbf{M}_{\mathfrak{c}}^{\frac{1-\lambda}{2}}(p')\mathbf{M}_{\mathfrak{c}}^{\frac{1-\lambda}{2}}(q')
	\end{align*}
	and 
	\begin{equation*}
		\bar{g}^{2}=g^{2}+\frac{1}{2}(p^{\mu}+q^{\mu}) \cdot\left(p^{\prime \mu}+q^{\prime \mu}-p^{\mu}-q^{\mu}\right),\quad \bar{s}=\bar{g}^{2}+4 \mathfrak{c}^{2}.
	\end{equation*}
	Applying Lorentz transformation for $\xi(p,q)$, one has
	\begin{align*}
		\xi(p,q)=\frac{\mathfrak{c}c_0^{1-\lambda}}{p^{0}q^{0}} \int_{\mathbb{R}^{3}} \frac{d q^{\prime}}{q^{\prime 0}} \int_{\mathbb{R}^{3}} \frac{d p^{\prime}}{p^{\prime 0}} s(\bar{p},p')\delta^{(4)}\left(\bar{p}^{\mu}+p^{\prime \mu}-\bar{q}^{\mu}-q^{\prime \mu}\right) e^{-(1-\lambda)\frac{\mathfrak{c}(p^{\prime 0}+q^{\prime 0})}{2T_0}},
	\end{align*}
	where $s(\bar{p},p')=-(\bar{p}^{\mu}+p^{\prime \mu})(\bar{p}_{\mu}+p'_{\mu})$.
	By similar arguments as in \cite{Strain}, one can show that
	\begin{align}\label{4.154-10}
		\xi(p,q)&=\frac{\mathfrak{c}c_0^{1-\lambda}\pi s^{3 / 2}}{4g p^{0} q^{0}} \int_{0}^{\infty} \frac{y\left(1+\sqrt{y^{2}+1}\right)}{\sqrt{y^{2}+1}}  e^{-\frac{1-\lambda}{2T_0}\mathfrak{c}(\bar{p}^{0}+\bar{q}^{0}) \sqrt{y^{2}+1} }I_{0}\left(\frac{(1-\lambda)\mathfrak{c}|\bar{p} \times \bar{q}|}{gT_0} y\right)d y\nonumber\\
		&=\frac{\mathfrak{c}c_0^{1-\lambda}\pi s^{3 / 2}}{4g p^{0} q^{0}} \int_{0}^{\infty} \frac{y\left(1+\sqrt{y^{2}+1}\right)}{\sqrt{y^{2}+1}}  e^{-\tilde{\boldsymbol{\ell}}\sqrt{y^{2}+1} }I_{0}\left(\tilde{\boldsymbol{j}}y\right)d y,
	\end{align}
	where 
	\begin{align*} 
		c_0=\frac{n_0}{4\pi \mathfrak{c}T_0K_2(\gamma)},\quad \tilde{\boldsymbol{\ell}}=(1-\lambda)\bar{\boldsymbol{\ell}},\quad  \tilde{\boldsymbol{j}}=(1-\lambda)\bar{\boldsymbol{j}},\quad 
		\bar{\boldsymbol{\ell}}=\mathfrak{c}\frac{\bar{p}^0+\bar{q}^0}{2T_0}, \quad
		\bar{\boldsymbol{j}}=\mathfrak{c}\frac{|\bar{p}\times \bar{q}|}{gT_0}.
	\end{align*}
    In view of \eqref{2.13-10}--\eqref{2.12}, we can rewrite \eqref{4.154-10} as 
	\begin{align*}
		\xi(p,q)&=\frac{\mathfrak{c}c_0^{1-\lambda}\pi s^{3 / 2}}{4g p^{0} q^{0}}[J_1(\tilde{\boldsymbol{\ell}},\tilde{\boldsymbol{j}})+J_2(\tilde{\boldsymbol{\ell}},\tilde{\boldsymbol{j}})].
	\end{align*}
    By similar arguments as in Lemma \ref{lem4.3-00}, one can prove
	\begin{align}\label{3.24}
		\xi(p,q)\lesssim  \Big[\frac{1}{\mathfrak{c}}+\frac{1}{|p-q|}\Big]e^{-(1-\lambda)\bar{c}_1|p-q|},
	\end{align}
    which yields that
	\begin{align}\label{2.33-0}
		\int_{\mathbb{R}^3}\xi^2(p,q)dq\lesssim \int_{\mathbb{R}^3}  \Big(\frac{1}{\mathfrak{c}^2}+\frac{1}{|p-q|^2}\Big)e^{-2(1-\lambda)\bar{c}_1|p-q|}dq<C<\infty,
	\end{align}
	where $C$ is a positive constant independent of $\mathfrak{c}$. Hence it follows from \eqref{2.31-0} and \eqref{2.33-0} that 
	\begin{align*}
		\mathbf{M}_{\mathfrak{c}}^{-\frac{\lambda}{2}}(p)|\mathbf{K}_{\mathfrak{c}2}f(p)|&\lesssim \Big(\int_{\mathbb{R}^3}\xi^2(p,q)dq\Big)^{\frac{1}{2}}\cdot \Big(\int_{\mathbb{R}^3}|\mathbf{M}_{\mathfrak{c}}^{-\frac{\lambda}{2}}(q)f(q)|^2dq\Big)^{\frac{1}{2}}\nonumber\\
		&\lesssim \|\mathbf{M}_{\mathfrak{c}}^{-\frac{\lambda}{2}}f\|_{2}.
	\end{align*}
	Therefore the proof of Lemma \ref{lem5.1} is completed.
\end{proof}

\begin{Lemma}\label{lem5.2}
	For any fixed $0\le \lambda<1$, it holds that 
	\begin{align*} 
		\left|\langle \mathbf{M}_{\mathfrak{c}}^{-\frac{\lambda}{2}}\mathbf{K}_{\mathfrak{c}1}f,\mathbf{M}_{\mathfrak{c}}^{-\frac{\lambda}{2}} f\rangle \right|+\left|\langle \mathbf{M}_{\mathfrak{c}}^{-\frac{\lambda}{2}} \mathbf{K}_{\mathfrak{c}2}f, \mathbf{M}_{\mathfrak{c}}^{-\frac{\lambda}{2}}f\rangle \right|\lesssim\|\mathbf{M}_{\mathfrak{c}}^{-\frac{\lambda}{2}} f\|_{2}^2.
	\end{align*}
\end{Lemma}
\begin{proof}
	It follows from \eqref{1.34-0} and Lemma \ref{lem4.3-00} that
	\begin{align*}
		&\left|\langle \mathbf{M}_{\mathfrak{c}}^{-\frac{\lambda}{2}}\mathbf{K}_{\mathfrak{c}1} f,\mathbf{M}_{\mathfrak{c}}^{-\frac{\lambda}{2}} f\rangle \right|\nonumber\\
		&\lesssim \iint_{\mathbb{R}^3\times\mathbb{R}^3}|p-q|\mathbf{M}_{\mathfrak{c}}^{\frac{1-\lambda}{2}}(p)\mathbf{M}_{\mathfrak{c}}^{\frac{1+\lambda}{2}}(q)\cdot|\mathbf{M}_{\mathfrak{c}}^{-\frac{\lambda}{2}}(p)f(p)|\cdot |\mathbf{M}_{\mathfrak{c}}^{-\frac{\lambda}{2}}(q)f(q)|dpdq\nonumber\\
		&\lesssim \iint_{\mathbb{R}^3\times\mathbb{R}^3}|p-q|e^{-(1-\lambda)\bar{c}_1|p|}e^{-(1+\lambda)\bar{c}_1|q|}\cdot|\mathbf{M}_{\mathfrak{c}}^{-\frac{\lambda}{2}}(p)f(p)|\cdot |\mathbf{M}_{\mathfrak{c}}^{-\frac{\lambda}{2}}(q)f(q)|dpdq\nonumber\\
		&\lesssim \Big(\iint_{\mathbb{R}^3\times\mathbb{R}^3}|p-q|e^{-(1-\lambda)\bar{c}_1|p|}e^{-(1+\lambda)\bar{c}_1|q|}\cdot|\mathbf{M}_{\mathfrak{c}}^{-\frac{\lambda}{2}}(p)f(p)|^2dpdq\Big)^{\frac{1}{2}}\nonumber\\
		&\qquad \times \Big(\iint_{\mathbb{R}^3\times\mathbb{R}^3}|p-q|e^{-(1-\lambda)\bar{c}_1|p|}e^{-(1+\lambda)\bar{c}_1|q|}\cdot|\mathbf{M}_{\mathfrak{c}}^{-\frac{\lambda}{2}}(q)f(q)|^2dpdq\Big)^{\frac{1}{2}}\nonumber\\
		&\lesssim\|\mathbf{M}_{\mathfrak{c}}^{-\frac{\lambda}{2}}f\|_{2}^2.
	\end{align*}
	Using \eqref{2.31-0} and \eqref{3.24}, one has
	\begin{align*}
		\left|\langle \mathbf{M}_{\mathfrak{c}}^{-\frac{\lambda}{2}}\mathbf{K}_{\mathfrak{c}2} f,\mathbf{M}_{\mathfrak{c}}^{-\frac{\lambda}{2}} f\rangle \right|
		&\lesssim \iint_{\mathbb{R}^3\times\mathbb{R}^3}\xi(p,q)|\mathbf{M}_{\mathfrak{c}}^{-\frac{\lambda}{2}}(q)f(q)|\cdot |\mathbf{M}_{\mathfrak{c}}^{-\frac{\lambda}{2}}(p)f(p)|dpdq\nonumber\\
		&\lesssim \Big(\iint_{\mathbb{R}^3\times \mathbb{R}^3}\xi(p,q)\cdot|\mathbf{M}_{\mathfrak{c}}^{-\frac{\lambda}{2}}(p)f(p)|^2dpdq\Big)^{\frac{1}{2}}\nonumber\\
		&\qquad \times \Big(\iint_{\mathbb{R}^3\times \mathbb{R}^3}\xi(p,q)\cdot|\mathbf{M}_{\mathfrak{c}}^{-\frac{\lambda}{2}}(q)f(q)|^2dpdq\Big)^{\frac{1}{2}}\nonumber\\
		&\lesssim\|\mathbf{M}_{\mathfrak{c}}^{-\frac{\lambda}{2}}f\|_{2}^2.
	\end{align*}
	Therefore the proof of Lemma \ref{lem5.2} is completed.
\end{proof}

\begin{Lemma}\label{lem5.3}
	For any fixed $0\le \lambda<1$, there exists a positive constant $C$ which is independent of $\mathfrak{c}$, such that
	\begin{align*} 
		\langle \mathbf{M}_{\mathfrak{c}}^{-\frac{\lambda}{2}} \mathbf{L}_{\mathfrak{c}}  f, \mathbf{M}_{\mathfrak{c}}^{-\frac{\lambda}{2}} f\rangle \geq\f12 \|\mathbf{M}_{\mathfrak{c}}^{-\frac{\lambda}{2}}f\|_{\nu_{\mathfrak{c}}}^2-C\|f\|_{\nu_{\mathfrak{c}}}^2.
	\end{align*}
\end{Lemma}
\begin{proof}
	For any $r>0$, it follows from Lemmas \ref{lem5.2} and \ref{lem2.9} that
	\begin{align*}
		\left|\langle \mathbf{M}_{\mathfrak{c}}^{-\frac{\lambda}{2}}\mathbf{K}_{\mathfrak{c}}f, \mathbf{M}_{\mathfrak{c}}^{-\frac{\lambda}{2}}f\rangle \right|&\lesssim\Big\{\int_{|p|\leq r}+\int_{|p|\geq r}\Big\}|\mathbf{M}_{\mathfrak{c}}^{-\frac{\lambda}{2}}(p)f(p)|^2dp\nonumber\\
		&\lesssim \max\Big\{\frac{1}{1+r},\frac{1}{\mathfrak{c}}\Big\}\|\mathbf{M}_{\mathfrak{c}}^{-\frac{\lambda}{2}}f\|_{\nu_{\mathfrak{c}}}^2+C_r\|f\|_{\nu_{\mathfrak{c}}}^2. 
	\end{align*}
	Noting $\mathfrak{c}\gg 1$, taking $r$ suitably large, we have
	\begin{align*}
		\left| \langle \mathbf{M}_{\mathfrak{c}}^{-\frac{\lambda}{2}}\mathbf{K} f, \mathbf{M}_{\mathfrak{c}}^{-\frac{\lambda}{2}} f\rangle \right|\leq\f12\|\mathbf{M}_{\mathfrak{c}}^{-\frac{\lambda}{2}}f\|_{\nu_{\mathfrak{c}}}^2+C\|f\|_{\nu_{\mathfrak{c}}}^2,
	\end{align*}
	which, together with $\mathbf{L}_{\mathfrak{c}}f=\nu_{\mathfrak{c}} f-\mathbf{K}_{\mathfrak{c}}f$, yields that
	\begin{align*}
		\langle \mathbf{M}_{\mathfrak{c}}^{-\frac{\lambda}{2}}\mathbf{L}_{\mathfrak{c}}  f, \mathbf{M}_{\mathfrak{c}}^{-\frac{\lambda}{2}}f\rangle&=\langle \mathbf{M}_{\mathfrak{c}}^{-\frac{\lambda}{2}} \nu_{\mathfrak{c}} f, \mathbf{M}_{\mathfrak{c}}^{-\frac{\lambda}{2}}f\rangle-\langle \mathbf{M}_{\mathfrak{c}}^{-\frac{\lambda}{2}}\mathbf{K}_{\mathfrak{c}}f, \mathbf{M}_{\mathfrak{c}}^{-\frac{\lambda}{2}}f\rangle\nonumber\\ 
		&\geq\f12\|\mathbf{M}_{\mathfrak{c}}^{-\frac{\lambda}{2}}f\|_{\nu_{\mathfrak{c}}}^2-C\|f\|_{\nu_{\mathfrak{c}}}^2.
	\end{align*}
	Therefore the proof of Lemma \ref{lem5.3} is completed.
\end{proof}

\begin{Proposition}\label{prop5.4}
	For any fixed $0\le \lambda<1$, $m>\frac{3}{2}$, suppose $g\in \mathcal{N}_{\mathfrak{c}}^\perp$ and
	\begin{align*}
		\|(1+|p|)^m \mathbf{M}_{\mathfrak{c}}^{-\frac{\lambda}{2}}g\|_{\infty}<\infty,
	\end{align*}
	then it holds
	\begin{align}\label{2.34}
		|\mathbf{L}_{\mathfrak{c}} ^{-1}g(p)|\lesssim\|(1+|p|)^m \mathbf{M}_{\mathfrak{c}}^{-\frac{\lambda}{2}}g\|_{\infty}\cdot \mathbf{M}_{\mathfrak{c}}^{\frac{\lambda}{2}}(p), \quad p\in \mathbb{R}^3,
	\end{align}
     where the constant is independent of $\mathfrak{c}$.
\end{Proposition}

\begin{proof}
	Let $f=\mathbf{L}_{\mathfrak{c}} ^{-1}g\in \mathcal{N}_{\mathfrak{c}}^\perp$, then we have $g=\mathbf{L}_{\mathfrak{c}}f=\nu_{\mathfrak{c}} f-\mathbf{K}_{\mathfrak{c}}f$. Using Lemma \ref{lem5.1}, we get 
	\begin{align}\label{2.45}
		\mathbf{M}_{\mathfrak{c}}^{-\frac{\lambda}{2}}(p)|f(p)|&\lesssim  \nu_{\mathfrak{c}}^{-1}(p)\mathbf{M}_{\mathfrak{c}}^{-\frac{\lambda}{2}}(p)|g(p)|+\nu_{\mathfrak{c}}^{-1}(p)\mathbf{M}_{\mathfrak{c}}^{-\frac{\lambda}{2}}(p)|\mathbf{K}_{\mathfrak{c}}f(p)|\nonumber\\
		&\lesssim \mathbf{M}_{\mathfrak{c}}^{-\frac{\lambda}{2}}(p)|g(p)|+\mathbf{M}_{\mathfrak{c}}^{-\frac{\lambda}{2}}(p)|\mathbf{K}_{\mathfrak{c}}f(p)|\nonumber\\
		&\lesssim \mathbf{M}_{\mathfrak{c}}^{-\frac{\lambda}{2}}(p)|g(p)|+\|\mathbf{M}_{\mathfrak{c}}^{-\frac{\lambda}{2}} f\|_{\nu_{\mathfrak{c}}}.
	\end{align}
	By Proposition \ref{prop4.13} and Lemma \ref{lem5.3}, we have
	\begin{align*} 
		\|\mathbf{M}_{\mathfrak{c}}^{-\frac{\lambda}{2}}f\|_{\nu_{\mathfrak{c}}}^2&\lesssim\langle \mathbf{M}_{\mathfrak{c}}^{-\frac{\lambda}{2}}\mathbf{L}_{\mathfrak{c}}  f, \mathbf{M}_{\mathfrak{c}}^{-\frac{\lambda}{2}}f\rangle+\|f\|_{\nu_{\mathfrak{c}}}^2\nonumber\\
		&\lesssim \langle \mathbf{M}_{\mathfrak{c}}^{-\frac{\lambda}{2}}\mathbf{L}_{\mathfrak{c}}  f,\mathbf{M}_{\mathfrak{c}}^{-\frac{\lambda}{2}}f\rangle+\langle\mathbf{L}_{\mathfrak{c}}  f, f\rangle\nonumber\\
		&\lesssim\int_{\mathbb{R}^3}(1+|p|)^m\mathbf{M}_{\mathfrak{c}}^{-\frac{\lambda}{2}}(p)|g(p)|\cdot(1+|p|)^{-m}\cdot |\mathbf{M}_{\mathfrak{c}}^{-\frac{\lambda}{2}}(p)f(p)|dp\nonumber\\
		&\lesssim\|\mathbf{M}_{\mathfrak{c}}^{-\frac{\lambda}{2}}f\|_{\nu_{\mathfrak{c}}} \cdot \|(1+|p|)^m \mathbf{M}_{\mathfrak{c}}^{-\frac{\lambda}{2}}g\|_{\infty}\cdot \Big(\int_{\mathbb{R}^3}\f{1}{(1+|p|)^{2m}}dp\Big)^{\f12},
	\end{align*}
	which, together with $m>\frac{3}{2}$, yields that
	\begin{align}\label{2.37}
		\|\mathbf{M}_{\mathfrak{c}}^{-\frac{\lambda}{2}}f\|_{\nu_{\mathfrak{c}}}\lesssim \|(1+|p|)^m \mathbf{M}_{\mathfrak{c}}^{-\frac{\lambda}{2}}g\|_{\infty}<\infty.
	\end{align}
	Combining \eqref{2.37} and \eqref{2.45}, one has
	\begin{align*}
		\mathbf{M}_{\mathfrak{c}}^{-\frac{\lambda}{2}}(p)|f(p)|&\lesssim \|(1+|p|)^m \mathbf{M}_{\mathfrak{c}}^{-\frac{\lambda}{2}}g\|_{\infty},
	\end{align*}
	which concludes \eqref{2.34}. Therefore the proof of Proposition \ref{prop5.4} is completed.
\end{proof}

\section{Uniform-in-$\mathfrak{c}$ estimates on the linear part of Hilbert expansion}

\subsection{Reformulation of $F_{n+1}^{\mathfrak{c}}$}

For $n=0,1,\cdots, 2k-2$, we decompose $F_{n+1}^{\mathfrak{c}}$ as
\begin{align*}
	\frac{F_{n+1}^{\mathfrak{c}}}{\sqrt{\mathbf{M}_{\mathfrak{c}}}}=\mathbf{P}_{\mathfrak{c}}\Big(\frac{F_{n+1}^{\mathfrak{c}}}{\sqrt{\mathbf{M}_{\mathfrak{c}}}}\Big)+\{\mathbf{I}-\mathbf{P}_{\mathfrak{c}}\}\Big(\frac{F_{n+1}^{\mathfrak{c}}}{\sqrt{\mathbf{M}_{\mathfrak{c}}}}\Big),
\end{align*}
where 
\begin{align} \label{5.1-20}
	\mathbf{P}_{\mathfrak{c}}\Big(\frac{F_{n+1}^{\mathfrak{c}}}{\sqrt{\mathbf{M}_{\mathfrak{c}}}}\Big)=\Big[a_{n+1}+b_{n+1}\cdot p+c_{n+1}\frac{p^0}{\mathfrak{c}}\Big]\sqrt{\mathbf{M}_{\mathfrak{c}}}.
\end{align}
Using \eqref{1.16-01}--\eqref{1.17-01} and Lemma \ref{lem4.11-0}, by tedious calculations, one has 
\begin{align*} 
	\int_{\mathbb{R}^3}F_{n+1}^{\mathfrak{c}}dp
	&=\int_{\mathbb{R}^3}\Big[a_{n+1}+b_{n+1}\cdot p+c_{n+1}\frac{p^0}{\mathfrak{c}}\Big] \mathbf{M}_{\mathfrak{c}}dp\nonumber\\
	&=\frac{n_0u^0}{\mathfrak{c}}a_{n+1}+\frac{e_0+P_0}{\mathfrak{c}^3}u^0(u\cdot b_{n+1})+\frac{e_0(u^0)^2+P_0|u|^2}{\mathfrak{c}^4}c_{n+1},\\
	\int_{\mathbb{R}^3} \frac{p_j p}{p^0} F_{n+1}^{\mathfrak{c}} d p 
	& =\int_{\mathbb{R}^3} \frac{p_j p}{p^0}\left[a_{n+1}+b_{n+1} \cdot p+c_{n+1}\frac{p^0}{\mathfrak{c}}\right] \mathbf{M}_{\mathfrak{c}} d p+\int_{\mathbb{R}^3} \frac{p_j p}{p^0} \sqrt{\mathbf{M}_{\mathfrak{c}}}\{\mathbf{I}-\mathbf{P}_{\mathfrak{c}}\}\left(\frac{F_{n+1}^{\mathfrak{c}}}{\sqrt{\mathbf{M}_{\mathfrak{c}}}}\right) d p \nonumber\\
	& =\frac{e_0+P_0}{\mathfrak{c}^3}u_j u a_{n+1}+\frac{n_0}{\mathfrak{c}\gamma K_2(\gamma)}\left(6 K_3(\gamma)+\gamma K_2(\gamma)\right) u_j u\left[\left(u \cdot b_{n+1}\right)+\frac{u^0}{\mathfrak{c}} c_{n+1}\right] \nonumber\\
	&\qquad  +\mathbf{e}_ja_{n+1}\frac{P_0}{\mathfrak{c}}+\frac{\mathfrak{c}n_0 K_3(\gamma)}{\gamma K_2(\gamma)}\left(u b_{n+1, j}+u_j b_{n+1}\right) \nonumber\\
	&\qquad  +\mathbf{e}_j \frac{\mathfrak{c}n_0 K_3(\gamma)}{\gamma K_2(\gamma)}\left[\left(u \cdot b_{n+1}\right)+\frac{u^0}{\mathfrak{c}} c_{n+1}\right]+\int_{\mathbb{R}^3} \frac{p_j p}{p^0} \sqrt{\mathbf{M}_{\mathfrak{c}}}\{\mathbf{I}-\mathbf{P}_{\mathfrak{c}}\}\left(\frac{F_{n+1}^{\mathfrak{c}}}{\sqrt{\mathbf{M}_{\mathfrak{c}}}}\right) d p,\\
	\int_{\mathbb{R}^3} \hat{p}_j F_{n+1}^{\mathfrak{c}} d p & =\int_{\mathbb{R}^3}\hat{p}_j\Big[a_{n+1}+b_{n+1} \cdot p+c_{n+1} \frac{p^0}{\mathfrak{c}}\Big] \mathbf{M}_{\mathfrak{c}}dp+\int_{\mathbb{R}^3} \hat{p}_j \sqrt{\mathbf{M}_{\mathfrak{c}}}\{\mathbf{I}-\mathbf{P}_{\mathfrak{c}}\}\left(\frac{F_{n+1}^{\mathfrak{c}}}{\sqrt{\mathbf{M}_{\mathfrak{c}}}}\right) d p\nonumber\\
	&=n_0 u_j a_{n+1}+\frac{e_0+P_0}{\mathfrak{c}^2} u_j\left(u \cdot b_{n+1}\right)+P_0 b_{n+1, j} \nonumber\\
	&\qquad  +\frac{e_0+P_0}{\mathfrak{c}^3}u^0 u_j c_{n+1}+\int_{\mathbb{R}^3} \hat{p}_j \sqrt{\mathbf{M}_{\mathfrak{c}}}\{\mathbf{I}-\mathbf{P}_{\mathfrak{c}}\}\left(\frac{F_{n+1}^{\mathfrak{c}}}{\sqrt{\mathbf{M}_{\mathfrak{c}}}}\right) d p,\\
	\int_{\mathbb{R}^3}p_j F_{n+1}^{\mathfrak{c}} d p
	&=\int_{\mathbb{R}^3} p_j\Big[a_{n+1}+b_{n+1} \cdot p+c_{n+1} \frac{p^0}{\mathfrak{c}}\Big] \mathbf{M}_{\mathfrak{c}} d p \nonumber\\
	&=\frac{n_0}{\mathfrak{c}\gamma K_2(\gamma)}\left[\left(6 K_3(\gamma)+\gamma K_2(\gamma)\right) u^0 u_j\left(u \cdot b_{n+1}\right)+\mathfrak{c}^2K_3(\gamma) u^0 b_{n+1, j}\right] \nonumber\\
	&\qquad +\frac{n_0}{\mathfrak{c}^2\gamma K_2(\gamma)}\left[\left(5 K_3(\gamma)+\gamma K_2(\gamma)\right)\left(u^0\right)^2+K_3(\gamma)|u|^2\right] u_j c_{n+1}\nonumber \\
	&\qquad +\frac{e_0+P_0}{\mathfrak{c}^3} u^0 u_j a_{n+1},\\
	\int_{\mathbb{R}^3} p^0 F_{n+1}^{\mathfrak{c}} d p
	& =\int_{\mathbb{R}^3} p^0\left[a_{n+1}+b_{n+1} \cdot p+c_{n+1}\frac{p^0}{\mathfrak{c}}\right] \mathbf{M}_{\mathfrak{c}} d p \nonumber\\
	& =\frac{n_0}{\mathfrak{c}\gamma K_2(\gamma)}\left[\left(5 K_3(\gamma)+\gamma K_2(\gamma)\right)\left(u^0\right)^2+K_3(\gamma)|u|^2\right]\left(u \cdot b_{n+1}\right) \nonumber\\
	&\qquad +\frac{n_0}{\mathfrak{c}^2\gamma K_2(\gamma)}\left[\left(3 K_3(\gamma)+\gamma K_2(\gamma)\right)\left(u^0\right)^2+3 K_3(\gamma)|u|^2\right] u^0 c_{n+1}\nonumber\\
	&\qquad +\frac{e_0\left(u^0\right)^2+P_0|u|^2}{\mathfrak{c}^3}a_{n+1},
\end{align*}
where $\mathbf{e}_j\ (j=1,2,3)$ are the unit base vectors in $\mathbb{R}^3$.

Next, we shall derive the equation for $(a_{n+1},b_{n+1},c_{n+1})$. Notice that
\begin{align}\label{2.26}
	\partial_t F_{n+1}^{\mathfrak{c}}+\hat{p}\cdot \nabla_x F_{n+1}^{\mathfrak{c}}=\sum_{\substack{i+j=n+2\\ i,j\ge 0}}Q_{\mathfrak{c}}(F_i^{\mathfrak{c}},F_i^{\mathfrak{c}})\ (\text{or}\ \sum_{\substack{i+j=n+2\\ i,j\ge 1}}Q_{\mathfrak{c}}(F_i^{\mathfrak{c}},F_i^{\mathfrak{c}})\ \text{when $n=2k-2$}).
\end{align}
Integrating \eqref{2.26} with respect to $p$, we have
\begin{align}\label{2.27}
	& \partial_t\left(\frac{n_0 u^0}{\mathfrak{c}} a_{n+1}+\frac{e_0+P_0}{\mathfrak{c}^3} u^0\left(u \cdot b_{n+1}\right)+\frac{e_0\left(u^0\right)^2+P_0|u|^2}{\mathfrak{c}^4}c_{n+1}\right) \nonumber\\
	&\qquad +\nabla_x \cdot\left(n_0 u a_{n+1}+\frac{e_0+P_0}{\mathfrak{c}^2} u\left(u \cdot b_{n+1}\right)+P_0 b_{n+1}+\frac{e_0+P_0}{\mathfrak{c}^3} u^0 u c_{n+1}\right) \nonumber\\
	&\qquad  +\nabla_x \cdot \int_{\mathbb{R}^3} \hat{p} \sqrt{\mathbf{M}_{\mathfrak{c}}}\{\mathbf{I}-\mathbf{P}_{\mathfrak{c}}\}\left(\frac{F_{n+1}^{\mathfrak{c}}}{\sqrt{\mathbf{M}_{\mathfrak{c}}}}\right) d p=0.
\end{align}	
Multiplying \eqref{2.26} by $p_j$ and integrating over $\mathbb{R}^3$, one gets
\begin{align}\label{2.28}
	& \partial_t\left(\frac{e_0+P_0}{\mathfrak{c}^3} u^0 u_j a_{n+1}+\frac{n_0}{\mathfrak{c}\gamma K_2(\gamma)}\left[\left(6 K_3(\gamma)+\gamma K_2(\gamma)\right) u^0 u_j\left(u \cdot b_{n+1}\right)+\mathfrak{c}^2K_3(\gamma) u^0 b_{n+1, j}\right]\right.\nonumber\\
	&\qquad \left. +\frac{n_0}{\mathfrak{c}^2\gamma K_2(\gamma)}\left[\left(5 K_3(\gamma)+\gamma K_2(\gamma)\right)\left(u^0\right)^2+K_3(\gamma)|u|^2\right] u_j c_{n+1}\right) \nonumber\\
	&\qquad  +\nabla_x \cdot\left(\frac{e_0+P_0}{\mathfrak{c}^2} u_j u a_{n+1}+\frac{n_0}{\gamma K_2(\gamma)}\left(6 K_3(\gamma)+\gamma K_2(\gamma)\right) u_j u\left[\left(u \cdot b_{n+1}\right)+\frac{u^0}{\mathfrak{c}} c_{n+1}\right]\right) \nonumber\\
	&\qquad  +\partial_{x_j}\left(P_0 a_{n+1}\right)+\nabla_x \cdot\left[\frac{\mathfrak{c}^2n_0 K_3(\gamma)}{\gamma K_2(\gamma)}\left(u b_{n+1, j}+u_j b_{n+1}\right)\right] \nonumber\\
	&\qquad +\partial_{x_j}\left(\frac{\mathfrak{c}^2n_0 K_3(\gamma)}{\gamma K_2(\gamma)}\left[\left(u \cdot b_{n+1}\right)+\frac{u^0}{\mathfrak{c}} c_{n+1}\right]\right)+\nabla_x \cdot \int_{\mathbb{R}^3}p_j\hat{p} \sqrt{\mathbf{M}_{\mathfrak{c}}}\{\mathbf{I}-\mathbf{P}_{\mathfrak{c}}\}\left(\frac{F_{n+1}^{\mathfrak{c}}}{\sqrt{\mathbf{M}_{\mathfrak{c}}}}\right) d p=0
\end{align}
for $j=1,2,3$ with $b_{n+1}=\left(b_{n+1,1}, b_{n+1,2}, b_{n+1,3}\right)^t$. 

Multiplying \eqref{2.26} by $\frac{p^0}{\mathfrak{c}}$ and integrating over $\mathbb{R}^3$, one obtains that
\begin{align}\label{2.29-10}
	& \partial_t\left(\frac{e_0\left(u^0\right)^2+P_0|u|^2}{\mathfrak{c}^4} a_{n+1}+\frac{n_0}{\mathfrak{c}^2\gamma K_2(\gamma)}\left[\left(5 K_3(\gamma)+\gamma K_2(\gamma)\right)\left(u^0\right)^2+K_3(\gamma)|u|^2\right]\left(u \cdot b_{n+1}\right)\right. \nonumber\\
	&\qquad \left.+\frac{n_0}{\mathfrak{c}^3\gamma K_2(\gamma)}\left[\left(3 K_3(\gamma)+\gamma K_2(\gamma)\right)\left(u^0\right)^2+3 K_3(\gamma)|u|^2\right] u^0 c_{n+1}\right) \nonumber\\
	&\qquad +\nabla_x \cdot\left(\frac{e_0+P_0}{\mathfrak{c}^3} u^0 u a_{n+1}+\frac{n_0}{\mathfrak{c}\gamma K_2(\gamma)}\left[\left(6 K_3(\gamma)+\gamma K_2(\gamma)\right) u^0 u\left(u \cdot b_{n+1}\right)+\mathfrak{c}^2K_3(\gamma)u^0 b_{n+1}\right]\right. \nonumber\\
	&\qquad\left. +\frac{n_0}{\mathfrak{c}^2\gamma K_2(\gamma)}\left[\left(5 K_3(\gamma)+\gamma K_2(\gamma)\right)\left(u^0\right)^2+K_3(\gamma)|u|^2\right] u c_{n+1}\right)=0.
\end{align}

After a tedious computation, we can rewrite \eqref{2.27}--\eqref{2.29-10} into the following linear symmetric hyperbolic system:
\begin{align}\label{2.19}
\mathbf{A}_0 \partial_t U_{n+1}+\sum_{i=1}^3 \mathbf{A}_i \partial_i U_{n+1}+\mathbf{B} U_{n+1}=\mathbf{S}_{n+1},
\end{align}
where
\begin{equation*} 
	U_{n+1}=\begin{pmatrix}
		a_{n+1} \\
		b_{n+1} \\
		c_{n+1}
	\end{pmatrix},\quad 
\mathbf{S}_{n+1}=\begin{pmatrix}
	-\nabla_x\cdot \int_{\mathbb{R}^3}\hat{p}\sqrt{\mathbf{M}_{\mathfrak{c}}}\{\mathbf{I}-\mathbf{P}_{\mathfrak{c}}\}\Big(\frac{F_{n+1}^{\mathfrak{c}}}{\sqrt{\mathbf{M}_{\mathfrak{c}}}}\Big)dp \\
	-\nabla_x\cdot \int_{\mathbb{R}^3}p\otimes \hat{p}\sqrt{\mathbf{M}_{\mathfrak{c}}}\{\mathbf{I}-\mathbf{P}_{\mathfrak{c}}\}\Big(\frac{F_{n+1}^{\mathfrak{c}}}{\sqrt{\mathbf{M}_{\mathfrak{c}}}}\Big)dp \\
	0
\end{pmatrix}.
\end{equation*}
The matrices $\mathbf{A}_0, \mathbf{A}_i\ (i=1,2,3)$ and $\mathbf{B}$ depend only on the smooth relativistic Euler solution $(n_0,u,T_0)$. To express these matrices, we denote 
$$
h(t,x):=\frac{e_0+P_0}{n_0},\quad h_1(t, x):=\frac{n_0}{\gamma K_2(\gamma)}\left(6 K_3(\gamma)+\gamma K_2(\gamma)\right), \quad h_2(t, x):=\frac{n_0 K_3(\gamma)}{\gamma K_2(\gamma)}.
$$
Then the matrices $\mathbf{A}_0, \mathbf{A}_i,(i=1,2,3)$ in \eqref{2.19} are
\begin{equation*} 
\mathbf{A}_0=\left(\begin{array}{ccc}
	\frac{n_0 u^0}{\mathfrak{c}} & \frac{n_0 u^0 h u^t}{\mathfrak{c}^3} & \frac{e_0\left(u^0\right)^2+P_0|u|^2}{\mathfrak{c}^4}\\
	\frac{n_0 u^0 h u^t}{\mathfrak{c}^3} & \left(\frac{h_1}{\mathfrak{c}} u \otimes u+\mathfrak{c} h_2 \mathbf{I}\right) u^0 & \left(\frac{h_1}{\mathfrak{c}^2}\left(u^0\right)^2-h_2\right) u \\
	\frac{e_0\left(u^0\right)^2+P_0|u|^2}{\mathfrak{c}^4}& \left(\frac{h_1}{\mathfrak{c}^2}\left(u^0\right)^2-h_2\right) u^t & \left(\frac{h_1}{\mathfrak{c}^3}\left(u^0\right)^2-\frac{3h_2}{\mathfrak{c} }\right) u^0
\end{array}\right)
\end{equation*}
and
\begin{equation*}
\mathbf{A}_i=\left(\begin{array}{ccc}
	n_0 u_i & \frac{1}{\mathfrak{c}^2}n_0 h u_i u^t+P_0 \mathbf{e}_i^t & \frac{1}{\mathfrak{c}^3}n_0 h u^0 u_i \\
	\frac{1}{\mathfrak{c}^2}n_0 h u_i u+P_0 \mathbf{e}_i & h_1 u_i u \otimes u+\mathfrak{c}^2h_2\left(u_i \mathbf{I}+\tilde{\mathbf{A}}_i\right) & \left(\frac{h_1}{\mathfrak{c}} u_i u+\mathfrak{c}h_2 \mathbf{e}_i\right) u^0 \\
	\frac{1}{\mathfrak{c}^3}n_0 h u^0 u_i & \left(\frac{h_1}{\mathfrak{c}} u_i u^t+\mathfrak{c}h_2 \mathbf{e}_i^t\right) u^0 & \left(\frac{h_1}{\mathfrak{c}^2}\left(u^0\right)^2-h_2\right) u_i
\end{array}\right),
\end{equation*}
where 
\begin{align*}
	\big(\tilde{\mathbf{A}_i}\big)_{jk}=\delta_{ij}u_k+\delta_{ik}u_j,\quad 1\le j,k\le 3.
\end{align*}
The matrix $\mathbf{B}=(b_{ij})$ has the form
\begin{align*}
	&b_{11}=0,\quad (b_{12},b_{13},b_{14})=\frac{n_0u^0}{\mathfrak{c}}\partial_t\Big(\frac{hu^t}{\mathfrak{c}^2}\Big)+n_0u^t\Big[\nabla_x\Big(\frac{hu}{\mathfrak{c}^2}\Big)\Big]^t+(\nabla_x P_0)^t,\nonumber\\
	&b_{15}=\frac{n_0u^0}{\mathfrak{c}^2}\partial_t\Big(\frac{hu^0}{\mathfrak{c}^2}\Big)+\frac{n_0u}{\mathfrak{c}}\cdot \nabla_x\Big(\frac{hu^0}{\mathfrak{c}^2}\Big)-\partial_t\Big(\frac{P_0}{\mathfrak{c}^2}\Big),\nonumber\\
	&(b_{21},b_{31},b_{41})=\frac{n_0u^0}{\mathfrak{c}}\partial_t\Big(\frac{hu}{\mathfrak{c}^2}\Big)+\nabla_x P_0+\nabla_{x}\Big(\frac{hu}{\mathfrak{c}^2}\Big)n_0u,\nonumber\\
	&(b_{j2},b_{j3},b_{j4})=\frac{n_0u^0}{\mathfrak{c}}\partial_t\Big[\frac{h_1}{n_0}u_ju^t+\frac{\mathfrak{c}^2h_2}{n_0}\mathbf{e}_j^t\Big]+n_0(u\cdot\nabla_x)\Big(\frac{h_1}{n_0}u_ju^t\Big)+n_0u^t \nabla_x\Big(\frac{\mathfrak{c}^2h_2}{n_0}\Big)\mathbf{e}_j^t\nonumber\\
	&\quad \quad \quad \quad +\Big[\nabla_x(\mathfrak{c}^2h_2u_j)\Big]^t+\partial_{x_j}(\mathfrak{c}^2h_2u^t),\nonumber\\
	&b_{j5}=-\partial_t(h_2u_j)+\frac{n_0u^0}{\mathfrak{c}^2}\partial_t\Big(\frac{h_1}{n_0}u_ju^0\Big)+\frac{n_0}{\mathfrak{c}}u^t\nabla_x\Big(\frac{h_1}{n_0}u_ju^0\Big)+\partial_{x_j}(\mathfrak{c}h_2u^0),\nonumber\\
	&b_{51}=\frac{n_0u^0}{\mathfrak{c}^2}\partial_t\Big(\frac{hu^0}{\mathfrak{c}^2}\Big)+\frac{n_0u}{\mathfrak{c}}\cdot \nabla_x\Big(\frac{hu^0}{\mathfrak{c}^2}\Big)-\partial_t\Big(\frac{P_0}{\mathfrak{c}^2}\Big),\nonumber\\
	&(b_{52},b_{53},b_{54})=\frac{n_0u^0}{\mathfrak{c}^2}\partial_t\Big(\frac{h_1}{n_0}u^0u^t\Big)-\partial_t(h_2u^t)+\frac{n_0}{\mathfrak{c}}u^t\Big[\nabla_x\Big(\frac{h_1}{n_0}u^0u\Big)\Big]^t+\Big(\nabla(\mathfrak{c}h_2u^0)\Big)^t,\nonumber\\
	&b_{55}=\frac{n_0u^0}{\mathfrak{c}^3}\partial_t\Big(\frac{h_1}{n_0}(u^0)^2-3\mathfrak{c}^2\frac{h_2}{n_0}|u|^2\Big)+\frac{n_0}{\mathfrak{c}^2}u\cdot \nabla_{x}\Big(\frac{h_1}{n_0}(u^0)^2-3\mathfrak{c}^2\frac{h_2}{n_0}|u|^2\Big)+\nabla_x\cdot(2h_2u).
\end{align*}

Next, we prove the positivity of $\mathbf{A}_0$. Set $\phi(\gamma):=\frac{K_3(\gamma)}{K_2(\gamma)}$. A direct calculation shows that 
\begin{align*}
	\det(\mathbf{A}_0)_{1\times 1}\ge \frac{n_0u^0}{\mathfrak{c}}>0,&\quad \det(\mathbf{A}_0)_{2\times 2}\ge \Big(\frac{n_0u^0}{\mathfrak{c}}\Big)^2\frac{\mathfrak{c}^2\phi}{\gamma}>0,\nonumber\\
	\det(\mathbf{A}_0)_{3\times 3}\ge \Big(\frac{n_0u^0}{\mathfrak{c}}\Big)^3 \Big(\frac{\mathfrak{c}^2\phi}{\gamma}\Big)^2>0,&\quad \det(\mathbf{A}_0)_{4\times 4}\ge \Big(\frac{n_0u^0}{\mathfrak{c}}\Big)^4 \Big(\frac{\mathfrak{c}^2\phi}{\gamma}\Big)^3>0,
\end{align*}
and 
\begin{align}\label{5.19-10}
	\det \mathbf{A}_0=\Big(\frac{n_0u^0}{\mathfrak{c}}\Big)^5\Big(\frac{\mathfrak{c}^2\phi}{\gamma}\Big)^3(u^0)^{-2}\Big\{|u|^2\mathfrak{c}^2(\Psi-\frac{\Psi}{ \gamma \phi}-\frac{\phi}{\gamma})+\mathfrak{c}^4(\Psi-\frac{1}{\gamma^2}-\frac{\phi}{\gamma})\Big\},
\end{align}
where $\Psi:=1+\frac{6}{\gamma}\phi-\phi^2$. To prove the positivity of \eqref{5.19-10}, we use \cite[Proposition 10]{Ruggeri} to get
\begin{align}\label{5.7-20}
	\phi^2-\frac{5}{\gamma}\phi+\frac{1}{\gamma^2}-1<0,
\end{align}
which yields that 
\begin{align}\label{5.21-10}
	\Psi-\frac{1}{\gamma^2}-\frac{\phi}{\gamma}=1+\frac{6}{\gamma}\phi-\phi^2-\frac{1}{\gamma^2}-\frac{\phi}{\gamma}=-(\phi^2-\frac{5}{\gamma}\phi+\frac{1}{\gamma^2}-1)>0.
\end{align}
A direct calculation shows that
\begin{align*}
	\mathfrak{h}:=&\Psi-\frac{\Psi}{\gamma \phi}-\frac{\phi}{\gamma}=\frac{1}{\phi \Psi}(\phi \Psi-\frac{\Psi}{\gamma}-\frac{\phi^2}{\gamma})\\
	=&\frac{1}{\phi \Psi}\Big\{\phi (1+\frac{6}{\gamma}\phi-\phi^2)-\frac{1}{\gamma}(1+\frac{6}{\gamma}\phi-\phi^2)-\frac{\phi^2}{\gamma}\Big\}\\
	=&-\frac{1}{\phi \Psi}\Big\{\phi^3+\frac{1}{\gamma}+\frac{6}{\gamma^2}\phi-\phi-\frac{6}{\gamma}\phi^2\Big\}.
\end{align*}
Observing $\phi=\frac{K_3(\gamma)}{K_2(\gamma)}=\frac{4}{\gamma}+\frac{K_1(\gamma)}{K_2(\gamma)}:=\frac{4}{\gamma}+\phi_1$, one has
\begin{align}\label{5.8-20}
	&\phi^3+\frac{1}{\gamma}+\frac{6}{\gamma^2}\phi-\phi-\frac{6}{\gamma}\phi^2 \nonumber\\
	&=\gamma \phi_1^3+6\phi_1^2+\frac{6}{\gamma}\phi_1-\gamma \phi_1-3-\frac{8}{\gamma^2} \nonumber\\
	&=(\gamma \phi_1^3+4\phi_1^2-\gamma \phi_1-1)+\frac{2}{\gamma^2}(\gamma^2 \phi_1^2+3\gamma \phi_1-\gamma^2-3)-\frac{2}{\gamma^2}.
\end{align}
Applying \cite[Proposion 10]{Ruggeri} again, we know that both of the first two terms on the RHS of \eqref{5.8-20} are negative. Thus one gets
\begin{align*}
	\mathfrak{h}>0,
\end{align*}
which, together with \eqref{5.19-10} and \eqref{5.21-10}, yields that
\begin{align*}
	\det \mathbf{A}_0\ge \Big(\frac{n_0u^0}{\mathfrak{c}}\Big)^5\Big(\frac{\mathfrak{c}^2\phi}{\gamma}\Big)^3\frac{\mathfrak{c}^4}{(u^0)^2}\{-	(\phi^2-\frac{5}{\gamma}\phi+\frac{1}{\gamma^2}-1)\}>0.
\end{align*}
Therefore $\mathbf{A}_0$ is actually a positive definite matrix.

\subsection{Uniform-in-$\mathfrak{c}$ estimates on $F^{\mathfrak{c}}_n$}

\begin{Proposition}\label{prop6.3}
	Let the local relativistic Maxwellian $F^{\mathfrak{c}}_0=\mathbf{M}_{\mathfrak{c}}(n_0, u, T_0; p)$ be as in \eqref{1.13-00} formed by $(n_0(t, x), u(t, x), T_0(t,x))$ which is a smooth solution to the relativistic Euler equations \eqref{4.8-0} on a time interval $[0, T] \times \mathbb{R}^3$. Then we can construct the smooth terms $F^{\mathfrak{c}}_1, \ldots, F^{\mathfrak{c}}_{2k-1}$ of the Hilbert expansion in $(t, x) \in[0, T] \times \mathbb{R}^3$ such that, for any $0<\lambda<1$, the following estimates hold
	\begin{align}\label{5.23-10}
		\left|F^{\mathfrak{c}}_n(t, x, p)\right| \leq C(\lambda) \mathbf{M}_{\mathfrak{c}}^{\lambda}(n_0(t, x), u(t, x), T_0(t,x); p), \quad n=1,2, \ldots, 2k-1
	\end{align}
	and 
	\begin{align}\label{5.24-10}
		\left|\partial^{m} F^{\mathfrak{c}}_n(t, x, p)\right|\leq  C(\lambda) \mathbf{M}_{\mathfrak{c}}^{\lambda}(n_0(t, x), u(t, x), T_0(t,x); p), \quad n=1,2, \ldots, 2k-1,\quad m\ge 1,
	\end{align}
	where $\partial^{m}:=\partial^m_{t,x}$. We emphasize that the constants in \eqref{5.23-10} and \eqref{5.24-10} are independent of $\mathfrak{c}$.
\end{Proposition}

\begin{proof}
It is noted that $\mathbf{A}_0$, $\mathbf{A}_i$ and $\mathbf{B}$ in \eqref{2.19} depend only on the smooth functions $n_0(t,x)$, $u(t,x)$ and $T_0(t,x)$. 
Denote $\psi_1:=\{\mathbf{I}-\mathbf{P}_{\mathfrak{c}}\}\Big(\frac{F_1^{\mathfrak{c}}}{\sqrt{\mathbf{M}_{\mathfrak{c}}}}\Big)$, then one has 
\begin{align*}
	F_1^{\mathfrak{c}}=\Big(a_1+b_1\cdot p+c_1\frac{p^0}{\mathfrak{c}}\Big)\mathbf{M}_{\mathfrak{c}}+\sqrt{\mathbf{M}_{\mathfrak{c}}}\psi_1,
\end{align*}
which yields that
\begin{align}\label{6.46-0}
	\partial F_1^{\mathfrak{c}}=\Big(\partial a_1+\partial b_1\cdot p+\partial c_1\frac{p^0}{\mathfrak{c}}\Big)\mathbf{M}_{\mathfrak{c}}+\Big(a_1+b_1\cdot p+c_1\frac{p^0}{\mathfrak{c}}\Big)\partial \mathbf{M}_{\mathfrak{c}}+\partial \sqrt{\mathbf{M}_{\mathfrak{c}}}\psi_1+\sqrt{\mathbf{M}_{\mathfrak{c}}}\partial \psi_1,
\end{align}
where $\partial=\partial_t$ or $\partial=\partial_{x_j}$ for $j=1,2,3$. A direct calculation shows that 
\begin{align*}
	|\partial \mathbf{M}_{\mathfrak{c}}|\le C \mathbf{M}_{\mathfrak{c}}^{1-},
\end{align*}
where $C$ depends on $\|\nabla_{t,x}(n_0,u,T_0)\|_{\infty}$. We denote
\begin{align*}
	g_1:=Q_{\mathfrak{c}}(\mathbf{M}_{\mathfrak{c}},\sqrt{\mathbf{M}_{\mathfrak{c}}}\psi_1)+Q_{\mathfrak{c}}(\sqrt{\mathbf{M}_{\mathfrak{c}}}\psi_1,\mathbf{M}_{\mathfrak{c}}),
\end{align*}
then it follows from \eqref{1.18-0} that
\begin{align}\label{5.29-10}
	g_1=\partial_t \mathbf{M}_{\mathfrak{c}}+\hat{p}\cdot \nabla_x \mathbf{M}_{\mathfrak{c}},
\end{align}
which implies that $|\partial^m g_1|\lesssim \mathbf{M}_{\mathfrak{c}}^{1-}$ for any $m\ge 0$. 

To estimate $\partial \psi_1$, we apply $\partial$ to \eqref{5.29-10} to obtain 
\begin{align*} 
	\partial  g_1&=Q_{\mathfrak{c}}(\partial  \mathbf{M}_{\mathfrak{c}},\sqrt{\mathbf{M}_{\mathfrak{c}}}\psi_1)+Q_{\mathfrak{c}}(\sqrt{\mathbf{M}_{\mathfrak{c}}}\psi_1,\partial \mathbf{M}_{\mathfrak{c}})+Q_{\mathfrak{c}}(\mathbf{M}_{\mathfrak{c}},\partial\sqrt{\mathbf{M}_{\mathfrak{c}}}\psi_1)+Q_{\mathfrak{c}}(\partial\sqrt{\mathbf{M}_{\mathfrak{c}}}\psi_1,\mathbf{M}_{\mathfrak{c}})\nonumber\\
	&\qquad +Q_{\mathfrak{c}}(\mathbf{M}_{\mathfrak{c}},\sqrt{\mathbf{M}_{\mathfrak{c}}}\partial \psi_1)+Q_{\mathfrak{c}}(\sqrt{\mathbf{M}_{\mathfrak{c}}}\partial \psi_1,\mathbf{M}_{\mathfrak{c}}),
\end{align*}
which yields that
\begin{align}\label{2.54}
	\mathbf{L}_{\mathfrak{c}}(\{\mathbf{I-P_{\mathfrak{c}}}\}\partial \psi_1)=\mathbf{L}_{\mathfrak{c}}\partial \psi_1&=-\frac{1}{\sqrt{\mathbf{M}_{\mathfrak{c}}}}\partial  g_1+\frac{1}{\sqrt{\mathbf{M}_{\mathfrak{c}}}}\Big\{Q_{\mathfrak{c}}(\partial  \mathbf{M}_{\mathfrak{c}},\sqrt{\mathbf{M}_{\mathfrak{c}}}\psi_1)+Q_{\mathfrak{c}}(\sqrt{\mathbf{M}_{\mathfrak{c}}}\psi_1,\partial \mathbf{M}_{\mathfrak{c}})\Big\}\nonumber\\
	&\qquad +\frac{1}{\sqrt{\mathbf{M}_{\mathfrak{c}}}}\Big\{Q_{\mathfrak{c}}(\mathbf{M}_{\mathfrak{c}},\partial\sqrt{\mathbf{M}_{\mathfrak{c}}}\psi_1)+Q_{\mathfrak{c}}(\partial\sqrt{\mathbf{M}_{\mathfrak{c}}}\psi_1,\mathbf{M}_{\mathfrak{c}})\Big\}.
\end{align}
Using the exponential decay of $\mathbf{L}_{\mathfrak{c}}^{-1}$ in Proposition \ref{prop5.4}, we have
\begin{align}\label{6.51-0}
	|\psi_1|=\Big|\{\mathbf{I}-\mathbf{P}_{\mathfrak{c}}\}\Big(\frac{F_1^{\mathfrak{c}}}{\sqrt{\mathbf{M}_{\mathfrak{c}}}}\Big)\Big|=\Big|\mathbf{L}_{\mathfrak{c}}^{-1}\Big(-\frac{1}{\sqrt{\mathbf{M}_{\mathfrak{c}}}}(\partial_t \mathbf{M}_{\mathfrak{c}}+\hat{p}\cdot \nabla_x \mathbf{M}_{\mathfrak{c}})\Big)\Big| \lesssim \mathbf{M}_{\mathfrak{c}}^{\frac{1}{2}-},
\end{align}
which, together with $|\partial g_1|\lesssim \mathbf{M}_{\mathfrak{c}}^{1-}$, yields that the RHS of \eqref{2.54} can be bounded by $\mathbf{M}_{\mathfrak{c}}^{\frac{1}{2}-}$. Using Proposition \ref{prop5.4} again, we obtain
\begin{align}\label{6.51-00}
	|\{\mathbf{I-P_{\mathfrak{c}}}\}\partial \psi_1|\lesssim \mathbf{M}_{\mathfrak{c}}^{\frac{1}{2}-}.
\end{align} 
On the other hand, it is clear that 
\begin{align*}
	|\mathbf{P}_{\mathfrak{c}}\partial \psi_1|\lesssim \mathbf{M}_{\mathfrak{c}}^{\frac{1}{2}-},
\end{align*} 
which, together with \eqref{6.51-00}, implies that
\begin{align*}
	|\partial \psi_1|\lesssim \mathbf{M}_{\mathfrak{c}}^{\frac{1}{2}-}.
\end{align*}
Similarly, one can deduce that
\begin{align}\label{6.56-0}
	|\partial^m \psi_1|\lesssim \mathbf{M}_{\mathfrak{c}}^{\frac{1}{2}-},\quad m\ge 1.
\end{align}

Next we consider the estimate on the macroscopic parts $(a_1,b_1,c_1)$. Using \eqref{6.56-0}, we get 
\begin{align*}
	\|\mathbf{S}_1\|_{H^{N_0-1}} \lesssim 1.
\end{align*}
One obtains from Lemma \ref{thm-re} that
\begin{align*}
\left\|\partial_t \mathbf{A}_0\right\|_{\infty} +\sum_{\alpha=0}^3\left\|\nabla_x\mathbf{A}_{\alpha}\right\|_{H^{N_0-1}}+\left\|\mathbf{B}\right\|_{H^{N_0-1}}\lesssim 1.
\end{align*}
Applying standard energy estimate, one gets
\begin{align*}
	\frac{d}{d t} \left\|\left(a_1, b_1, c_1\right)(t)\right\|_{H^{N_0-3}}^2
	& \lesssim \left\|\left(a_1, b_1, c_1\right)(t)\right\|_{H^{N_0-3}}^2+\left\|\left(a_1, b_1, c_1\right)(t)\right\|_{H^{N_0-3}},
\end{align*}
which, together with Gr\"{o}nwall's inequality, yields that
\begin{align}\label{4.46-0}
	\left\|\left(a_1, b_1, c_1\right)(t)\right\|_{H^{N_0-3}}\lesssim 1.
\end{align}
Hence it follows from \eqref{5.1-20} that
\begin{align*} 
	\Big|\mathbf{P}_{\mathfrak{c}}\Big(\frac{F_1^{\mathfrak{c}}}{\sqrt{\mathbf{M}_{\mathfrak{c}}}}\Big)\Big|\lesssim \mathbf{M}_{\mathfrak{c}}^{\frac{1}{2}-},
\end{align*}
which, together with \eqref{6.51-0}, yields that
\begin{align*}
	|F_1^{\mathfrak{c}}|\lesssim \mathbf{M}_{\mathfrak{c}}^{1-}. 
\end{align*}

For $\partial F_1^{\mathfrak{c}}$, on account of \eqref{6.46-0}, \eqref{6.56-0} and  \eqref{4.46-0}, one obtains
\begin{align*} 
	|\partial F_1^{\mathfrak{c}}|\lesssim \mathbf{M}_{\mathfrak{c}}^{1-}.
\end{align*}
Similar arguments lead to
\begin{align*}
	|\partial^m F_1^{\mathfrak{c}}|\lesssim \mathbf{M}_{\mathfrak{c}}^{1-},\quad m\ge 1.
\end{align*}
By induction, we can prove that
\begin{align*}
	|F_{n+1}^{\mathfrak{c}}|\lesssim \mathbf{M}_{\mathfrak{c}}^{1-},\quad |\partial^m F_{n+1}^{\mathfrak{c}}|\lesssim \mathbf{M}_{\mathfrak{c}}^{1-},\quad n=0,1,\cdots, 2k-2,\quad m\ge 1.
\end{align*}
Therefore the proof is completed.
\end{proof}

\section{Uniform in $\mathfrak{c}$ and $\varepsilon$ estimates on the remainder $F^{\varepsilon,\mathfrak{c}}_R$}
In this section, we shall prove our main results, Theorem \ref{thm1.1-0} and Theorem \ref{thm1.3-0}.
As in \cite{Guo2,Speck}, we define 
\begin{align}\label{7.1-0}
	f_R^{\varepsilon,\mathfrak{c}}(t,x,p)=\frac{F^{\varepsilon,\mathfrak{c}}_R(t,x,p)}{\sqrt{\mathbf{M}_{\mathfrak{c}}(t,x,p)}} 
\end{align}
and 
\begin{align}\label{7.2-0}
	h_R^{\varepsilon,\mathfrak{c}}(t,x,p)=\frac{F^{\varepsilon,\mathfrak{c}}_R(t,x,p)}{\sqrt{J_{\mathfrak{c}}(p)}}.
\end{align}
We first present two uniform-in-$\mathfrak{c}$ estimates on the nonlinear operators.

 \begin{Lemma}\label{lem7.3}
	It holds that
	\begin{align*} 
		\Big| \frac{w_{\ell}}{\sqrt{\mathbf{M}_{\mathfrak{c}}}}Q_{\mathfrak{c}}(h_1\sqrt{\mathbf{M}_{\mathfrak{c}}},h_2\sqrt{\mathbf{M}_{\mathfrak{c}}}) \Big|\lesssim \nu_{\mathfrak{c}}(p)\|h_1\|_{\infty,\ell}\|h_2\|_{\infty,\ell},
	\end{align*}	
    where the constant is independent of $\mathfrak{c}$.
\end{Lemma}

\begin{proof}
	Noting
	\begin{align*}
		p^0+q^0=p^{\prime 0}+q^{\prime 0},\quad p+q=p^{\prime}+q^{\prime },
	\end{align*}
	we claim that
	\begin{align}\label{7.10-0}
		|p|\lesssim |p'|+|q'|,\quad |q|\lesssim |p'|+|q'|.
	\end{align}
	Actually, without loss of generality, we may assume that $|p|\le |q|$. Denote $r:=\max\{|p'|,|q'|\}$, then one has
	\begin{align*}
		2\sqrt{\mathfrak{c}^2+|p|^2}\le \sqrt{\mathfrak{c}^2+|p|^2}+\sqrt{\mathfrak{c}^2+|q|^2}=\sqrt{\mathfrak{c}^2+|p'|^2}+\sqrt{\mathfrak{c}^2+|q'|^2}\le 2\sqrt{\mathfrak{c}^2+r^2},
	\end{align*}
	which yields that
	\begin{align*}
		|p|^2\le r^2\le |p'|^2+|q'|^2,
	\end{align*}
	Thus it holds 
	\begin{align}\label{6.6-10}
		|p|\le |p'|+|q'|.
	\end{align}
	
	If $|p|\le \frac{|q|}{2}$, one has $|p+q|\ge |q|-|p|\ge \frac{|q|}{2}$, which yields that
	\begin{align*}
		\frac{|q|}{2}\le |p+q|=|p'+q'|\le |p'|+|q'|.
	\end{align*}
	
	If $\frac{|q|}{2}\le |p|\le |q|$, it follows from \eqref{6.6-10} that 
	\begin{align*}
		|q|\le 2|p|\le 2(|p'|+|q'|).
	\end{align*}
	Hence the claim \eqref{7.10-0} holds.
	
	Now it follows from \eqref{7.10-0} that
	\begin{align*} 
		w_{\ell}(p)\lesssim w_{\ell}(p')w_{\ell}(q'),
	\end{align*}
	which, together with from \eqref{2.90}, yields that 
	\begin{align*}
		&\Big| \frac{w_{\ell}}{\sqrt{\mathbf{M}_{\mathfrak{c}}}}Q_{\mathfrak{c}}(h_1\sqrt{\mathbf{M}_{\mathfrak{c}}},h_2\sqrt{\mathbf{M}_{\mathfrak{c}}}) \Big|\nonumber\\
		&\le \frac{w_{\ell}(p)}{\sqrt{\mathbf{M}_{\mathfrak{c}}(p)}}\int_{\mathbb{R}^3}\int_{\mathbb{S}^2}v_\phi\Big|h_1(p')h_2(q')\sqrt{\mathbf{M}_{\mathfrak{c}}(p')\mathbf{M}_{\mathfrak{c}}(q')}-h_1(p)h_2(q)\sqrt{\mathbf{M}_{\mathfrak{c}}(p)\mathbf{M}_{\mathfrak{c}}(q)}\Big|d\omega dq\nonumber\\
		&\le \int_{\mathbb{R}^3}\int_{\mathbb{S}^2}v_\phi\Big[|w_{\ell}(p')h_1(p')|\cdot |w_{\ell}(q')h_2(q')|+|w_{\ell}(p)h_1(p)|\cdot |h(q)|\Big]\sqrt{\mathbf{M}_{\mathfrak{c}}(q)}d\omega dq\nonumber\\
		&\lesssim \nu_{\mathfrak{c}}(p)\|h_1\|_{\infty,\ell}\|h_2\|_{\infty,\ell}.
	\end{align*}
	Therefore the proof is completed.
\end{proof}

\begin{Lemma}\label{lem7.4} 
	For any $\ell \geq 9$, it holds that
	\begin{align*} 
		\left|\left\langle\Gamma_{\mathfrak{c}}\left(h_1, h_2\right), h_3\right\rangle \right| \lesssim \left\|h_3\right\|_{\infty, \ell}\left\|h_2\right\|_2\left\|h_1\right\|_2 .
	\end{align*}
	Furthermore, if $\chi(p)$ satisfies $|\chi(p)|\lesssim e^{-\delta_1|p|}$ for some positive constant $\delta_1>0$, then we have
	\begin{align*} 
		\left|\left\langle\Gamma_{\mathfrak{c}}\left(h_1, \chi\right), h_3\right\rangle \right|+\left|\left\langle\Gamma_{\mathfrak{c}}\left(\chi, h_1\right), h_3\right\rangle \right| \lesssim \left\|h_3\right\|_{\nu_{\mathfrak{c}}}\left\|h_1\right\|_{\nu_{\mathfrak{c}}},
	\end{align*}
   where the constants are independent of $\mathfrak{c}$. 
\end{Lemma}
    We point out that Lemma \ref{lem7.4} has been proved in \cite{Speck} when $\mathfrak{c}=1$. For the general case, the proof is very similar to the one in \cite{Speck} and we omit the details here for brevity. 

To establish the uniform in $\mathfrak{c}$ and $\varepsilon$ estimates for the remainder $F^{\varepsilon,\mathfrak{c}}_R$, we shall use the  $L^2-L^{\infty}$ framework from \cite{Guo1}. We first consider the $L^2$ estimate.

\begin{Lemma}[$L^2$ Estimate]\label{lem7.1}
	Let $(n_0(t, x), u(t, x), T_0(t, x))$ be the smooth solution to the relativistic Euler equations \eqref{4.8-0} generated by Lemma \ref{thm-re}. Let $\mathbf{M}_{\mathfrak{c}}(n_0, u, T_0 ; p)$, $f_R^{\varepsilon,\mathfrak{c}}$, $h_R^{\varepsilon,\mathfrak{c}}$ be defined in \eqref{1.13-00}, \eqref{7.1-0}  and \eqref{7.2-0}, respectively, and let $\zeta_0>0$ be the positive constant in Proposition \ref{prop4.13}. Then there exist constants $\varepsilon_0>0$ and $C>0$, such that for all $\varepsilon \in (0, \varepsilon_0]$, it holds
  \begin{align}\label{6.20-10}
    \frac{d}{d t}\left\|f_R^{\varepsilon,\mathfrak{c}}\right\|_2^2(t)+\frac{\zeta_0}{2 \varepsilon}\left\|\{\mathbf{I}-\mathbf{P}_{\mathfrak{c}}\} f_R^{\varepsilon,\mathfrak{c}}\right\|_{\nu_{\mathfrak{c}}}^2(t) \leq C\Big\{\sqrt{\varepsilon}\|\varepsilon^{\frac{3}{2}} h_R^{\varepsilon,\mathfrak{c}}\|_{\infty, \ell}(t)+1\Big\}\left\{\left\|f_R^{\varepsilon,\mathfrak{c}}\right\|_2^2+\left\|f_R^{\varepsilon,\mathfrak{c}}\right\|_2\right\},
  \end{align}
  where the constant $C$ depends upon the $L^2$ norms and the $L^{\infty}$ norms of the terms $\mathbf{M}_{\mathfrak{c}}, F_1^{\mathfrak{c}}, \ldots, F^{\mathfrak{c}}_{2k-1}$ as well as their first derivatives, and $C$ is independent of $\mathfrak{c}$.
\end{Lemma}

   \begin{proof}
   	Plugging $F^{\varepsilon,\mathfrak{c}}_R=f_R^{\varepsilon,\mathfrak{c}}\sqrt{\mathbf{M}_{\mathfrak{c}}}$ into \eqref{1.19-0}, one has
   	\begin{align}\label{7.25-0}
   		& \partial_t f_R^{\varepsilon,\mathfrak{c}}+\hat{p} \cdot \nabla_x f_R^{\varepsilon,\mathfrak{c}}+\frac{1}{\varepsilon}\mathbf{L}_{\mathfrak{c}}f_R^{\varepsilon,\mathfrak{c}}=-\frac{\{\partial_t+\hat{p}\cdot \nabla_x\}\sqrt{\mathbf{M}_{\mathfrak{c}}}}{\sqrt{\mathbf{M}_{\mathfrak{c}}}}f_R^{\varepsilon,\mathfrak{c}}+\varepsilon^{k-1}\Gamma_{\mathfrak{c}}(f_R^{\varepsilon,\mathfrak{c}},f_R^{\varepsilon,\mathfrak{c}})\nonumber\\
   		&\quad +\sum_{i=1}^{2k-1} \varepsilon^{i-1}\Big\{\Gamma_{\mathfrak{c}}\Big(\frac{F_i^{\mathfrak{c}}}{\sqrt{\mathbf{M}_{\mathfrak{c}}}},f_R^{\varepsilon,\mathfrak{c}}\Big)+\Gamma_{\mathfrak{c}}\Big(f_R^{\varepsilon,\mathfrak{c}},\frac{F_i^{\mathfrak{c}}}{\sqrt{\mathbf{M}_{\mathfrak{c}}}}\Big)\Big\}+\varepsilon^k \bar{A},
   	\end{align}
   	where
   	\begin{align*}
   		\bar{A}: =\sum_{\substack{i+j\ge 2k+1 \\ 2 \leq i, j \leq 2k-1}} \varepsilon^{i+j-1-2k} \Gamma_{\mathfrak{c}}\Big(\frac{F_i^{\mathfrak{c}}}{\sqrt{\mathbf{M}_{\mathfrak{c}}}},\frac{F_i^{\mathfrak{c}}}{\sqrt{\mathbf{M}_{\mathfrak{c}}}}\Big).
   	\end{align*}
    Multiplying \eqref{7.25-0} by $f_R^{\varepsilon,\mathfrak{c}}$ and integrating over $\mathbb{R}^3\times \mathbb{R}^3$, one has 
    \begin{align*}
    	 &\big\langle\partial_t f_R^{\varepsilon,\mathfrak{c}}+\hat{p} \cdot \nabla_x f_R^{\varepsilon,\mathfrak{c}}+\frac{1}{\varepsilon} \mathbf{L}_{\mathfrak{c}}f_R^{\varepsilon,\mathfrak{c}}, f_R^{\varepsilon,\mathfrak{c}}\big\rangle=-\Big\langle\Big(\frac{\left\{\partial_t+\hat{p} \cdot \nabla_x\right\} \sqrt{\mathbf{M}_{\mathfrak{c}}}}{\sqrt{\mathbf{M}_{\mathfrak{c}}}}\Big) f_R^{\varepsilon,\mathfrak{c}}, f_R^{\varepsilon,\mathfrak{c}}\Big\rangle+\langle \varepsilon^{k-1}\Gamma_{\mathfrak{c}}(f_R^{\varepsilon,\mathfrak{c}},f_R^{\varepsilon,\mathfrak{c}}),f_R^{\varepsilon,\mathfrak{c}}\rangle\nonumber\\
    	 &\qquad \qquad +\Big\langle \sum_{i=1}^{2k-1} \varepsilon^{i-1}\Big\{\Gamma_{\mathfrak{c}}\Big(\frac{F_i^{\mathfrak{c}}}{\sqrt{\mathbf{M}_{\mathfrak{c}}}},f_R^{\varepsilon,\mathfrak{c}}\Big)+\Gamma_{\mathfrak{c}}\Big(f_R^{\varepsilon,\mathfrak{c}},\frac{F_i^{\mathfrak{c}}}{\sqrt{\mathbf{M}_{\mathfrak{c}}}}\Big)\Big\}, f_R^{\varepsilon,\mathfrak{c}}\Big\rangle+\langle \varepsilon^k \bar{A}, f_R^{\varepsilon,\mathfrak{c}}\rangle. 
    \end{align*}
   	It follows from Proposition \ref{prop4.13} that
   	\begin{align*}
   	  \big\langle\partial_t f_R^{\varepsilon,\mathfrak{c}}+\hat{p} \cdot \nabla_x f_R^{\varepsilon,\mathfrak{c}}+\frac{1}{\varepsilon} \mathbf{L}_{\mathfrak{c}}f_R^{\varepsilon,\mathfrak{c}}, f_R^{\varepsilon,\mathfrak{c}}\big\rangle \geq \frac{1}{2} \frac{d}{d t}\left\|f_R^{\varepsilon,\mathfrak{c}}\right\|_2^2+\frac{\zeta_0}{\varepsilon}\left\|\{\mathbf{I}-\mathbf{P}_{\mathfrak{c}}\} f_R^{\varepsilon,\mathfrak{c}}\right\|_{\nu_{\mathfrak{c}}}^2 .
   \end{align*}
   	
   	For $\partial=\partial_t$ or $\partial=\partial_{x_i}$, it holds that
   	\begin{align}\label{6.25-10}
   		&\frac{\partial \mathbf{M}_{\mathfrak{c}}}{\mathbf{M}_{\mathfrak{c}}}=\frac{\partial n_0}{n_0}-3\frac{\partial T_0}{T_0}+\frac{\partial T_0}{T^2_0}\Big(u^0p^0-\mathfrak{c}^2\frac{K_1(\gamma)}{K_2(\gamma)}\Big)-\frac{\partial T_0}{T^2_0}\sum_{i=1}^3u_ip_i +\frac{1}{T_0}\Big(\sum_{i=1}^3p_i\partial u_i-\frac{\partial u\cdot u}{u^0}p^0\Big).
   	\end{align}
   A direct calculation shows that $$\Big|u^0p^0-\mathfrak{c}^2\frac{K_1(\gamma)}{K_2(\gamma)}\Big|\lesssim (1+|p|)^2 C(n_0, u, T_0),$$
   which, together with \eqref{6.25-10}, yields that
   	\begin{align*}
   	\Big|\frac{\left\{\partial_t+\hat{p} \cdot \nabla_x\right\} \sqrt{\mathbf{M}_{\mathfrak{c}}}}{\sqrt{\mathbf{M}_{\mathfrak{c}}}}\Big| \lesssim  (1+|p|)^3 C(n_0, u, T_0).
   \end{align*}
   	For any $0<\sqrt{\varepsilon}\le \kappa$, we obtain
   		\begin{align*}
   		&\Big|\Big\langle\Big(\frac{\left\{\partial_t+\hat{p} \cdot \nabla_x\right\} \sqrt{\mathbf{M}_{\mathfrak{c}}}}{\sqrt{\mathbf{M}_{\mathfrak{c}}}}\Big) f_R^{\varepsilon,\mathfrak{c}}, f_R^{\varepsilon,\mathfrak{c}}\Big\rangle \Big| \nonumber\\
   		&\le \Big|\int_{\{1+|p| \geq \frac{\kappa}{\sqrt{\varepsilon}}\}} dxdp\Big|+\Big|\int_{\{1+|p| \le \frac{\kappa}{\sqrt{\varepsilon}}\}} dxdp \Big| \nonumber\\
   		&\leq C_{\kappa}\varepsilon^2 \|\nabla_x(n_0,u,T_0)\|_{2} \cdot \|h_R^{\varepsilon,\mathfrak{c}}\|_{\infty,\ell} \cdot  \|f_R^{\varepsilon,\mathfrak{c}}\|_2\nonumber\\
   		&\qquad +C \|\nabla_x(n_0,u,T_0)\|_{L^{\infty}}\cdot \|(1+|p|)^{\frac{3}{2}}f_R^{\varepsilon,\mathfrak{c}}\mathbf{1}_{\{1+|p| \le \frac{\kappa}{\sqrt{\varepsilon}}\}}\|_2^2\nonumber\\
   		&\leq C_{\kappa}\varepsilon^2 \|h_R^{\varepsilon,\mathfrak{c}}\|_{\infty,\ell}\cdot \|f_R^{\varepsilon,\mathfrak{c}}\|_2+C\|(1+|p|)^{\frac{3}{2}}\mathbf{P}_{\mathfrak{c}}f_R^{\varepsilon,\mathfrak{c}}\mathbf{1}_{\{1+|p| \le \frac{\kappa}{\sqrt{\varepsilon}}\}}\|_2^2\nonumber\\
   		&\qquad +C\|(1+|p|)^{\frac{3}{2}}\{\mathbf{I-P_{\mathfrak{c}}}\}f_R^{\varepsilon,\mathfrak{c}}\mathbf{1}_{\{1+|p| \le \frac{\kappa}{\sqrt{\varepsilon}}\}}\|_2^2\nonumber\\
   		&\leq C_{\kappa}\varepsilon^2 \|h_R^{\varepsilon,\mathfrak{c}}\|_{\infty,\ell}\cdot \|f_R^{\varepsilon,\mathfrak{c}}\|_2+C\|f_R^{\varepsilon,\mathfrak{c}}\|^2_2+\frac{C\kappa^2}{\varepsilon}\|\{\mathbf{I-P_{\mathfrak{c}}}\}f_R^{\varepsilon,\mathfrak{c}}\|_{\nu_{\mathfrak{c}}}^2.
   	\end{align*}
   
    It follows from Lemma \ref{lem7.4} that
    \begin{align*}
    	|\langle \varepsilon^{k-1}\Gamma_{\mathfrak{c}}(f_R^{\varepsilon,\mathfrak{c}},f_R^{\varepsilon,\mathfrak{c}}),f_R^{\varepsilon,\mathfrak{c}}\rangle|\lesssim \varepsilon^{k-1} \|f_R^{\varepsilon,\mathfrak{c}}\|_{\infty,\ell}\cdot \|f_R^{\varepsilon,\mathfrak{c}}\|^2_2\lesssim \varepsilon^{k-1} \|h_R^{\varepsilon,\mathfrak{c}}\|_{\infty,\ell}\cdot \|f_R^{\varepsilon,\mathfrak{c}}\|^2_2
    \end{align*}
   	and 
   	\begin{align*}
   		&\Big| \Big\langle \sum_{i=1}^{2k-1} \varepsilon^{i-1}\Big\{\Gamma_{\mathfrak{c}}\Big(\frac{F_i^{\mathfrak{c}}}{\sqrt{\mathbf{M}_{\mathfrak{c}}}},f_R^{\varepsilon,\mathfrak{c}}\Big)+\Gamma_{\mathfrak{c}}\Big(f_R^{\varepsilon,\mathfrak{c}},\frac{F_i^{\mathfrak{c}}}{\sqrt{\mathbf{M}_{\mathfrak{c}}}}\Big)\Big\},f_R^{\varepsilon,\mathfrak{c}} \Big \rangle \Big|\nonumber\\
   		&\lesssim \sum_{i=1}^{2k-1} \varepsilon^{i-1}\|f_R^{\varepsilon,\mathfrak{c}}\|^2_{\nu_{\mathfrak{c}}}\lesssim  \|\mathbf{P}_{\mathfrak{c}}f_R^{\varepsilon,\mathfrak{c}}\|^2_{\nu_{\mathfrak{c}}}+ \|\{\mathbf{I-P_{\mathfrak{c}}}\}f_R^{\varepsilon,\mathfrak{c}}\|^2_{\nu_{\mathfrak{c}}}\nonumber\\
   		&\lesssim \|f_R^{\varepsilon,\mathfrak{c}}\|^2_{2}+ \|\{\mathbf{I-P_{\mathfrak{c}}}\}f_R^{\varepsilon,\mathfrak{c}}\|^2_{\nu_{\mathfrak{c}}}.
   	\end{align*}
   
   	 Similarly, for the last term, one has
   	\begin{align*}
   		\Big| \Big\langle\varepsilon^k \bar{A}, f_R^{\varepsilon,\mathfrak{c}} \Big\rangle \Big| &\lesssim \varepsilon^k\sum_{\substack{i+j\ge 2k+1 \\ 2 \leq i, j \leq 2k-1}} \varepsilon^{i+j-1-2k} \Big| \Big\langle \Gamma_{\mathfrak{c}}\Big(\frac{F_i^{\mathfrak{c}}}{\sqrt{\mathbf{M}_{\mathfrak{c}}}},\frac{F_i^{\mathfrak{c}}}{\sqrt{\mathbf{M}_{\mathfrak{c}}}}\Big), f_R^{\varepsilon,\mathfrak{c}} \Big\rangle \Big| \nonumber\\
   		&\lesssim \varepsilon^k\|f_R^{\varepsilon,\mathfrak{c}}\|_2\lesssim \|f_R^{\varepsilon,\mathfrak{c}}\|_2.
   	\end{align*}
   	Collecting all the above estimates, one has
   	\begin{align*}
   		&\frac{1}{2} \frac{d}{d t}\left\|f_R^{\varepsilon,\mathfrak{c}}\right\|_2^2+\frac{\zeta_0}{\varepsilon}\left\|\{\mathbf{I}-\mathbf{P}_{\mathfrak{c}}\} f_R^{\varepsilon,\mathfrak{c}}\right\|_{\nu_{\mathfrak{c}}}^2 \le C_{\kappa}\varepsilon^2 \|h_R^{\varepsilon,\mathfrak{c}}\|_{\infty,\ell}\cdot \|f_R^{\varepsilon,\mathfrak{c}}\|_2+C\|f_R^{\varepsilon,\mathfrak{c}}\|^2_2+C\|f_R^{\varepsilon,\mathfrak{c}}\|_2\nonumber\\
   		&\qquad \qquad \qquad +C\Big(\frac{\kappa^2}{\varepsilon}+1\Big)\|\{\mathbf{I-P_{\mathfrak{c}}}\}f_R^{\varepsilon,\mathfrak{c}}\|_{\nu_{\mathfrak{c}}}^2+C\varepsilon^{k-1} \|h_R^{\varepsilon,\mathfrak{c}}\|_{\infty,\ell}\cdot \|f_R^{\varepsilon,\mathfrak{c}}\|^2_2.
   	\end{align*}
   We choose $\kappa= \sqrt{\frac{\zeta_0}{4C}}$, then we suppose that $0<\varepsilon \le \varepsilon_0\le  \frac{\zeta_0}{4C}$. Thus one gets \eqref{6.20-10}. Therefore the proof is completed.
   \end{proof}

Next we consider the $L^{\infty}$ estimate for $h_R^{\varepsilon,\mathfrak{c}}$. Recall $J_{\mathfrak{c}}(p)$ in \eqref{1.50-0}. We define
\begin{align*}
	\mathcal{L}_{\mathfrak{c}}h:=-J_{\mathfrak{c}}^{-\frac{1}{2}}\{Q_{\mathfrak{c}}(\mathbf{M}_{\mathfrak{c}}, \sqrt{J_{\mathfrak{c}}} h)+Q_{\mathfrak{c}}(\sqrt{J_{\mathfrak{c}}} h, \mathbf{M}_{\mathfrak{c}})\}=\nu_{\mathfrak{c}} h-\mathcal{K}_{\mathfrak{c}}h,
\end{align*}
where $\mathcal{K}_{\mathfrak{c}}=\mathcal{K}_{\mathfrak{c}2}-\mathcal{K}_{\mathfrak{c}1}$. More specifically, $\nu_{\mathfrak{c}}$ is defined in \eqref{1.34-00} and operators $\mathcal{K}_{\mathfrak{c}1}h$ and $\mathcal{K}_{\mathfrak{c}2}h$ are defined as
\begin{align*}
	\mathcal{K}_{\mathfrak{c}1}h &:=J_{\mathfrak{c}}^{-\frac{1}{2}} Q_{\mathfrak{c}}^{-}(\mathbf{M}_{\mathfrak{c}}, \sqrt{J_{\mathfrak{c}}} h) 
	=\int_{\mathbb{R}^3}\int_{\mathbb{S}^2} v_\phi \Big\{\sqrt{J_{\mathfrak{c}}(q)} \frac{\mathbf{M}_{\mathfrak{c}}(p)}{\sqrt{J_{\mathfrak{c}}(p)}} h(q)\Big\}d \omega dq , \\
    \mathcal{K}_{\mathfrak{c}2}h&:=J_{\mathfrak{c}}^{-\frac{1}{2}}\left\{Q_{\mathfrak{c}}^{+}(\mathbf{M}_{\mathfrak{c}}, \sqrt{J_{\mathfrak{c}}} h)+Q_{\mathfrak{c}}^{+}(\sqrt{J_{\mathfrak{c}}} h, \mathbf{M}_{\mathfrak{c}})\right\}\nonumber\\
	&=\int_{\mathbb{R}^3}\int_{\mathbb{S}^2}  v_\phi \Big\{\mathbf{M}_{\mathfrak{c}}(p') \frac{\sqrt{J_{\mathfrak{c}}(q')}}{\sqrt{J_{\mathfrak{c}}(p)}} h(q')\Big\}d \omega dq+\int_{\mathbb{R}^3}\int_{\mathbb{S}^2}  v_\phi \Big\{\mathbf{M}_{\mathfrak{c}}(q') \frac{\sqrt{J_{\mathfrak{c}}(p')}}{\sqrt{J_{\mathfrak{c}}(p)}} h(p')\Big\}d \omega dq.
\end{align*}
    Noting \eqref{1.53-0}, by similar arguments as in \cite{Strain}, one can show that
    \begin{align*}
    	|\mathcal{K}_{\mathfrak{c}i}(h)|\lesssim \int_{\mathbb{R}^3}\hat{k}_i(p,q)|h(q)|dq,\quad i=1,2,
    \end{align*}
   where
   \begin{align*}
   	\hat{k}_1(p,q)=|p-q|e^{-\delta_2|p|}e^{-\delta_2|q|},\quad 
    \hat{k}_2(p,q)=\frac{1}{|p-q|}e^{-\frac{\delta_2}{2}|p-q|}
   \end{align*}
   with $\delta_2:=\alpha-\frac{1}{2}>0$. We denote $\hat{k}(p,q):=\hat{k}_1(p,q)+\hat{k}_2(p,q)$. Then it holds that
   \begin{align*}
   	|\mathcal{K}_{\mathfrak{c}}(h)|\lesssim \int_{\mathbb{R}^3}\hat{k}(p,q)|h(q)|dq,\quad i=1,2.
   \end{align*}
    Denote 
   \begin{align*}
   \hat{k}_w(p,q):=\hat{k}(p,q)\frac{w_{\ell}(p)}{w_{\ell}(q)}.
   \end{align*}
   By similar arguments as in Lemmas \ref{lem4.4-00}-\ref{lem2.8}, one has
   \begin{align}\label{6.37-10}
   	\int_{\mathbb{R}^3}\hat{k}_w(p,q)e^{\frac{\delta_2}{4}|p-q|}dq+\int_{\mathbb{R}^3}\hat{k}^2_w(p,q)dq\lesssim \max{\Big\{\frac{1}{\mathfrak{c}},\frac{1}{1+|p|}\Big\}}.
   \end{align}
   For later use, we introduce
   \begin{align*}
   	   \widehat{\nu}_{\mathfrak{c}}(p):=\int_{\mathbb{R}^3}\int_{\mathbb{S}^2}v_{\phi}J_{\mathfrak{c}}(q)d\omega dq\cong \nu_{\mathfrak{c}}(p).
   \end{align*}
   
   \begin{Lemma}[$L^{\infty}$ Estimate]\label{lem7.2}
   	Under the assumptions of Lemma \ref{lem7.1}, there exist $\varepsilon_0>0$ and a positive constant $C>0$, such that for all $\varepsilon \in (0, \varepsilon_0]$ and for any $\ell \geq 9$, it holds that
   	\begin{align*} 
   		\sup_{0 \leq s \leq T}\|\varepsilon^{\frac{3}{2}}h_R^{\varepsilon,\mathfrak{c}}(s)\|_{\infty,\ell} \leq C\Big(\|\varepsilon^{\frac{3}{2}}h_0\|_{\infty,\ell}+ \sup_{0 \leq s \leq T}\|f_R^{\varepsilon,\mathfrak{c}}(s)\|_2+ \varepsilon^{k+\frac{5}{2}}\Big),
   	\end{align*}
   where $C$ is independent of $\mathfrak{c}$.
   \end{Lemma}

   \begin{proof} 
   	   	Plugging $F^{\varepsilon,\mathfrak{c}}_R=h_R^{\varepsilon,\mathfrak{c}}\sqrt{J_{\mathfrak{c}}}$ into \eqref{1.19-0}, one has
   	   \begin{align}\label{6.42-10}
   	   	& \partial_t h_R^{\varepsilon,\mathfrak{c}}+\hat{p} \cdot \nabla_x h_R^{\varepsilon,\mathfrak{c}}+\frac{\nu_{\mathfrak{c}}}{\varepsilon}h_R^{\varepsilon,\mathfrak{c}}=\frac{1}{\varepsilon}\mathcal{K}(h_R^{\varepsilon,\mathfrak{c}})+\varepsilon^{k-1}Q_{\mathfrak{c}}(h_R^{\varepsilon,\mathfrak{c}}\sqrt{J_{\mathfrak{c}}},\sqrt{J_{\mathfrak{c}}}h_R^{\varepsilon,\mathfrak{c}})\nonumber\\
   	   	&\quad +\sum_{i=1}^{2k-1} \varepsilon^{i-1}\frac{1}{\sqrt{J_{\mathfrak{c}}}}\Big\{Q_{\mathfrak{c}}(F_i^{\mathfrak{c}},\sqrt{J_{\mathfrak{c}}}h_R^{\varepsilon,\mathfrak{c}})+Q_{\mathfrak{c}}(\sqrt{J_{\mathfrak{c}}}h_R^{\varepsilon,\mathfrak{c}},F_i^{\mathfrak{c}})\Big\}+\varepsilon^k \tilde{A},
   	   \end{align}
   	   where
   	   \begin{align*}
   	   	\tilde{A}: =\sum_{\substack{i+j\ge 2k+1 \\ 2 \leq i, j \leq 2k-1}} \varepsilon^{i+j-1-2k}\frac{1}{\sqrt{J_{\mathfrak{c}}}}Q_{\mathfrak{c}}(F_i^{\mathfrak{c}},F_i^{\mathfrak{c}}).
   	   \end{align*}
   	  Denote $y_1:=x-\hat{p}(t-s)$ and 
   	  \begin{align*}
   	  	\tilde{\nu_{\mathfrak{c}}}(t,s):=\int_s^t\nu_{\mathfrak{c}}(\mathbf{M}_{\mathfrak{c}})(\tau,x-\hat{p}(t-\tau),p)d\tau \cong (t-s)\widehat{\nu_{\mathfrak{c}}}.
   	  \end{align*}
   	  Integrating \eqref{6.42-10} along the backward trajectory, one has
   	  \begin{align}\label{7.44-0}
   	  	&h_R^{\varepsilon,\mathfrak{c}}(t,x,p)\nonumber\\
   	  	&=\exp{\Big(-\frac{\tilde{\nu_{\mathfrak{c}}}(t,0)}{\varepsilon}\Big)} h_0(x-\hat{p}t,p)\nonumber\\
   	  	&\qquad +\frac{1}{\varepsilon}\int_{0}^t \exp{\Big(-\frac{\tilde{\nu_{\mathfrak{c}}}(t,s)}{\varepsilon}\Big)}\mathcal{K}_{\mathfrak{c}}h_R^{\varepsilon,\mathfrak{c}}(s,y_1,p)ds\nonumber\\
   	  	&\qquad +\int_{0}^t \exp{\Big(-\frac{\tilde{\nu_{\mathfrak{c}}}(t,s)}{\varepsilon}\Big)}\frac{\varepsilon^{k-1}}{\sqrt{J_{\mathfrak{c}}}}Q_{\mathfrak{c}}(h_R^{\varepsilon,\mathfrak{c}}\sqrt{J_{\mathfrak{c}}},h_R^{\varepsilon,\mathfrak{c}}\sqrt{J_{\mathfrak{c}}})(s,y_1,p)ds\nonumber\\
   	  	&\qquad +\int_{0}^t \exp{\Big(-\frac{\tilde{\nu_{\mathfrak{c}}}(t,s)}{\varepsilon}\Big)}\sum_{i=1}^{2k-1} \varepsilon^{i-1}\frac{1}{\sqrt{J_{\mathfrak{c}}}}\Big\{Q_{\mathfrak{c}}(F_i^{\mathfrak{c}},\sqrt{J_{\mathfrak{c}}}h_R^{\varepsilon,\mathfrak{c}})+Q_{\mathfrak{c}}(\sqrt{J_{\mathfrak{c}}}h_R^{\varepsilon,\mathfrak{c}},F_i^{\mathfrak{c}})\Big\}(s,y_1,p)ds\nonumber\\
   	  	&\qquad +\int_{0}^t \exp{\Big(-\frac{\tilde{\nu_{\mathfrak{c}}}(t,s)}{\varepsilon}\Big)}\varepsilon^k \tilde{A}(s,y_1,p)ds\nonumber\\
   	  	&=\sum_{j=1}^5\mathcal{J}_j.
   	  \end{align}
   	  It is clear that
   	  \begin{align*} 
   	  	|\varepsilon^{\frac{3}{2}}w_{\ell}\mathcal{J}_1|\le \|\varepsilon^{\frac{3}{2}}h_0\|_{\infty,\ell}.
   	  \end{align*}
     
   	  For $\mathcal{J}_3$, it follows from Lemma \ref{lem7.3} that
   	  \begin{align*} 
   	  		|\varepsilon^{\frac{3}{2}}w_{\ell}\mathcal{J}_3|&\lesssim \varepsilon^{k+\frac{1}{2}}\int_{0}^t \exp{\Big(-\frac{\tilde{\nu_{\mathfrak{c}}}(t,s)}{\varepsilon}\Big)}\Big|\frac{w_{\ell}}{\sqrt{J_{\mathfrak{c}}}}Q_{\mathfrak{c}}(h_R^{\varepsilon,\mathfrak{c}}\sqrt{J_{\mathfrak{c}}},h_R^{\varepsilon,\mathfrak{c}}\sqrt{J_{\mathfrak{c}}})(s,y_1,p)\Big|ds\nonumber\\
   	  		&\lesssim \varepsilon^{k-\frac{5}{2}}\int_{0}^t \exp{\Big(-\frac{\tilde{\nu_{\mathfrak{c}}}(t,s)}{\varepsilon}\Big)}\widehat{\nu_{\mathfrak{c}}}(p)ds\cdot \sup_{0 \leq s \leq T}\|\varepsilon^{\frac{3}{2}}h_R^{\varepsilon,\mathfrak{c}}(s)\|^2_{\infty,\ell}\nonumber\\
   	  		&\lesssim \varepsilon^{k-\frac{3}{2}}\sup_{0 \leq s \leq T}\|\varepsilon^{\frac{3}{2}}h_R^{\varepsilon,\mathfrak{c}}(s)\|^2_{\infty,\ell}.
   	  \end{align*}
     
   	   Similarly, we have
   	   \begin{align*} 
   	   	&|\varepsilon^{\frac{3}{2}}w_{\ell}\mathcal{J}_4|\nonumber\\
   	   	&\lesssim \varepsilon^{\frac{3}{2}}\sum_{i=1}^{2k-1} \varepsilon^{i-1}\int_{0}^t \exp{\Big(-\frac{\tilde{\nu_{\mathfrak{c}}}(t,s)}{\varepsilon}\Big)}\Big|\frac{w_{\ell}}{\sqrt{J_{\mathfrak{c}}}}\Big\{Q_{\mathfrak{c}}(F_i^{\mathfrak{c}},\sqrt{J_{\mathfrak{c}}}h_R^{\varepsilon,\mathfrak{c}})+Q_{\mathfrak{c}}(\sqrt{J_{\mathfrak{c}}}h_R^{\varepsilon,\mathfrak{c}},F_i^{\mathfrak{c}})\Big\}(s,y_1,p)\Big|ds\nonumber\\
   	   	&\lesssim \sum_{i=1}^{2k-1} \varepsilon^{i-1}\int_{0}^t \exp{\Big(-\frac{\tilde{\nu_{\mathfrak{c}}}(t,s)}{\varepsilon}\Big)}\widehat{\nu_{\mathfrak{c}}}(p)ds\cdot \sup_{0 \leq s \leq T}\|\varepsilon^{\frac{3}{2}}h_R^{\varepsilon,\mathfrak{c}}(s)\|_{\infty,\ell} \cdot \sup_{0 \leq s \leq T}\Big\|\frac{F_i^{\mathfrak{c}}(s)}{\sqrt{J_{\mathfrak{c}}}}\Big\|_{\infty,\ell}\nonumber\\
   	   	&\lesssim \varepsilon \sup_{0 \leq s \leq T}\|\varepsilon^{\frac{3}{2}}h_R^{\varepsilon,\mathfrak{c}}(s)\|_{\infty,\ell}
   	   \end{align*}
   	   and
   	   \begin{align*} 
   	   	|\varepsilon^{\frac{3}{2}}w_{\ell}\mathcal{J}_5|&\lesssim \varepsilon^{k+\frac{3}{2}}\sum_{\substack{i+j\ge 2k+1 \\ 2 \leq i, j \leq 2k-1}} \varepsilon^{i+j-1-2k}\int_{0}^t \exp{\Big(-\frac{\tilde{\nu_{\mathfrak{c}}}(t,s)}{\varepsilon}\Big)}\Big|\frac{w_{\ell}}{\sqrt{J_{\mathfrak{c}}}}Q_{\mathfrak{c}}(F_i^{\mathfrak{c}},F_i^{\mathfrak{c}})\Big|ds\nonumber\\
   	   	&\lesssim \varepsilon^{k+\frac{3}{2}} \int_{0}^t \exp{\Big(-\frac{\tilde{\nu_{\mathfrak{c}}}(t,s)}{\varepsilon}\Big)}\widehat{\nu_{\mathfrak{c}}}(p)ds\cdot  \sup_{0 \leq s \leq T}\Big\|\frac{F_i^{\mathfrak{c}}(s)}{\sqrt{J_{\mathfrak{c}}}}\Big\|_{\infty,\ell}\cdot \sup_{0 \leq s \leq T}\Big\|\frac{F_i^{\mathfrak{c}}(s)}{\sqrt{J_{\mathfrak{c}}}}\Big\|_{\infty,\ell}\nonumber\\
   	   	&\lesssim \varepsilon^{k+\frac{5}{2}}.
   	   \end{align*}
   	   Collecting the above estimates, we have established
   	   \begin{align}\label{7.50-0}
   	   	\sup_{0 \leq s \leq T}\|\varepsilon^{\frac{3}{2}}h_R^{\varepsilon,\mathfrak{c}}(s)\|_{\infty,\ell} &\leq  C \varepsilon \sup_{0 \leq s \leq T}\|\varepsilon^{\frac{3}{2}}h_R^{\varepsilon,\mathfrak{c}}(s)\|_{\infty,\ell}+C \varepsilon^{k-\frac{3}{2}} \sup_{0 \leq s \leq T}\|\varepsilon^{\frac{3}{2}}h_R^{\varepsilon,\mathfrak{c}}(s)\|_{\infty,\ell}^2 \nonumber\\
   	   	&\qquad +C \varepsilon^{k+\frac{5}{2}}+C\|\varepsilon^{\frac{3}{2}}h_0\|_{\infty,\ell}+C w_{\ell}(p) \varepsilon^{\frac{3}{2}} |\mathcal{J}_2| .
   	   \end{align}
      
   	   To bound the last term $\mathcal{J}_2$, we denote $y_2: =y_1-\hat{q}\left(s-s^{\prime}\right)= x-\hat{p}(t-s)-\hat{q}\left(s-s^{\prime}\right)$ and 
   	    \begin{align*}
   	   	\tilde{\nu_{\mathfrak{c}}}'(s,s'):=\int_{s'}^s\nu_{\mathfrak{c}}(\mathbf{M}_{\mathfrak{c}})(\tau,y_1-\hat{q}(s-\tau),q)d\tau\cong (s-s')\widehat{\nu_{\mathfrak{c}}}.
   	   \end{align*}
   	   We substitute \eqref{7.44-0} into $\mathcal{J}_2$ to obtain
   	    \begin{align*} 
   	   	|\mathcal{J}_2|&\lesssim \frac{1}{\varepsilon}\int_{0}^t \exp{\Big(-\frac{\tilde{\nu_{\mathfrak{c}}}(t,s)}{\varepsilon}\Big)}ds\int_{\mathbb{R}^3}\hat{k}(p,q)|h_R^{\varepsilon,\mathfrak{c}}(s,y_1,q)|dq\nonumber\\
   	   	&\lesssim \frac{1}{\varepsilon}\int_{0}^t \exp{\Big(-\frac{\tilde{\nu_{\mathfrak{c}}}(t,s)}{\varepsilon}\Big)}ds\int_{\mathbb{R}^3}\hat{k}(p,q)\exp{\Big(-\frac{\tilde{\nu_{\mathfrak{c}}}'(s,0)}{\varepsilon}\Big)} |h_0(y_1-\hat{q}s,q)|dq\nonumber\\
   	   	&\quad +\frac{1}{\varepsilon^2}\int_{0}^t \exp{\Big(-\frac{\tilde{\nu_{\mathfrak{c}}}(t,s)}{\varepsilon}\Big)}ds\int_{0}^s \exp{\Big(-\frac{\tilde{\nu_{\mathfrak{c}}}'(s,s')}{\varepsilon}\Big)}ds'\int_{\mathbb{R}^3}\hat{k}(p,q)\big|\mathcal{K}_{\mathfrak{c}}h_R^{\varepsilon,\mathfrak{c}}(s',y_2,q)\big|dq\nonumber\\
   	   	&\quad +\varepsilon^{k-2}\int_{0}^t \exp{\Big(-\frac{\tilde{\nu_{\mathfrak{c}}}(t,s)}{\varepsilon}\Big)}ds\int_{0}^s \exp{\Big(-\frac{\tilde{\nu_{\mathfrak{c}}}'(s,s')}{\varepsilon}\Big)}ds'\nonumber\\
   	   	&\qquad \qquad \times \int_{\mathbb{R}^3}\hat{k}(p,q)\frac{1}{\sqrt{J_{\mathfrak{c}}}}\Big|Q_{\mathfrak{c}}(h_R^{\varepsilon,\mathfrak{c}}\sqrt{J_{\mathfrak{c}}},h_R^{\varepsilon,\mathfrak{c}}\sqrt{J_{\mathfrak{c}}})(s',y_2,q)\Big|dq\nonumber\\
   	   	&\quad +\frac{1}{\varepsilon}\sum_{i=1}^{2k-1} \varepsilon^{i-1}\int_{0}^t \exp{\Big(-\frac{\tilde{\nu_{\mathfrak{c}}}(t,s)}{\varepsilon}\Big)}ds\int_{0}^s \exp{\Big(-\frac{\tilde{\nu_{\mathfrak{c}}}'(s,s')}{\varepsilon}\Big)}ds'\nonumber\\
   	   	&\qquad \qquad \times \int_{\mathbb{R}^3}\hat{k}(p,q)\frac{1}{\sqrt{J_{\mathfrak{c}}}}\Big|\Big\{Q_{\mathfrak{c}}(F_i^{\mathfrak{c}},\sqrt{J_{\mathfrak{c}}}h_R^{\varepsilon,\mathfrak{c}})+Q_{\mathfrak{c}}(\sqrt{J_{\mathfrak{c}}}h_R^{\varepsilon,\mathfrak{c}},F_i^{\mathfrak{c}})\Big\}(s',y_2,q)\Big|dq\nonumber\\
   	   	&\quad + \varepsilon^{k-1}\int_{0}^t \exp{\Big(-\frac{\tilde{\nu_{\mathfrak{c}}}(t,s)}{\varepsilon}\Big)}ds\int_{0}^s \exp{\Big(-\frac{\tilde{\nu_{\mathfrak{c}}}'(s,s')}{\varepsilon}\Big)}ds'\int_{\mathbb{R}^3}\hat{k}(p,q)  |\tilde{A}(s',y_2,q)|dq\nonumber\\
   	   	&=\sum_{j=1}^5\mathcal{J}_{2j}.
   	   \end{align*}
      
   	   By Lemma \ref{lem2.9}, there exists a positive constant $\nu_0$ which is independent of $\mathfrak{c}$, such that 
   	   \begin{align*}
   	   	\nu_{\mathfrak{c}}(p)\ge \nu_0, \quad p\in \mathbb{R}^3.
   	   \end{align*}
   	   For $\mathcal{J}_{21}$, one has from \eqref{6.37-10} that
   	   \begin{align*} 
   	   	|\varepsilon^{\frac{3}{2}}w_{\ell}(p)\mathcal{J}_{21}|
   	   	&\lesssim \frac{1}{\varepsilon}\int_{0}^t \exp{\Big(-\frac{\nu_0t}{\varepsilon}\Big)}ds\int_{\mathbb{R}^3}\hat{k}_w(p,q) |\varepsilon^{\frac{3}{2}}w_{\ell}(q)h_0(y_1-\hat{q}s,q)|dq\nonumber\\
   	   	&\lesssim \|\varepsilon^{\frac{3}{2}}h_0\|_{\infty,\ell}.
   	   \end{align*}
   	   Similarly, using Lemma \ref{lem7.3}, we get
   	    \begin{align*} 
   	   	|\varepsilon^{\frac{3}{2}}w_{\ell}(p)\mathcal{J}_{23}|
   	    &\lesssim \varepsilon^{k-\frac{1}{2}}\int_{0}^t \exp{\Big(-\frac{\nu_0(t-s)}{\varepsilon}\Big)}ds\int_{\mathbb{R}^3}\hat{k}_w(p,q)dq\nonumber\\
   	    &\qquad \qquad \times \int_{0}^s \exp{\Big(-\frac{\tilde{\nu_{\mathfrak{c}}}'(s,s')}{\varepsilon}\Big)}\frac{w_{\ell}(q)}{\sqrt{J_{\mathfrak{c}}}}\Big|Q_{\mathfrak{c}}(h_R^{\varepsilon,\mathfrak{c}}\sqrt{J_{\mathfrak{c}}},h_R^{\varepsilon,\mathfrak{c}}\sqrt{J_{\mathfrak{c}}})(s',y_2,q)\Big|ds'\nonumber\\
   	    &\lesssim  \varepsilon^{k-\frac{3}{2}}\sup_{0 \leq s \leq T}\|\varepsilon^{\frac{3}{2}}h_R^{\varepsilon,\mathfrak{c}}(s)\|^2_{\infty,\ell}
       	\end{align*}
   	    and
   	    \begin{align*} 
   	    	|\varepsilon^{\frac{3}{2}}w_{\ell}(p)\mathcal{J}_{24}|
   	        &\lesssim \varepsilon^{\frac{1}{2}}\sum_{i=1}^{2k-1} \varepsilon^{i-1}\int_{0}^t \exp{\Big(-\frac{\nu_0(t-s)}{\varepsilon}\Big)}ds\int_{\mathbb{R}^3}\hat{k}_w(p,q)dq\nonumber\\ 
   	        &\qquad \times \int_{0}^s \exp{\Big(-\frac{\tilde{\nu_{\mathfrak{c}}}'(s,s')}{\varepsilon}\Big)}\frac{w_{\ell}(q)}{\sqrt{J_{\mathfrak{c}}}}\Big|\Big\{Q_{\mathfrak{c}}(F_i^{\mathfrak{c}},\sqrt{J_{\mathfrak{c}}}h_R^{\varepsilon,\mathfrak{c}})+Q_{\mathfrak{c}}(\sqrt{J_{\mathfrak{c}}}h_R^{\varepsilon,\mathfrak{c}},F_i^{\mathfrak{c}})\Big\}(s',y_2,q)\Big|ds'\nonumber\\
   	        &\lesssim  \varepsilon \sup_{0 \leq s \leq T}\|\varepsilon^{\frac{3}{2}}h_R^{\varepsilon,\mathfrak{c}}(s)\|_{\infty,\ell}.
     	\end{align*}
        For $\mathcal{J}_{25}$, one has
        \begin{align*} 
        	|\varepsilon^{\frac{3}{2}}w_{\ell}(p)\mathcal{J}_{25}|
        	&\lesssim \varepsilon^{k+\frac{1}{2}}\sum_{\substack{i+j\ge 2k+1 \\ 2 \leq i, j \leq 2k-1}} \varepsilon^{i+j-1-2k}\int_{0}^t \exp{\Big(-\frac{\nu_0(t-s)}{\varepsilon}\Big)}ds\int_{\mathbb{R}^3}\hat{k}_w(p,q)dq\nonumber\\
        	&\qquad \times \int_{0}^s \exp{\Big(-\frac{\tilde{\nu_{\mathfrak{c}}}'(s,s')}{\varepsilon}\Big)} \frac{w_{\ell}(q)}{\sqrt{J_{\mathfrak{c}}}}|Q_{\mathfrak{c}}(F_i^{\mathfrak{c}},F_i^{\mathfrak{c}})(s',y_2,q)| ds'\nonumber\\
        	&\lesssim  \varepsilon^{k+\frac{5}{2}}.
        \end{align*}
    
       Now we focus on the estimate of $\mathcal{J}_{22}$. It holds that
       \begin{align*}
       		|\varepsilon^{\frac{3}{2}}w_{\ell}(p)\mathcal{J}_{22}|
       		&\lesssim \frac{1}{\varepsilon^2}\int_{0}^t \exp{\Big(-\frac{\tilde{\nu_{\mathfrak{c}}}(t,s)}{\varepsilon}\Big)}ds\int_{0}^s \exp{\Big(-\frac{\tilde{\nu_{\mathfrak{c}}}'(s,s')}{\varepsilon}\Big)}ds'\nonumber\\
       		&\qquad \times \int_{\mathbb{R}^3}\hat{k}_w(p,q)dq \int_{\mathbb{R}^3}\hat{k}_w(q,q')|\varepsilon^{\frac{3}{2}}w_{\ell}(q')h_R^{\varepsilon,\mathfrak{c}}(s',y_2,q')dq'.
       \end{align*}
        We divide the estimate into four cases. \\
        {\it{Case 1}}: $|p|\ge N$. Using \eqref{6.37-10}, one has
        \begin{align*}
        	|\varepsilon^{\frac{3}{2}}w_{\ell}(p)\mathcal{J}_{22}|&\lesssim \max{\Big\{\frac{1}{\mathfrak{c}},\frac{1}{1+|p|}\Big\}}\sup_{0 \leq s \leq T}\|\varepsilon^{\frac{3}{2}}h_R^{\varepsilon,\mathfrak{c}}(s)\|_{\infty,\ell}\nonumber\\
        	&\lesssim \max{\Big\{\frac{1}{\mathfrak{c}},\frac{1}{N}\Big\}}\sup_{0 \leq s \leq T}\|\varepsilon^{\frac{3}{2}}h_R^{\varepsilon,\mathfrak{c}}(s)\|_{\infty,\ell}.
        \end{align*}
        {\it{Case 2}}: $|p|\le N$, $|q|\ge 2N$ or $|q|\le 2N$, $|q'|\ge 3N$. Using \eqref{6.37-10} again, we have
        \begin{align*}
        	&\frac{1}{\varepsilon^2}\int_{0}^t \exp{\Big(-\frac{\tilde{\nu_{\mathfrak{c}}}(t,s)}{\varepsilon}\Big)}ds\int_{0}^s \exp{\Big(-\frac{\tilde{\nu_{\mathfrak{c}}}'(s,s')}{\varepsilon}\Big)}ds'\nonumber\\
        	&\qquad \times \Big\{\iint_{|p|\le N,|q|\ge 2N}+\iint_{|q|\le 2N,|q'|\ge 3N}\Big\}\nonumber\\
        	&\lesssim e^{-\frac{\delta_2}{4}N} \sup_{0 \leq s \leq T}\|\varepsilon^{\frac{3}{2}}h_R^{\varepsilon,\mathfrak{c}}(s)\|_{\infty,\ell}
        	\lesssim \frac{1}{N} \sup_{0 \leq s \leq T}\|\varepsilon^{\frac{3}{2}}h_R^{\varepsilon,\mathfrak{c}}(s)\|_{\infty,\ell}.
        \end{align*}
         {\it{Case 3}}: For $s-s'\le \kappa \varepsilon$ and $|p|\le N$, $|q|\le 2N$, $|q'|\le 3N$, one has
         \begin{align*}
         	&\frac{1}{\varepsilon^2}\int_{0}^t \exp{\Big(-\frac{\tilde{\nu_{\mathfrak{c}}}(t,s)}{\varepsilon}\Big)}ds\int_{s-\kappa \varepsilon}^s \exp{\Big(-\frac{\tilde{\nu_{\mathfrak{c}}}'(s,s')}{\varepsilon}\Big)}ds'\nonumber\\
         	&\qquad \times \int_{|q|\le 2N}\hat{k}_w(p,q)dq \int_{|q'|\le 3N}\hat{k}_2(q,q')|\varepsilon^{\frac{3}{2}}w_{\ell}(q')h_R^{\varepsilon,\mathfrak{c}}(s',y_2,q')|dq'\nonumber\\
         	&\lesssim \kappa \sup_{0 \leq s \leq T}\|\varepsilon^{\frac{3}{2}}h_R^{\varepsilon,\mathfrak{c}}(s)\|_{\infty,\ell}.
         \end{align*}
          {\it{Case 4}}: For $s-s'\ge \kappa \varepsilon$ and $|p|\le N$, $|q|\le 2N$, $|q'|\le 3N$, this is the last remaining case. Using \eqref{6.37-10}, one has
          \begin{align*}
          	&\int_{|q|\le 2N}\int_{|q'|\le 3N}\hat{k}_w(p,q)\hat{k}_w(q,q')|w_{\ell}(q')h_R^{\varepsilon,\mathfrak{c}}(s',y_2,q')|dqdq'\nonumber\\
          	&\le C_N \int_{|q|\le 2N}\int_{|q'|\le 3N}\hat{k}_w(p,q)\hat{k}_w(q,q')|f_R^{\varepsilon,\mathfrak{c}}(s',y_2,q')|dqdq'\nonumber\\
          	&\le C_N\Big(\int_{|q|\le 2N}\int_{|q'|\le 3N}\hat{k}^2_w(p,q)\hat{k}^2_w(q,q') dqdq'\Big)^{\frac{1}{2}}\nonumber\\
          	&\qquad \times \Big(\int_{|q|\le 2N}\int_{|q'|\le 3N}|f_R^{\varepsilon,\mathfrak{c}}(s',y_2,q')|^2 dqdq'\Big)^{\frac{1}{2}}\nonumber\\
          	&\le C_N \Big(\int_{\mathbb{R}^3}\int_{\mathbb{R}^3}|f_R^{\varepsilon,\mathfrak{c}}(s',y_2,q')|^2\cdot \varepsilon^{-3}\kappa^{-3} dy_2dq'\Big)^{\frac{1}{2}}\nonumber\\
          	&\le \frac{C_{N,\kappa}}{\varepsilon^{\frac{3}{2}}} \sup_{0 \leq s \leq T}\|f_R^{\varepsilon,\mathfrak{c}}(s)\|_2,
          \end{align*}
           where we have  made a change of variables $q\mapsto y_2$ with
           \begin{align*}
           	\Big|\frac{dy_2}{dq}\Big|=\frac{\mathfrak{c}^5}{(q^0)^5}(s-s')^3\ge \frac{\kappa^3\varepsilon^3}{3^5}.
           \end{align*}
          Here we take $1\le N\le \mathfrak{c}$. Thus we have
          \begin{align*}
          	&\frac{1}{\varepsilon^2}\int_{0}^t \exp{\Big(-\frac{\tilde{\nu_{\mathfrak{c}}}(t,s)}{\varepsilon}\Big)}ds\int_{0}^{s-\kappa \varepsilon} \exp{\Big(-\frac{\tilde{\nu_{\mathfrak{c}}}'(s,s')}{\varepsilon}\Big)}ds'\nonumber\\
          	&\qquad \times \int_{|q|\le 2N}\hat{k}_w(p,q)dq \int_{|q'|\le 3N}\hat{k}_2(q,q')|\varepsilon^{\frac{3}{2}}w_{\ell}(q')h_R^{\varepsilon,\mathfrak{c}}(s',y_2,q')|dq'\nonumber\\
          	&\le C_{N,\kappa}  \sup_{0 \leq s \leq T}\|f_R^{\varepsilon,\mathfrak{c}}(s)\|_2.
          \end{align*}
      
          Collecting all the four cases, we obtain
          \begin{align}\label{7.63-0}
          	|\varepsilon^{\frac{3}{2}}w_{\ell}(p)\mathcal{J}_{22}|&\le C\Big(\kappa+\frac{1}{N}\Big)\sup_{0 \leq s \leq T}\|\varepsilon^{\frac{3}{2}}h_R^{\varepsilon,\mathfrak{c}}(s)\|_{\infty,\ell}+C_{N,\kappa}  \sup_{0 \leq s \leq T}\|f_R^{\varepsilon,\mathfrak{c}}(s)\|_2.
          \end{align}
         Therefore, combining \eqref{7.50-0} and \eqref{7.63-0}, one obtains
         \begin{align}\label{7.64-0}
         	\sup_{0 \leq s \leq T}\|\varepsilon^{\frac{3}{2}}h_R^{\varepsilon,\mathfrak{c}}(s)\|_{\infty,\ell}
         	&\leq  C\Big(\varepsilon+\kappa+\frac{1}{N}\Big)\sup_{0 \leq s \leq T}\|\varepsilon^{\frac{3}{2}}h_R^{\varepsilon,\mathfrak{c}}(s)\|_{\infty,\ell}+C\|\varepsilon^{\frac{3}{2}}h_0\|_{\infty,\ell}\nonumber\\
         	& \qquad +C \varepsilon^{k-\frac{3}{2}} \sup_{0 \leq s \leq T}\|\varepsilon^{\frac{3}{2}}h_R^{\varepsilon,\mathfrak{c}}(s)\|_{\infty,\ell}^2 +C \varepsilon^{k+\frac{5}{2}}+C_{N,\kappa}  \sup_{0 \leq s \leq T}\|f_R^{\varepsilon,\mathfrak{c}}(s)\|_2.
         \end{align}
         Choosing $N$ suitably large and $\kappa$, $\varepsilon$ suitably small, one gets from \eqref{7.64-0} that 
         \begin{align*}
         	\sup_{0 \leq s \leq T}\|\varepsilon^{\frac{3}{2}}h_R^{\varepsilon,\mathfrak{c}}(s)\|_{\infty,\ell} \leq C\|\varepsilon^{\frac{3}{2}}h_0\|_{\infty,\ell}+C \sup_{0 \leq s \leq T}\|f_R^{\varepsilon,\mathfrak{c}}(s)\|_2+C \varepsilon^{k+\frac{5}{2}}.
         \end{align*}
         Therefore the proof of Lemma \ref{lem7.2} is completed.
   \end{proof}
  
  \begin{proof}[\textbf{Proof of Theorem \ref{thm1.1-0}}]
With Lemmas \ref{lem7.1}--\ref{lem7.2} in hand, the rest proof is the same as \cite{Guo2,Speck}. We omit the details here for brevity. Therefore the proof of Theorem \ref{thm1.1-0} is completed.
  \end{proof} 
    
    Using Theorem \ref{thm1.1-0}, we can  prove Theorem \ref{thm1.3-0} as follows.
    \begin{proof}[\textbf{Proof of Theorem \ref{thm1.3-0}}]
        Recall $\bar{c}_1$ and $\bar{c}_2$ in \eqref{3.90-10}. 
    	Using \eqref{1.53-00}, for any $(t,x,p)\in [0,T]\times \mathbb{R}^3\times \mathbb{R}^3$, one has
    	\begin{align}\label{6.68-10}
    		|F^{\varepsilon,\mathfrak{c}}(t,x,p)-\mathbf{M}_{\mathfrak{c}}(t,x,p)|\lesssim \varepsilon\sqrt{J_{\mathfrak{c}}(p)} \lesssim \varepsilon e^{-\frac{|p|}{2T_M}}.
    	\end{align}
       A direct calculation shows that
       \begin{align}\label{6.70-10}
       	&\mu(t,x,p)-\mathbf{M}_{\mathfrak{c}}(t,x,p)\nonumber\\
       	&=\frac{\rho}{(2\pi \theta)^{\frac{3}{2}}}\exp{\Big\{-\frac{|p-\mathfrak{u}|^2}{2\theta}\Big\}}-\frac{n_0\gamma}{4\pi \mathfrak{c}^3K_2(\gamma)}\exp\Big\{\frac{u^{\mu}p_{\mu}}{T_0}\Big\}\nonumber\\
       	&=\frac{\rho}{(2\pi \theta)^{\frac{3}{2}}}\exp{\Big\{-\frac{|p-\mathfrak{u}|^2}{2\theta}\Big\}}-\frac{n_0}{(2\pi T_0)^{\frac{3}{2}}}\exp\Big\{\frac{\mathfrak{c}^2+u^{\mu}p_{\mu}}{T_0}\Big\}(1+O(\gamma^{-1}))\nonumber\\
       	&=O(\gamma^{-1})\frac{n_0}{(2\pi T_0)^{\frac{3}{2}}}\exp\Big\{\frac{\mathfrak{c}^2+u^{\mu}p_{\mu}}{T_0}\Big\}+\Big(\frac{\rho}{(2\pi \theta)^{\frac{3}{2}}}-\frac{n_0}{(2\pi T_0)^{\frac{3}{2}}}\Big)\exp{\Big\{-\frac{|p-\mathfrak{u}|^2}{2\theta}\Big\}}\nonumber\\
       	&\qquad +\frac{n_0}{(2\pi T_0)^{\frac{3}{2}}}\Big(\exp{\Big\{-\frac{|p-\mathfrak{u}|^2}{2\theta}\Big\}}-\exp\Big\{\frac{\mathfrak{c}^2+u^{\mu}p_{\mu}}{T_0}\Big\}\Big)\nonumber\\
       	&:=\mathcal{A}_1+\mathcal{A}_2+\mathcal{A}_3.
       \end{align}
      It follows from Proposition \eqref{thm-retoce} that
      \begin{align*}
      	|\mathcal{A}_1|\lesssim \frac{1}{\mathfrak{c}^2}e^{-2\bar{c}_1 |p|},\quad |\mathcal{A}_2|\lesssim \frac{1}{\mathfrak{c}^2}e^{-\bar{c}_2 |p|}.
      \end{align*}
      For $\mathcal{A}_3$, if $|p|\ge \mathfrak{c}^{\frac{1}{8}}$, one has
      \begin{align*} 
      	|\mathcal{A}_3|&\lesssim \exp{\Big\{-\frac{|p|^2}{4\theta}\Big\}}+\exp{\Big\{-\frac{|p|}{2T_0}\Big\}}\nonumber\\
      	&\lesssim \exp{\Big\{-\frac{\mathfrak{c}^{\frac{1}{4}}}{8\theta}\Big\}}\exp{\Big\{-\frac{|p|^2}{8\theta}\Big\}}+\exp{\Big\{-\frac{\mathfrak{c}^{\frac{1}{8}}}{4T_0}\Big\}}\exp{\Big\{-\frac{|p|}{4T_0}\Big\}}\nonumber\\
      	&\lesssim \frac{1}{\mathfrak{c}^2}\big(e^{-\frac{\bar{c}_2}{2}|p|}+e^{-\bar{c}_1|p|}\big).
      \end{align*}
      If $|p|\le \mathfrak{c}^{\frac{1}{8}}$, it follows from \eqref{4.68-00}-\eqref{4.69-00} that
      \begin{align}\label{6.73-10}
      		|\mathcal{A}_3|&\le\frac{n_0}{(2\pi T_0)^{\frac{3}{2}}}\exp{\Big\{-\frac{|p-\mathfrak{u}|^2}{2\theta}\Big\}}\Big|1-\exp\Big\{\frac{|p-\mathfrak{u}|^2}{2\theta}+\frac{\mathfrak{c}^2+u^{\mu}p_{\mu}}{T_0}\Big\}\Big|\lesssim \mathfrak{c}^{-\frac{3}{2}}e^{-\bar{c}_2|p|}.
      \end{align}
      Combining \eqref{6.70-10}--\eqref{6.73-10}, one has
      \begin{align}\label{6.76-20}
      		|\mu(t,x,p)-\mathbf{M}_{\mathfrak{c}}(t,x,p)|\lesssim \mathfrak{c}^{-\frac{3}{2}}(e^{-\frac{\bar{c}_2}{2}|p|}+e^{-\bar{c}_1|p|}).
      \end{align}
     Using \eqref{6.68-10}, \eqref{6.76-20} and taking 
     \begin{align*} 
     	\delta_0:=\min{\Big(\frac{1}{2T_M},\, \bar{c}_1,\,\frac{\bar{c}_2}{2} \Big)}>0,
     \end{align*}
     one has
     \begin{align*} 
     	|F^{\varepsilon,\mathfrak{c}}(t)-\mu(t)|\lesssim \varepsilon e^{-\frac{|p|}{2T_M}}+ \mathfrak{c}^{-\frac{3}{2}}(e^{-\frac{\bar{c}_2}{2}|p|}+e^{-\bar{c}_1|p|})\lesssim (\varepsilon+\mathfrak{c}^{-\frac{3}{2}}) e^{-\delta_0 |p|},
     \end{align*}
     which implies that
     \begin{align*} 
     	\sup_{0\le t\le T}\Big\|\big(F^{\varepsilon,\mathfrak{c}}-\mu\big)(t)e^{\delta_0 |p|} \Big\|_{\infty} \lesssim \varepsilon+ \mathfrak{c}^{-\frac{3}{2}}.
     \end{align*}
    Therefore the proof of Theorem \ref{thm1.3-0} is completed.
    \end{proof}
\section{Appendix: Derivation of the orthonormal basis of $\mathcal{N}_{\mathfrak{c}}$}
In this part, we derive the orthonormal basis of $\mathcal{N}_{\mathfrak{c}}$. One needs to use \eqref{1.16-01}--\eqref{1.17-01} and Lemma \ref{lem4.11-0} frequently. Suppose that 
\begin{align*} 
	\chi^{\mathfrak{c}}_0=\mathfrak{a}_0\mathbf{\sqrt{M_{\mathfrak{c}}}},\quad  \chi^{\mathfrak{c}}_j=\frac{p_j-\mathfrak{a}_j}{\mathfrak{b}_j}\mathbf{\sqrt{M_{\mathfrak{c}}}}\ (j=1,2,3),\quad \chi^{\mathfrak{c}}_4=\frac{p^0/\mathfrak{c}+\sum_{i=1}^3\lambda_ip_i+\mathfrak{e}}{\zeta}\mathbf{\sqrt{M_{\mathfrak{c}}}}
\end{align*}
form an orthonormal basis of $\mathcal{N}_{\mathfrak{c}}$. Using $\langle \chi^{\mathfrak{c}}_0,\chi^{\mathfrak{c}}_0\rangle=1$, one has $\mathfrak{a}_0=\Big(\int_{\mathbb{R}^3}\mathbf{M}_{\mathfrak{c}}dp\Big)^{-\frac{1}{2}}=\frac{1}{\sqrt{I^0}}$. To compute $\mathfrak{a}_j$, since $\langle \chi^{\mathfrak{c}}_0,\chi^{\mathfrak{c}}_j\rangle=0$, we have
\begin{align*}
	0=\int_{\mathbb{R}^3}(p_j-\mathfrak{a}_j)\mathbf{M}_{\mathfrak{c}}dp=T^{0j}-\mathfrak{a}_jI^0,
\end{align*}
which yields that $\mathfrak{a}_j=\frac{T^{0j}}{I^0}$. For $\mathfrak{b}_j$, using $\langle \chi^{\mathfrak{c}}_j,\chi^{\mathfrak{c}}_j\rangle=1$, one has
\begin{align*}
	\mathfrak{b}_j^2&=\int_{\mathbb{R}^3}(p_j-\mathfrak{a}_j)^2\mathbf{M}_{\mathfrak{c}}dp=\int_{\mathbb{R}^3}(p^2_j+\mathfrak{a}_j^2-2\mathfrak{a}_jp_j)\mathbf{M}_{\mathfrak{c}}dp\\
	&=T^{0jj}+\mathfrak{a}_j^2I^0-2\mathfrak{a}_jT^{0j}=T^{0jj}-\frac{(T^{0j})^2}{I^0},
\end{align*}
which yields that $\mathfrak{b}_j=\sqrt{T^{0jj}-\frac{(T^{0j})^2}{I^0}}$, $j=1,2,3$.

To determine the coefficients $\lambda_i$, $i=1,2,3$, due to $\langle \chi^{\mathfrak{c}}_4,\chi^{\mathfrak{c}}_0\rangle=\langle \chi^{\mathfrak{c}}_4,\chi^{\mathfrak{c}}_j\rangle=0$, we have
\begin{align*}
	\int_{\mathbb{R}^3}(p^0/\mathfrak{c}+\sum_{i=1}^3\lambda_ip_i+\mathfrak{e})\mathbf{M}_{\mathfrak{c}}dp&=0,\\
	\int_{\mathbb{R}^3}(p^0/\mathfrak{c}+\sum_{i=1}^3\lambda_ip_i+\mathfrak{e})(p_j-\mathfrak{a}_j)\mathbf{M}_{\mathfrak{c}}dp&=0,\ j=1,2,3.
\end{align*}
That is
\begin{align*}
     \frac{T^{00}}{\mathfrak{c}}+\sum_{i=1}^3\lambda_iT^{0i}+\mathfrak{e} I^0&=0,\\
	 \frac{T^{00j}}{\mathfrak{c}}-\frac{\mathfrak{a}_j}{\mathfrak{c}}T^{00}+\sum_{i=1}^3\lambda_i(T^{0ij}-\mathfrak{a}_jT^{0i})+\mathfrak{e}(T^{0j}-\mathfrak{a}_jI^0)&=0,\ j=1,2,3.
\end{align*}
One can rewrite the above linear system as 
{\small
\begin{equation}\label{5.1-0}
	\left( \setlength\arraycolsep{3.0pt}
	    \begin{array}{cccc}
		\vspace{1.2ex}
		T^{01} & T^{02} & T^{03} & I^0 \\
		\vspace{1.2ex}
		T^{011}-\mathfrak{a}_1T^{01}& T^{021}-\mathfrak{a}_1T^{02} & T^{031}-\mathfrak{a}_1T^{03}& T^{01}-\mathfrak{a}_1I^0\\
		\vspace{1.2ex}
		T^{012}-\mathfrak{a}_2T^{01}& T^{022}-\mathfrak{a}_2T^{02} & T^{032}-\mathfrak{a}_2T^{03}& T^{02}-\mathfrak{a}_2I^0\\
		\vspace{1.2ex}
		T^{013}-\mathfrak{a}_3T^{01}& T^{023}-\mathfrak{a}_3T^{02} & T^{033}-\mathfrak{a}_3T^{03}& T^{03}-\mathfrak{a}_3I^0
	\end{array}\right)
    \begin{pmatrix}
	\vspace{1.2ex}
	\lambda_1\\  \vspace{1.2ex}  \lambda_2\\  \vspace{1.2ex} \lambda_3\\  \vspace{1.2ex}  \mathfrak{e}
     \end{pmatrix}
    =\begin{pmatrix}
	\vspace{1.2ex}
	-\frac{T^{00}}{\mathfrak{c}}\\
	\vspace{1.2ex}
	 \frac{\mathfrak{a}_1T^{00}}{\mathfrak{c}}-\frac{T^{001}}{\mathfrak{c}}\\ 
	 \vspace{1.2ex}
	\frac{\mathfrak{a}_2T^{00}}{\mathfrak{c}}-\frac{T^{002}}{\mathfrak{c}}\\
	\vspace{1.2ex}
	\frac{\mathfrak{a}_3T^{00}}{\mathfrak{c}}-\frac{T^{003}}{\mathfrak{c}}
    \end{pmatrix}.
\end{equation}
}
Denote 
\begin{align*}
	\mathfrak{a}:=\frac{n_0u^0}{\mathfrak{c}}\frac{K_3(\gamma)}{K_2(\gamma)},\quad \mathfrak{b}:=\frac{n_0u^0}{\mathfrak{c}\gamma K_2(\gamma)}(6K_3(\gamma)+\gamma K_2(\gamma)).
\end{align*}
By a tedious calculation, one can transform \eqref{5.1-0} into the following system
{\small
\begin{align}\label{5.2-0}
	&\left( \setlength\arraycolsep{0.5pt}
	\begin{array}{cccc} 
		\vspace{1.2ex}
		0 & 0 & 0 & \mathfrak{a}\frac{K_2(\gamma)}{K_3(\gamma)}-\mathfrak{a}\frac{K_2(\gamma)}{K_3(\gamma)}(\frac{K_3(\gamma)}{K_2(\gamma)}-\frac{\mathfrak{b}}{\mathfrak{a}})\frac{|u|^2}{T_0}\\
		\vspace{1.2ex}
		\mathfrak{a}T_0& 0 & 0& \mathfrak{a}\frac{K_2(\gamma)}{K_3(\gamma)}(\frac{K_3(\gamma)}{K_2(\gamma)}-\frac{\mathfrak{b}}{\mathfrak{a}})u_1\\
		\vspace{1.2ex}
		0& \mathfrak{a}T_0 & 0& \mathfrak{a}\frac{K_2(\gamma)}{K_3(\gamma)}(\frac{K_3(\gamma)}{K_2(\gamma)}-\frac{\mathfrak{b}}{\mathfrak{a}})u_2\\
		\vspace{1.2ex}
		0& 0 & \mathfrak{a}T_0& \mathfrak{a}\frac{K_2(\gamma)}{K_3(\gamma)}(\frac{K_3(\gamma)}{K_2(\gamma)}-\frac{\mathfrak{b}}{\mathfrak{a}})u_3\\
	\end{array}
    \right)
	\begin{pmatrix}
		\vspace{1.5ex}
		\lambda_1\\ \vspace{1.5ex}  \lambda_2\\  \vspace{1.5ex} \lambda_3\\ \vspace{1.5ex}  \mathfrak{e}
	\end{pmatrix}
    =\begin{pmatrix} \vspace{1.2ex}
		\frac{n_0}{\gamma}-\frac{\mathfrak{a}u^0}{\mathfrak{c}}-\frac{n_0}{\gamma}(\frac{K_3(\gamma)}{K_2(\gamma)}-\frac{\mathfrak{b}}{\mathfrak{a}})\frac{|u|^2}{T_0}\\ \vspace{1.2ex} \frac{n_0}{\gamma}(\frac{K_3(\gamma)}{K_2(\gamma)}-\frac{\mathfrak{b}}{\mathfrak{a}})u_1\\ \vspace{1.2ex}
		\frac{n_0}{\gamma}(\frac{K_3(\gamma)}{K_2(\gamma)}-\frac{\mathfrak{b}}{\mathfrak{a}})u_2\\ \vspace{1.2ex}
		\frac{n_0}{\gamma}(\frac{K_3(\gamma)}{K_2(\gamma)}-\frac{\mathfrak{b}}{\mathfrak{a}})u_3
	\end{pmatrix}.
\end{align}
}
   Observing \eqref{5.7-20}, one has $\frac{K_3(\gamma)}{K_2(\gamma)}-\frac{\mathfrak{b}}{\mathfrak{a}}<0$, which implies that \eqref{5.2-0} has a unique solution. More precisely, we can write down it explicitly   
   {\small
   \begin{align*}
   	&\begin{pmatrix}\vspace{1.2ex}
   		\lambda_1\\ \vspace{1.2ex}   \lambda_2\\ \vspace{1.2ex}   \lambda_3\\ \vspace{1.2ex}  \mathfrak{e}
   	\end{pmatrix}
   = \frac{1}{\frac{u^0}{\mathfrak{c}}-\Big(\frac{K_3(\gamma)}{K_2(\gamma)}-\frac{6}{\gamma}-\frac{K_2(\gamma)}{K_3(\gamma)}\Big)\frac{u^0|u|^2}{\mathfrak{c}T_0}}
   \begin{pmatrix}\vspace{1.2ex}
   	\Big(\frac{K_3(\gamma)}{K_2(\gamma)}-\frac{6}{\gamma}-\frac{K_2(\gamma)}{K_3(\gamma)}\Big)\frac{(u^0)^2}{\mathfrak{c}^2 T_0}u_1\\  \vspace{1.2ex}  \Big(\frac{K_3(\gamma)}{K_2(\gamma)}-\frac{6}{\gamma}-\frac{K_2(\gamma)}{K_3(\gamma)}\Big)\frac{(u^0)^2}{\mathfrak{c}^2 T_0}u_2\\  \vspace{1.2ex} \Big(\frac{K_3(\gamma)}{K_2(\gamma)}-\frac{6}{\gamma}-\frac{K_2(\gamma)}{K_3(\gamma)}\Big)\frac{(u^0)^2}{\mathfrak{c}^2 T_0}u_3\\  \vspace{1.2ex} 
   	\frac{1}{\gamma}-\frac{(u^0)^2}{\gamma T_0}\frac{K_3(\gamma)}{K_2(\gamma)}-\Big(\frac{K_3(\gamma)}{K_2(\gamma)}-\frac{6}{\gamma}-\frac{K_2(\gamma)}{K_3(\gamma)}\Big)\frac{|u|^2}{\gamma T_0}
   \end{pmatrix}.
   \end{align*}
    }
   For $\zeta$, it follow from $\langle \chi^{\mathfrak{c}}_4,\chi^{\mathfrak{c}}_4\rangle=1$ that
   \begin{align*}
   	\zeta^2&=\int_{\mathbb{R}^3}(p^0/\mathfrak{c}+\sum_{i=1}^3\lambda_ip_i+\mathfrak{e})^2\mathbf{M}_{\mathfrak{c}}dp\nonumber\\
   	&=\int_{\mathbb{R}^3}(\frac{(p^0)^2}{\mathfrak{c}^2}+\mathfrak{e}^2+\sum_{i,j=1}^3\lambda_i\lambda_jp_ip_j+2\frac{\mathfrak{e}}{\mathfrak{c}}p^0+2\sum_{i=1}^3\lambda_i\mathfrak{e} p_i+2\sum_{i=1}^3\frac{\lambda_i}{\mathfrak{c}}p^0p_i)\mathbf{M}_{\mathfrak{c}}dp\\
   	&=\frac{T^{000}}{\mathfrak{c}^2}+\mathfrak{e}^2I^0+\sum_{i,j=1}^3\lambda_i\lambda_jT^{0ij}+2\frac{\mathfrak{e}}{\mathfrak{c}}T^{00}+2\sum_{i=1}^3\lambda_i\mathfrak{e} T^{0i}+2\sum_{i=1}^3\frac{\lambda_i}{\mathfrak{c}}T^{00i},
   \end{align*}
   which yields that
   \begin{align*}
   	\zeta=\sqrt{\frac{T^{000}}{\mathfrak{c}^2}+\mathfrak{e}^2I^0+\sum_{i,j=1}^3\lambda_i\lambda_jT^{0ij}+2\frac{\mathfrak{e}}{\mathfrak{c}}T^{00}+2\sum_{i=1}^3\lambda_i\mathfrak{e} T^{0i}+2\sum_{i=1}^3\frac{\lambda_i}{\mathfrak{c}}T^{00i}}.
   \end{align*}
    Consequently, we obtain the desired orthonormal basis of $\mathcal{N}_{\mathfrak{c}}$.
    
    
    \medskip
    

     \medskip
    
\noindent{\bf Acknowledgments.} 
Yong Wang's research is partially supported by National Key R\&D Program of China No. 2021YFA1000800, National Natural Science Foundation of China No. 12022114, 12288201,  CAS Project for Young Scientists in Basic Research, Grant No. YSBR-031, and Youth Innovation Promotion Association of the Chinese Academy of Science No. 2019002. Changguo Xiao's research is partially supported by National Natural Science Foundation of China No. 12361045 and Guangxi Natural Science Foundation (Grant No. 2023GXNSFAA026066). 

	\medskip

\noindent{\bf Conflict of Interest:} The authors declare that they have no conflict of interest.

\end{document}